\documentclass[10pt, a4]{article}

\DeclareMathAlphabet\mathbfcal{OMS}{cmsy}{b}{n}
\usepackage[dvipsnames]{xcolor}
\usepackage{hyperref}
\usepackage{fourier}
\usepackage[tmargin=1.0in,bmargin=1.0in, lmargin=1.0in, rmargin=1.0in]{geometry}
\usepackage{lipsum}
\usepackage{amsthm}
\usepackage{amsfonts}
\usepackage{graphicx}
\usepackage{bm}
\usepackage{extarrows}
\usepackage{stmaryrd}
\usepackage{amssymb}
\usepackage{tikz-cd}
\usepackage{relsize}%
\usepackage{cleveref}
\usepackage{multirow}
\usepackage[normalem]{ulem}
\usepackage{textcomp}

\newcommand{\lbm}{lattice Boltzmann~}
\newcommand{\fd}{Finite Difference~}
\newcommand{\velocitynumber}{q}
\newcommand{\consmomentsnumber}{N}
\newcommand{\spacestep}{\Delta x}
\newcommand{\timestep}{\Delta t}
\newcommand{\strong}[1]{\textbf{#1}}
\newcommand{\spatialdimensionality}{d}
\newcommand{\latticevelocity}{\lambda}
\newcommand{\reals}{\mathbb{R}}
\newcommand{\naturals}{\mathbb{N}}
\newcommand{\relatives}{\mathbb{Z}}

\newcommand{\definitionequality}{:=}
\newcommand{\lattice}{\mathfrak{L}}
\newcommand{\timelattice}{\mathfrak{Z}}
\newcommand{\velocityletter}{e}
\newcommand{\normalizedvelocityletter}{c}
\newcommand{\populationindex}{j}
\newcommand{\matricial}[1]{\bm{#1}}
\newcommand{\vectorial}[1]{\bm{#1}}
\newcommand{\integerinterval}[2]{[#1\,..\,#2]}
\newcommand{\distributionletter}{f}
\newcommand{\momentsmatrix}{\matricial{M}}
\newcommand{\lineargroup}[2]{\text{GL}_{#1}(#2)}
\newcommand{\transpose}[1]{#1^{\intercal}}
\newcommand{\momentletter}{m}
\newcommand{\identity}{I}
\newcommand{\relaxationmatrix}{\matricial{S}}
\newcommand{\matrixspace}[2]{\mathcal{M}_{#1}(#2)}
\newcommand{\diagmatrix}{\text{diag}}
\newcommand{\relaxparletter}{s}
\newcommand{\indicemoments}{i}
\newcommand{\indiceconserved}{\indicemoments}
\newcommand{\atequilibrium}{\text{eq}}
\newcommand{\timevariable}{t}
\newcommand{\spacevariable}{x}
\newcommand{\collided}{\star}
\newcommand{\shiftoperator}[1]{\mathsf{t}_{#1}}
\newcommand{\genericfunction}{f}
\newcommand{\isomorphic}{\cong}
\newcommand{\groupunitsshiftoperators}{\mathsf{T}}
\newcommand{\setfinitedifferenceoperators}{\mathsf{D}}
\newcommand{\genericunit}{\mathsf{t}}
\newcommand{\coeffgenericfinitedifference}{\alpha}
\newcommand{\sumoperators}{+}
\newcommand{\prodoperators}{\circ}
\newcommand{\basicshift}{\mathsf{x}}
\newcommand{\canonicalbasisvector}{\vectorial{e}}
\newcommand{\polynomialring}[2]{#1 [#2]}
\newcommand{\streammoments}{\matricial{\mathsf{T}}}
\newcommand{\schememoments}{\matricial{\mathsf{A}}}
\newcommand{\schemeequil}{\matricial{\mathsf{B}}}
\newcommand{\timeshiftoperator}{\mathsf{z}}
\newcommand{\tensorProductRing}[1]{\otimes_{#1}}
\newcommand{\determinant}{\text{det}}
\newcommand{\adjugate}{\text{adj}}
\newcommand{\indicetimeshift}{k}
\newcommand{\indicepolynomials}{\indicetimeshift}
\newcommand{\coeffcharact}{\mathsf{c}}

\newcommand{\polynomialunknown}{X}
\newcommand{\genericmatrix}{\matricial{C}}
\newcommand{\genericmatrixtwo}{\matricial{D}}
\newcommand{\genericmatrixthree}{\matricial{E}}
\newcommand{\genericdiscretematrix}{\matricial{\mathsf{C}}}
\newcommand{\discretetimespaceoperators}{\polynomialring{\reals}{\timeshiftoperator} \tensorProductRing{\reals} \setfinitedifferenceoperators}
\newcommand{\genericcommutativering}{\mathfrak{R}}
\newcommand{\indiceslines}{i}
\newcommand{\cutmatrixsquare}[2]{#1_{#2}}
\newcommand{\setoflines}{I}
\newcommand{\schememomentsother}{\schememoments^{\diamond}}
\newcommand{\vectorequilibrium}{\vectorial{\mathsf{b}}}
\newcommand{\indicescolumns}{j}
\newcommand{\bigO}[1]{O(#1)}
\newcommand{\timespacedifferentialoperators}{\mathcal{D}}
\newcommand{\taylorseries}{\mathcal{S}}
\newcommand{\formalpolynomialring}[2]{#1 \llbracket #2 \rrbracket}
\newcommand{\generictaylorseries}{\delta}
\newcommand{\perturbationorderindices}{r}
\newcommand{\termatorder}[2]{#1^{(#2)}}
\newcommand{\termatorderparenthesis}[2]{(#1)^{(#2)}}
\newcommand{\asymtoticequivalence}{\asymp}
\newcommand{\generictimespacediscreteop}{\mathsf{d}}
\newcommand{\duboisoperatormatrix}{\mathbfcal{G}} 
\newcommand{\duboisoperatormatrixentry}{\mathcal{G}}
\newcommand{\multiindice}{\vectorial{\nu}}
\newcommand{\multiindicemodule}[1]{|#1|}
\newcommand{\multiindicederivative}[1]{\partial^{#1}}
\newcommand{\straightderivative}[2]{\frac{\text{d}#1}{\text{d}#2}}
\newcommand{\gammaDubois}[2]{\gamma_{#1, #2}}
\newcommand{\gammaDuboisScalar}[1]{\gamma_{#1}}
\newcommand{\asymptoticschememoments}{\mathbfcal{A}}
\newcommand{\asymptoticschemeequil}{\mathbfcal{B}}
\newcommand{\asymptoticstreammoments}{\mathbfcal{T}}
\newcommand{\asymptotictimeshiftoperator}{\zeta}
\newcommand{\asymptoticschememomentsother}{\asymptoticschememoments^{\diamond}}
\newcommand{\productrelaxation}{\Pi}
\newcommand{\differentialMatrixOperator}[3]{\text{D}_{#1} \left ( #2 \right )(#3)}
\newcommand{\secondDifferentialMatrixOperator}[4]{\text{D}_{#1 #1} \left ( #2 \right )(#3)(#4)}
\newcommand{\trace}{\text{tr}}
\newcommand{\genericcontinuousmatrix}{\mathbfcal{C}}
\newcommand{\ordermaxwelliteration}[2]{#1^{[#2]}}
\newcommand{\indicesmaxwelliteration}{k}
\newcommand{\conj}[1]{\overline{#1}}

\newcommand{\diffusivityScaling}{\mu}
\newcommand{\indiceMomentsTransport}{\Omega}

\newcommand{\scheme}[2]{$\text{D}_{#1}\text{Q}_{#2}$}

\providecommand{\keywords}[1]{\textbf{\textit{Keywords: }} #1}
\providecommand{\MSC}[1]{\newline \textbf{\textit{2000 MSC: }} #1}

\theoremstyle{plain}
\newtheorem{theorem}{Theorem}[section]

\newtheorem{lemma}[theorem]{Lemma}
\newtheorem{proposition}[theorem]{Proposition}

\theoremstyle{definition}
\newtheorem{definition}[theorem]{Definition}
\newtheorem{assumptions}[theorem]{Assumptions}
\newtheorem{remark}[theorem]{Remark}
\newtheorem{example}[theorem]{Example}

\begin{document}


\title{Truncation errors and modified equations for the lattice Boltzmann method \emph{via} the corresponding Finite Difference schemes}

\author{Thomas Bellotti \\ \footnotesize{(\url{thomas.bellotti@polytechnique.edu})} \\ CMAP, CNRS, \'Ecole polytechnique, Institut Polytechnique de Paris\\
91120 Palaiseau, France.}


    \maketitle

    \begin{abstract}
        Lattice Boltzmann schemes are efficient numerical methods to solve a broad range of problems under the form of conservation laws. However, they suffer from a chronic lack of clear theoretical foundations. In particular, the consistency analysis and the derivation of the modified equations are still open issues. This has prevented, until today, to have an analogous of the Lax equivalence theorem for Lattice Boltzmann schemes.
        We propose a rigorous consistency study and the derivation of the modified equations for any \lbm scheme under acoustic and diffusive scalings. 
        This is done by passing from a kinetic (lattice Boltzmann) to a macroscopic (Finite Difference) point of view at a fully discrete level in order to eliminate the non-conserved moments relaxing away from the equilibrium.
        We rewrite the \lbm scheme as a multi-step \fd scheme on the conserved variables, as introduced in our previous contribution. We then perform the usual analyses for \fd by exploiting its precise characterization using  matrices of \fd operators.
        Though we present the derivation of the modified equations until second-order under acoustic scaling, we provide all the elements to extend it to higher orders, since the kinetic-macroscopic connection is conducted at the fully discrete level.
        Finally, we show that our strategy yields, in a more rigorous setting, the same results as previous works in the literature. 
    \end{abstract}

    \MSC{65M75, 65M06, 65M12, 65M15}

\section{Introduction}

Lattice Boltzmann methods form a vast category of numerical schemes to address the approximation of the solution of Partial Differential Equations (PDEs) under the form of conservation laws, called macroscopic equations or target PDEs.
These numerical schemes act in a kinetic fashion by employing a certain number $\velocitynumber \in \naturals^{\star}$ of discrete velocities, larger than the number $\consmomentsnumber \in \naturals^{\star}$ of macroscopic equations to be solved.
The scheme proceeds \emph{via} a kinetic-like algorithm made up of two distinct steps.
The first one is a local non-linear collision phase on each site of the mesh, followed by a lattice-constrained transport which is inherently linear.
The local nature of the collision phase allows for massive parallelization of the method and the fact that the ``particles'' are constrained to dwell on the lattice allows to implement the stream phase as a pointer shift in memory.
This results in a very efficient numerical method capable of reaching problems of important size in terms of computational and memory cost.
The historical seminal papers from the end of the eighties are \cite{mcnamara1988use} and \cite{higuera1989boltzmann}, while for a general modern presentation of the \lbm schemes and their extremely broad fields of application, including hyperbolic systems of conservation laws, the quasi-incompressible Navier-Stokes equations, multi-phase systems and porous media, the interested reader can consult \cite{succi2001lattice}, \cite{guo2013lattice} and \cite{kruger2017lattice}.
The presentation of this plethora of interesting applications is however beyond the scope of our contribution.

To our understanding, the highest price to pay for this highly efficient implementation of the method is the lack of pure theoretical understanding on why the overall procedure works well at approximating the solution of the target PDEs.
This is essentially due to the fact that -- the standpoint of the \lbm schemes being kinetic \cite{simonis2020relaxation} -- the number of discrete velocities is larger than the number of macroscopic equations.
Therefore, the formal analyses for \lbm schemes available in the literature try to bridge the gap between a kinetic and a macroscopic point of view relying essentially on the quasi-equilibrium of the non-conserved variables. 
In particular, as far as the consistency with the macroscopic equations and the modified equations are concerned in the limit of small discretization parameters, two main approaches are at our disposal.
The first one is based on the Chapman-Enskog expansion \cite{chapman1990mathematical, huang1987} from statistical mechanics, shaped to the context of \lbm schemes, see for example \cite{chen1998lattice, qian2000higher, lallemand2000theory}.
The second approach features the so-called equivalent equations introduced by Dubois \cite{dubois2008equivalent,dubois2019nonlinear}, consisting in performing a Taylor expansion of the scheme both for the conserved and non-conserved moments and progressively re-inject the developments order-by-order. 
This approach has proved to yield information in accordance with the numerical simulations, see \cite{dubois2009towards, dubois2011quartic, bellotti2021high}.
Despite their proved empirical reliability and the fact that they yield the same results at the dominant orders (see \cite{dubois2019hal} for instance) these two strategies are both formal, especially for the computation of the truncation errors. Indeed, the Chapman-Enskog expansion relies on the introduction of two time variables with different scalings which are not present in the discrete \lbm scheme. Moreover, in this approach and in the method of the equivalent equations, the values of the non-conserved variables are assumed to stem from the point-wise discretization of smooth functions, whose existence and smoothness cannot be guaranteed because they are absent from the target PDEs.
Other approaches known in the literature are the asymptotic analysis under parabolic scaling deployed in \cite{junk2003rigorous, junk2005asymptotic, junk2009convergence} as well as the Maxwell iteration method \cite{yong2016theory, zhao2017maxwell}, which shares strong bonds with the equivalent equations method presented before.
The previous list of formal analysis techniques does not aim at being exhaustive (the interested reader can refer to \cite{kruger2017lattice}) and one should be aware that, despite efforts in this direction \cite{caiazzo2009comparison}, there is no consensus on which is the right method to use \cite{kruger2017lattice}.

    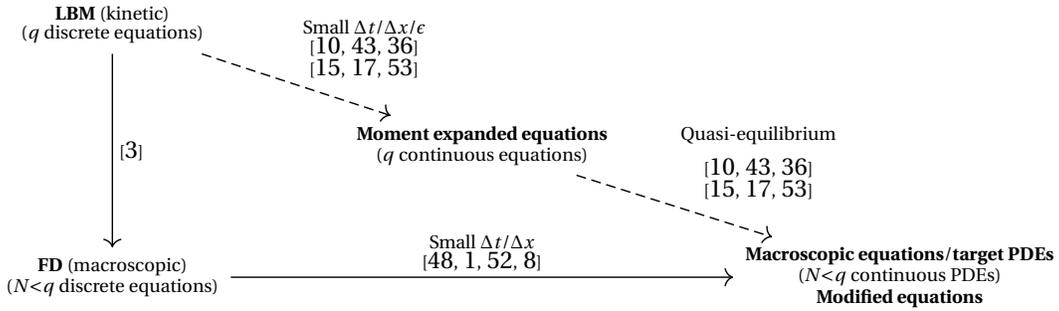
\begin{figure}
      \begin{center}
      \begin{normalsize}
      \begin{equation*}
      \begin{tikzcd}[row sep=large, column sep=0.6in]
        \substack{\text{\strong{LBM} (kinetic)} \\ (\velocitynumber \text{ discrete equations})} \arrow[dd, "\cite{bellotti2021fd}"] \arrow[dr, dashrightarrow, "\substack{\text{Small }\Delta t /  \Delta x /\epsilon \\ \cite{chen1998lattice, qian2000higher, lallemand2000theory} \\ \cite{dubois2008equivalent,dubois2019nonlinear,yong2016theory}}"] & & \phantom{A} \\
& \substack{\text{\strong{Moment expanded equations}} \\ (\velocitynumber \text{ continuous equations})}  \arrow[dr, dashrightarrow, "\substack{\text{Quasi-equilibrium} \\ \\ \cite{chen1998lattice, qian2000higher, lallemand2000theory} \\ \cite{dubois2008equivalent,dubois2019nonlinear,yong2016theory} }"] & \phantom{A} \\
\substack{\text{\strong{FD} (macroscopic)} \\ (\consmomentsnumber<\velocitynumber \text{ discrete equations})} \arrow[rr, "\substack{\text{Small }\Delta t /  \Delta x \\ \cite{strikwerda2004finite,allaire2007numerical,warming1974modified, carpentier1997derivation}}"] & & \substack{\text{\strong{Macroscopic equations}}/\text{\strong{target PDEs}} \\ (\consmomentsnumber<\velocitynumber \text{ continuous PDEs}) \\ \text{\strong{Modified equations}}} \\
      \end{tikzcd}
      \end{equation*}
      \end{normalsize}
      \end{center}\caption{\label{fig:PlanOfTheWork}Different paths to recover the macroscopic equations and the modified equations. The formal approaches available in the literature \cite{chen1998lattice, qian2000higher, lallemand2000theory, dubois2008equivalent,dubois2019nonlinear,yong2016theory} rely on the path marked with dashed arrows. They perform Taylor expansions for small discretization parameters and then utilize the quasi-equilibrium of the non-conserved moments to get rid of them. Our way of proceeding is marked with full arrows: we eliminate exactly the non-conserved moments at the discrete level as in \cite{bellotti2021fd} and we perform the usual analyses for \fd schemes as in \cite{strikwerda2004finite,allaire2007numerical, warming1974modified, carpentier1997derivation}.}
      \end{figure}

    A staple of all the previously mentioned approaches is that the expansion for the discretization parameters (time and space steps) tending to zero is performed on the kinetic numerical scheme, where both conserved and non-conserved variables are present. Eventually, the non-conserved variables are formally eliminated from the continuous formulation by scaling arguments, so to speak, using quasi-equilibrium. 
    This corresponds to follow the diagonal path on \Cref{fig:PlanOfTheWork}.
    In this contribution, we develop the other path, namely the top-down movement followed by the left-right one on \Cref{fig:PlanOfTheWork}.
    In particular, in order to fill the hollow between \lbm schemes and traditional approaches known to numerical analysts, such as \fd schemes, we recently introduced \cite{bellotti2021fd} a formalism to recast any \lbm scheme, regardless of its linearity, as a multi-step \fd scheme solely on the conserved moments.
    It should be stressed that our standpoint, where \lbm schemes are studied in terms of their \fd counterpart, must not be seen as the right way of implementing them, because one would lose most of the previously mentioned computational efficiency coming from the kinetic vision.
    Conversely, our way of writing the scheme should be seen as a sort of one-way mathematical transform to pass from a kinetic standpoint to a macroscopic one in a purely discrete setting.
    The elimination of the non-conserved moments is carried exactly on the discrete formulation by algebraic devices, thus independently from the time-space scaling. The price to pay for the non-conserved moments relaxing away from the equilibria is the multi-step nature of the \fd scheme.
    In our previous proposal \cite{bellotti2021fd}, it has been crucial to be able to provide, thanks to a systematic mathematical approach, a precise description of the main ingredient needed to reduce the \lbm scheme to a \fd scheme, namely the characteristic polynomial of matrices of \fd operators. 
    We are therefore allowed to utilize this characteristic polynomial as a tool satisfying certain properties alone from the particular underlying \lbm scheme.
    Quite the opposite, using the algorithm proposed by \cite{fuvcik2021equivalent}, one is compelled to explicitly write down the corresponding \fd scheme in order to perform the Taylor expansions to recover the target PDEs. In our case, the mathematical understanding that we \emph{a priori} have on the corresponding (macroscopic) \fd schemes, regardless of the (kinetic) \lbm scheme they stand for, allows the following theoretical discussion. 
    The theory of \fd schemes features two important notions. One is the concept of truncation error (Definition 5.1.3 \cite{gustafsson1995time} or Definition 2.2.4 \cite{allaire2007numerical}), which is rigorous and is the basic ingredient to prove the celebrated Lax equivalence theorem \cite{lax1956survey}. The computations of the truncation error are perfectly justified because of the existence and smoothness results on the target PDEs (\emph{e.g.} transport equation with smooth initial datum, Burgers equation with smooth non-decreasing initial datum, \emph{etc.}). The second one is the concept of modified equation \cite{warming1974modified, carpentier1997derivation}, which is formal. The modified equations are those which the numerical scheme is ``more consistent'' with, compared to the target PDEs, and thus they yield essential but formal information on the behavior of the scheme.  
    The modified equations cannot be fully justified even for \fd schemes because they assume that smooth functions which equal the discrete solution of the scheme at the grid points exist.

    The main findings presented in the paper are the following. 
    \begin{itemize}
      \item We propose a procedure to rigorously analyze the consistency (\emph{i.e.} the truncation error) of any \lbm scheme \emph{via} its corresponding \fd scheme.
      \item In the case of acoustic and diffusive scaling between time and space discretizations, we rigorously find the expression of the target PDEs approximated by any scheme and the truncation error. For the acoustic scaling, we also write the formal modified equations up to order two.
      \item Under acoustic scaling, these modified equations are the same than the ones obtained by \cite{dubois2019nonlinear} until second order.
      \item We rewrite the Maxwell iteration \cite{yong2016theory, zhao2017maxwell} for general \lbm schemes. This allows to show that both for the acoustic and diffusive scaling, the modified equations obtained through the corresponding \fd scheme are the same than the ones from the Maxwell iteration at any order.
    \end{itemize}
  Our derivation of the truncation errors is rigorous -- as the ones for \fd schemes -- and the formal modified equations rely on less unjustified assumptions than the existing approaches, for two main reasons. The first one is that Taylor expansions are applied to the conserved moments only, which also appear in the macroscopic equations. Therefore, one only postulates that the discrete conserved moments stem from the point-wise evaluation of smooth functions. The second one is that we solely rely on the link between time and space steps as the lattices are refined and which must be specified for any time-space numerical method.
    
    The paper is structured as follows.
In \Cref{sec:LBMSchemes}, we set notations and assumptions concerning the \lbm schemes we shall work with. 
\Cref{sec:FDforLBMSchemes} is devoted to recall the main results from our previous work \cite{bellotti2021fd} concerning the recast of any \lbm scheme as a \fd scheme. These results are then stated in a slightly different manner, facilitating the following analysis.
The main results of the work are stated in \Cref{sec:MainResults} and come under the form of two theorems. The proof of the first is detailed in \Cref{sec:DetailedProofs} under the assumption of dealing with one conservation law, for the sake of keeping the presentation and the notations as simple as possible.
In \Cref{sec:ExtansionSeveralConservedMoments}, we indicate how the previous proof is easily extended to several conservation laws, whereas \Cref{sec:LinkExistingApproaches} is devoted to hint the links with some available approaches to find the modified equations of \lbm schemes.
The conclusions and perspectives of this work are drawn in \Cref{sec:Conclusions}.

\section{Lattice Boltzmann schemes}\label{sec:LBMSchemes}

To start our contribution, we present the classical framework of the  multiple-relaxation-times (known as MRT) \lbm schemes, as introduced by \cite{dhumieres1992}.
For the sake of simplicity, we do not consider source terms which can be effortlessly introduced in the analysis.
This fixes the perimeter of the schemes we shall be allowed to treat and study in the sequel.

\subsection{Spatial and temporal discretization}
We set the problem in spatial dimension $\spatialdimensionality = 1, 2, 3$ considering the whole space $\reals^{\spatialdimensionality}$ since this work is not focused on the enforcement of boundary conditions.
All the following material is valid only under the assumption of working on unbounded domains or sufficiently far from a boundary.
The space is discretized by a $\spatialdimensionality$-dimensional lattice denoted $\lattice \definitionequality \spacestep \relatives^{\spatialdimensionality}$ with constant step $\spacestep > 0$.
The time is uniformly discretized with step $\timestep > 0$, rendering a time lattice $\timelattice \definitionequality \timestep \naturals$. The role of the initial conditions is not investigated and is a subject on its own, see \cite{van2009smooth, rheinlander2007analysis}.
We introduce the so-called ``lattice velocity'' $\latticevelocity > 0$ defined by $\latticevelocity \definitionequality \spacestep/\timestep$.
Observe that in the sequel, namely in \Cref{sec:MainResults}, we shall introduce particular relations between space step $\spacestep$ and $\timestep$ when $\spacestep \to 0$, in order to provide the main results of the work, as done in \cite{dubois2008equivalent,dubois2019nonlinear}.
However, until the end of \Cref{sec:FDforLBMSchemes}, the discussion remains valid for any choice of these parameters.

\subsection{Discrete velocities}

The discrete velocities are an essential ingredient of any lattice Boltzmann scheme. One has to choose $(\vectorial{\velocityletter}_{\populationindex})_{\populationindex = 1}^{\populationindex = \velocitynumber} \subset \reals^{\spatialdimensionality}$ with $\velocitynumber \in \naturals^{\star}$, discrete velocities, which are multiple of the lattice velocity $\latticevelocity$, namely $\vectorial{\velocityletter}_{\populationindex} = \latticevelocity \vectorial{\normalizedvelocityletter}_{\populationindex}$ for any $\populationindex \in \integerinterval{1}{\velocitynumber}$\footnote{This shall be a notation to indicate closed intervals of integers, namely for $a, b \in \relatives$ with $a \leq b$, then $\integerinterval{a}{b} \definitionequality \{ a, a+1, \dots, b\}$.} with $(\vectorial{\normalizedvelocityletter}_{\populationindex})_{\populationindex = 1}^{\populationindex = \velocitynumber} \subset \relatives^{\spatialdimensionality}$. 
Thus, the virtual particles are stuck to the lattice $\lattice$ at each time step of the method.
We denote the distribution density of the virtual particles moving with velocity $\vectorial{\velocityletter}_{\populationindex}$ by $\distributionletter_{\populationindex} = \distributionletter_{\populationindex}(\timevariable, \vectorial{\spacevariable})$ for every $\populationindex \in \integerinterval{1}{\velocitynumber}$, depending on the time and space variables.

\subsection{Lattice Boltzmann algorithm: collide and stream}
As mentioned in the Introduction, any lattice Boltzmann scheme consists in a kinetic algorithm made up of two phases: a local collision phase performed on each site of the lattice $\lattice$ and a stream phase where particles are exchanged between different sites of the lattice. 
Let us independently introduce them.

\begin{itemize}
    \item \emph{Collision phase}.
      We adopt the general point of view of the multiple-relaxation-times schemes, where the collision phase is written as a diagonal relaxation towards some equilibria in the moments basis, see \cite{dhumieres1992}.
      We introduce a change of basis called moment matrix $\momentsmatrix \in \lineargroup{\velocitynumber}{\reals}$.
      Gathering the distributions into $\vectorial{\distributionletter} = \transpose{(\distributionletter_{1}, \dots, \distributionletter_{\velocitynumber})}$, the moments are recovered by $\vectorial{\momentletter} = \momentsmatrix \vectorial{\distributionletter}$ and \emph{vice versa}.
      We also introduce
      \begin{itemize}
          \item the matrix $\matricial{\identity} \in \lineargroup{\velocitynumber}{\reals}$, the identity matrix of size $\velocitynumber$;
          \item the matrix $\relaxationmatrix \in \matrixspace{\velocitynumber}{\reals}$, called relaxation matrix. This matrix is diagonal with $\consmomentsnumber \in \integerinterval{1}{ \velocitynumber - 1}$ being the number of conserved moments $\relaxationmatrix = \diagmatrix(\relaxparletter_{1}, \dots, \relaxparletter_{\consmomentsnumber}, \relaxparletter_{\consmomentsnumber + 1}, \dots, \relaxparletter_{\velocitynumber})$, where $\relaxparletter_{\indiceconserved} \in \reals$ for $\indiceconserved \in \integerinterval{1}{\consmomentsnumber}$ for the conserved moments and $\relaxparletter_{\indicemoments} \in ]0, 2]$ for $\indicemoments \in \integerinterval{\consmomentsnumber + 1}{\velocitynumber}$, see \cite{dubois2008equivalent}, for the non-conserved ones.
          Observe that the relaxation parameters corresponding to the conserved moments do not play any role in the \lbm algorithm, therefore the matrix $\relaxationmatrix$ can be singular, without any specific issue.
          In particular, we shall prove in \Cref{sec:InvarianceChoiceRelaxationParamtersConserved} that the choice of relaxation parameter for the conserved variables does not have any influence on the outcomes presented in this work.
          For the sake of presentation, we start numbering the moments by the conserved ones;
          \item we employ the notation $\vectorial{\momentletter}^{\atequilibrium}(\timevariable, \vectorial{\spacevariable}) = \vectorial{\momentletter}^{\atequilibrium} (\momentletter_{1} (\timevariable, \vectorial{\spacevariable}), \dots, \momentletter_{\consmomentsnumber} (\timevariable, \vectorial{\spacevariable}))$ for $\timevariable \in \timelattice$ and $\vectorial{\spacevariable} \in \lattice$, where $\vectorial{\momentletter}^{\atequilibrium}: \reals^{\consmomentsnumber} \to \reals^{\velocitynumber}$ are possibly non-linear functions of the conserved moments.
          In order to guarantee that the first $\consmomentsnumber$ moments are conserved through the collision process, irrespective of the values of $\relaxparletter_{1}, \dots, \relaxparletter_{\consmomentsnumber}$, the constraints
          \begin{equation}\label{eq:EquilibriumConservedMoments}
              \momentletter_{\indiceconserved}^{\atequilibrium}(\momentletter_1, \dots, \momentletter_{\consmomentsnumber}) = \momentletter_{\indiceconserved}, \qquad \forall \indiceconserved \in \integerinterval{1}{\consmomentsnumber},
          \end{equation}
          must hold \cite{bouchut2000diffusive}.
      \end{itemize}
      Let $\timevariable\in \timelattice$ and $\vectorial{\spacevariable} \in \lattice$, the collision phase reads, denoting by $\collided$ any post-collision state
      \begin{equation}\label{eq:CollisionPhase}
        \vectorial{\momentletter}^{\collided}(\timevariable, \vectorial{\spacevariable}) = (\matricial{I} - \relaxationmatrix) \vectorial{\momentletter}(\timevariable, \vectorial{\spacevariable}) + \relaxationmatrix \vectorial{\momentletter}^{\atequilibrium}(\timevariable, \vectorial{\spacevariable}).
      \end{equation}
      \item \emph{Stream phase}.
      The stream phase is diagonal in the space of the distributions and consist in an exact upwind advection of the particle distribution densities. It can be written, for $\timevariable \in \timelattice$ and $\vectorial{\spacevariable} \in \lattice$, as
      \begin{equation}\label{eq:StreamPhase}
        \distributionletter_{\populationindex} (\timevariable + \timestep, \vectorial{\spacevariable}) = \distributionletter_{\populationindex}^{\collided} (\timevariable, \vectorial{\spacevariable} - \vectorial{\normalizedvelocityletter}_{\populationindex} \spacestep),
      \end{equation}
      for any $\populationindex \in \integerinterval{1}{\velocitynumber}$.
\end{itemize}

\section{Finite Difference formulation of a \lbm scheme}\label{sec:FDforLBMSchemes}

Having defined the \lbm schemes, we briefly introduce the setting allowing us to rewrite any \lbm scheme (kinetic) as a multi-step \fd scheme (macroscopic) on the $\consmomentsnumber$ conserved moments of interest.
The interested reader can refer to our previous contribution \cite{bellotti2021fd} for more details.
Then, the formulation of the multi-step \fd scheme is given using a more compact notation which is more suitable to the following discussion.
We start with the assumptions needed in the sequel.
\begin{assumptions}[\fd assumptions]\label{ass:FDAssumptionsDependencies}
    The entries of $\momentsmatrix$ and $\relaxationmatrix$ can depend on $\spacestep$ and/or on $\timestep$ but cannot be a function of the time and space variables.
\end{assumptions}

\subsection{Algebraic setting}
Let us first introduce the necessary algebraic setting.
In particular, we define the shift operators associated with each discrete velocity as well as the derived \fd operators in space.
In the following Definition, the time variable does not play any role since kept frozen, thus it is not listed for the sake of readability.
\begin{definition}[Shift and \fd operators in space]\label{def:ShiftandFDOperators}
    Let $\vectorial{z} \in \relatives^{\spatialdimensionality}$, then the associated shift operator on the lattice $\lattice$, denoted $\shiftoperator{\vectorial{z}}$, is defined in the following way.
    Take $\genericfunction: \lattice \to \reals$ be any function defined on the lattice, then the action of $\shiftoperator{\vectorial{z}}$ is
    \begin{displaymath}
      (\shiftoperator{\vectorial{z}}\genericfunction)(\vectorial{\spacevariable}) = \genericfunction(\vectorial{\spacevariable} - \vectorial{z} \spacestep), \qquad \forall \vectorial{\spacevariable} \in \lattice.
    \end{displaymath}
    We also introduce $\groupunitsshiftoperators \definitionequality \{\shiftoperator{\vectorial{z}} \quad \text {with} \quad \vectorial{z} \in \relatives^{\spatialdimensionality} \} \isomorphic \relatives^{\spatialdimensionality}$.
    The product $\prodoperators : \groupunitsshiftoperators  \times \groupunitsshiftoperators  \to \groupunitsshiftoperators $ of two shift operators is defined by
    \begin{equation*}
        \shiftoperator{\vectorial{z}} \prodoperators \shiftoperator{\vectorial{w}} \definitionequality \shiftoperator{\vectorial{z} + \vectorial{w}}, \qquad \forall \vectorial{z}, \vectorial{w} \in \relatives^{\spatialdimensionality}.
    \end{equation*}
    The set of \fd operators on the lattice $\lattice$ is defined as
    \begin{equation}\label{eq:CompactSupport}
      \setfinitedifferenceoperators \definitionequality \reals \groupunitsshiftoperators =  \left \{ \sum\nolimits_{\genericunit \in \groupunitsshiftoperators} \coeffgenericfinitedifference_{\genericunit} \genericunit, \quad\text{where} \quad \coeffgenericfinitedifference_{\genericunit} \in \reals \quad \text{and} \quad \coeffgenericfinitedifference_{\genericunit} = 0 \quad \text{almost everywhere} \right  \},
    \end{equation}
    the group ring (or group algebra) of $\groupunitsshiftoperators$ over $\reals$.
    The sum $\sumoperators : \setfinitedifferenceoperators \times \setfinitedifferenceoperators \to \setfinitedifferenceoperators$ and the product $\prodoperators : \setfinitedifferenceoperators \times \setfinitedifferenceoperators \to \setfinitedifferenceoperators$ of two elements are defined by
      \begin{equation*}
        \left (\sum_{\genericunit \in \groupunitsshiftoperators} \alpha_{\genericunit} \genericunit \right ) \sumoperators \left (\sum_{\genericunit \in \groupunitsshiftoperators} \beta_{\genericunit} \genericunit \right ) = \sum_{\genericunit \in \groupunitsshiftoperators} (\alpha_{\genericunit} + \beta_{\genericunit}) \genericunit, \qquad
        \left (\sum_{\genericunit \in \groupunitsshiftoperators} \alpha_{\genericunit} \genericunit \right ) \prodoperators \left (\sum_{\mathsf{h} \in \groupunitsshiftoperators} \beta_{\mathsf{h}} \mathsf{h} \right ) = \sum_{\genericunit, \mathsf{h} \in \groupunitsshiftoperators} (\alpha_{\genericunit}\beta_{\mathsf{h}}) (\genericunit \prodoperators \mathsf{h}).
      \end{equation*}
      Furthermore, the product of $\sigma \in \reals$ with elements of $\setfinitedifferenceoperators$ is given by
      \begin{equation*}
        \sigma \left (\sum_{\genericunit \in \groupunitsshiftoperators} \alpha_{\genericunit} \genericunit \right ) = \sum_{\genericunit \in \groupunitsshiftoperators} (\sigma \alpha_{\genericunit})\genericunit .
      \end{equation*}
  \end{definition}
  In the sequel, the products $\prodoperators$ are understood.
  \begin{remark}
    We could achieve exactly the same construction, following Chapter 2 in \cite{cheng2003partial}, by considering functions on the lattice $\lattice$ as sequences and the \fd operators as sequences with compact support (whence the almost everywhere requirement in \eqref{eq:CompactSupport}). Then, the product $\prodoperators$ can be seen as a convolution (Cauchy) product between compactly supported sequences and the action of a \fd operator on a function as the convolution of a finitely supported sequence with a generic sequence.    
  \end{remark}
      Upon introducing the generating displacements along each axis $\basicshift_{k} \definitionequality \shiftoperator{\canonicalbasisvector_k}$ where $\canonicalbasisvector_k$ is the $k$-th vector of the canonical basis, for any $k \in \integerinterval{1}{\spatialdimensionality}$, we can isomorphically identify $\setfinitedifferenceoperators \isomorphic \polynomialring{\reals}{\basicshift_1, \basicshift_1^{-1}, \dots, \basicshift_{\spatialdimensionality}, \basicshift_{\spatialdimensionality}^{-1}}$, the ring of multivariate Laurent polynomials.
      The real numbers $\reals$ can be viewed as sub-ring of $\setfinitedifferenceoperators$, being the constant polynomials.
      This identification can be somehow interpreted as the historical starting point of umbral calculus \cite{roman2005umbral}, also known as calculus of Finite Differences \cite{miller1960finitedifference}: allow to interchange indices in sequences (operators or functions) with exponents (in polynomials).
  The stream phase \eqref{eq:StreamPhase} can be recast under its non-diagonal form in the space of moments \cite{yong2016theory,dubois2019nonlinear} by introducing what we call the moments-stream matrix $\streammoments \definitionequality \momentsmatrix \diagmatrix(\shiftoperator{\vectorial{\normalizedvelocityletter}_{1}}, \dots, \shiftoperator{\vectorial{\normalizedvelocityletter}_{\velocitynumber}}) \momentsmatrix^{-1} \in \matrixspace{\velocitynumber}{\setfinitedifferenceoperators}$ and merged with the collision phase \eqref{eq:CollisionPhase} to obtain the scheme, for any $\timevariable \in \timelattice$ and for any $\vectorial{\spacevariable} \in \lattice$
  \begin{equation}\label{eq:SchemeAB}
    \vectorial{\momentletter}(\timevariable + \timestep, \vectorial{\spacevariable}) = \schememoments \vectorial{\momentletter}(\timevariable, \vectorial{\spacevariable}) + \schemeequil \vectorial{\momentletter}^{\atequilibrium}(\timevariable, \vectorial{\spacevariable}),
  \end{equation}
  where $\schememoments \definitionequality \streammoments (\vectorial{\identity} - \relaxationmatrix) \in \matrixspace{\velocitynumber}{\setfinitedifferenceoperators}$ and $\schemeequil \definitionequality \streammoments \relaxationmatrix \in \matrixspace{\velocitynumber}{\setfinitedifferenceoperators}$.

  \subsection{Corresponding \fd schemes}
  With this new compact algebraic form of any \lbm scheme, namely \eqref{eq:SchemeAB}, we are able to recall the main results proved in \cite{bellotti2021fd}.
  These results encompass the findings from \cite{suga2010accurate}, \cite{dellacherie2014construction} and \cite{fuvcik2021equivalent}.
  The version for one conserved moment can be formulated as follows.

  \begin{proposition}[Corresponding \fd scheme for $\consmomentsnumber = 1$, \cite{bellotti2021fd}]\label{prop:ReductionFiniteDifferenceGeneralOld}
    Consider $\consmomentsnumber = 1$. Then the \lbm scheme given by \eqref{eq:SchemeAB} corresponds to a multi-step explicit macroscopic \fd scheme on the conserved moment $\momentletter_1$ under the form
    \begin{equation} \label{eq:FDSchemeOldPaper}
      \momentletter_{1} (\timevariable + \timestep, \vectorial{\spacevariable}) = -\sum_{\indicetimeshift = 0}^{\velocitynumber - 1} \coeffcharact_{\indicetimeshift} \momentletter_{1} (\timevariable +  (1 - \velocitynumber + \indicetimeshift) \timestep, \vectorial{\spacevariable}) + \left (\sum_{\indicetimeshift = 0}^{\velocitynumber-1} \left ( \sum_{\ell=0}^{\indicetimeshift} \coeffcharact_{\velocitynumber + \ell - \indicetimeshift} \schememoments^{\ell} \right ) \schemeequil \vectorial{\momentletter}^{\atequilibrium} (\timevariable - \indicetimeshift \timestep, \vectorial{\spacevariable}) \right )_1, 
    \end{equation}
    for all $\timevariable \in \timelattice$ and for all $\vectorial{\spacevariable} \in \lattice$, where $(\coeffcharact_{\indicepolynomials})_{\indicepolynomials = 0}^{\indicepolynomials = \velocitynumber}\subset \setfinitedifferenceoperators$ are the coefficients of $\determinant (\polynomialunknown \matricial{\identity} - \schememoments) = \sum_{\indicepolynomials = 0}^{\indicepolynomials = \velocitynumber} \coeffcharact_{\indicepolynomials} \polynomialunknown^{\indicepolynomials}$, the characteristic polynomial of $\schememoments$, with $\determinant(\cdot)$ indicating the determinant of a matrix.
  \end{proposition}
  The proof -- given in \cite{bellotti2021fd} -- relies on the fact that $\setfinitedifferenceoperators$ is a commutative ring and that therefore the Cayley-Hamilton theorem \cite{brewer1986linear}, stipulating that any square matrix with entries in a commutative ring annihilates its characteristic polynomial, holds.

  This result is easily generalized to the case of multiple conservation laws, namely $\consmomentsnumber > 1$.
  For this, let us introduce a new notation. 
  For any square matrix $\genericmatrix \in \matrixspace{\velocitynumber}{\genericcommutativering}$ on a commutative ring $\genericcommutativering$, consider $\cutmatrixsquare{\genericmatrix}{\setoflines} \definitionequality  ( \sum_{\indiceslines \in \setoflines} \canonicalbasisvector_{\indiceslines} \otimes \canonicalbasisvector_{\indiceslines}  ) \genericmatrix   ( \sum_{\indiceslines \in \setoflines} \canonicalbasisvector_{\indiceslines} \otimes \canonicalbasisvector_{\indiceslines}  ) \in \matrixspace{\velocitynumber}{\genericcommutativering}$ for any $\setoflines \subset \integerinterval{1}{\velocitynumber}$, corresponding to the matrix where only the entries with row and column indices in $\setoflines$ are kept and the remaining ones are set to zero.
  Then we have the following statement.
  \begin{proposition}[Corresponding \fd schemes for $\consmomentsnumber \geq 1$, \cite{bellotti2021fd}]\label{prop:ReductionFiniteDifferenceNGeq1Old}
    Consider $\consmomentsnumber \geq 1$. Then the \lbm scheme given by \eqref{eq:SchemeAB} corresponds to a family of multi-step explicit macroscopic \fd schemes on the conserved moments $\momentletter_{1}, \dots, \momentletter_{\consmomentsnumber}$.
    This is, for any $\indiceconserved \in \integerinterval{1}{\consmomentsnumber}$
    \begin{align}
      \momentletter_{\indiceconserved} (\timevariable + \timestep, \vectorial{\spacevariable}) = -\sum_{\indicetimeshift = 0}^{\velocitynumber - \consmomentsnumber} \coeffcharact_{\indiceconserved, \indicetimeshift} \momentletter_{\indiceconserved} (\timevariable + (\indicetimeshift - \velocitynumber + \consmomentsnumber) \timestep, \vectorial{\spacevariable{}}) &+ \left ( \sum_{\indicetimeshift = 0}^{\velocitynumber-\consmomentsnumber} \left ( \sum_{\ell=0}^{\indicetimeshift} \coeffcharact_{\indiceconserved, \velocitynumber +1 - \consmomentsnumber + \ell - \indicetimeshift} \schememoments_{\indiceconserved}^{\ell} \right ) \schememomentsother_{\indiceconserved}\vectorial{\momentletter}(\timevariable - \indicetimeshift \timestep, \vectorial{\spacevariable}) \right )_{\indiceconserved} \label{eq:FDSchemeOldPaperNGeq1} \\
      &+ \left ( \sum_{\indicetimeshift = 0}^{\velocitynumber-\consmomentsnumber} \left ( \sum_{\ell=0}^{\indicetimeshift} \coeffcharact_{\indiceconserved, \velocitynumber + 1 - \consmomentsnumber + \ell - \indicetimeshift} \schememoments_{\indiceconserved}^{\ell} \right ) \schemeequil \vectorial{\momentletter}^{\atequilibrium}(\timevariable - \indicetimeshift \timestep, \vectorial{\spacevariable}) \right )_{\indiceconserved}, \nonumber
    \end{align} 
    for all $\timevariable \in \timelattice$ and $ \vectorial{\spacevariable} \in \lattice$, where $ \schememoments_{\indiceconserved} \definitionequality \cutmatrixsquare{\schememoments}{\{ \indiceconserved \} \cup \integerinterval{\consmomentsnumber + 1}{\velocitynumber}}$ and $ \schememomentsother_{\indiceconserved} \definitionequality \schememoments - \schememoments_{\indiceconserved}$, with  $(\coeffcharact_{\indiceconserved, \indicepolynomials})_{\indicepolynomials = 0}^{\indicepolynomials = \velocitynumber + 1 -\consmomentsnumber}\subset \setfinitedifferenceoperators$ being the coefficients of $\determinant (\polynomialunknown \matricial{\identity} - \schememoments_{\indiceconserved}) = \polynomialunknown^{\consmomentsnumber - 1} \sum_{\indicepolynomials = 0}^{\indicepolynomials = \velocitynumber + 1 - \consmomentsnumber}\coeffcharact_{\indiceconserved, \indicepolynomials} \polynomialunknown^{\indicepolynomials} $, the characteristic polynomial of $ \schememoments_{\indiceconserved} $.
  \end{proposition}
  This result is the natural generalization of \Cref{prop:ReductionFiniteDifferenceGeneralOld} to the case $\consmomentsnumber > 1$, in the sense that each sub-problem for any $\indiceconserved \in \integerinterval{1}{\consmomentsnumber}$ deals with one conserved moment (the $\indiceconserved$-th) at each time, only trying to eliminate the non-conserved moments while keeping the conserved ones other than the $\indiceconserved$-th.
  This is achieved by using a tailored characteristic polynomial for each conserved moment in the problem. 

  \begin{figure}[h]
    \begin{center}
        \includegraphics[width = 0.7\textwidth]{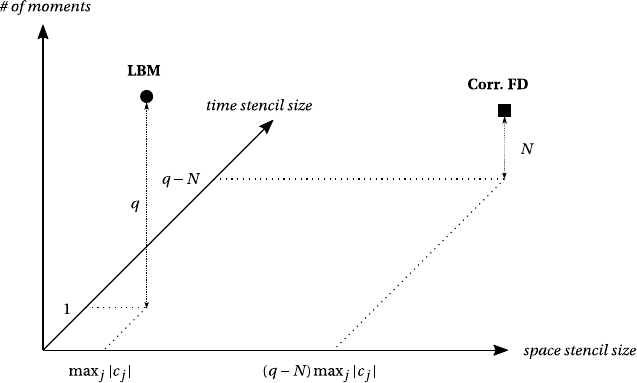}
    \end{center}\caption{\label{fig:TimeSpaceTransformation}Comparison between \lbm scheme (circle) and corresponding \fd schemes (square) in terms of involved moments (respectively $\velocitynumber$ and $\consmomentsnumber$), number of time steps (respectively $1$ and $\velocitynumber - \consmomentsnumber$) and size of the maximal spatial stencil (respectively $\max_{\populationindex}|\normalizedvelocityletter_{\populationindex}|$ and $(\velocitynumber - \consmomentsnumber)\max_{\populationindex}|\normalizedvelocityletter_{\populationindex}|$).}
  \end{figure}

  Observe that each scheme in \eqref{eq:FDSchemeOldPaperNGeq1} is \emph{a priori} a $(\velocitynumber - \consmomentsnumber)$-steps scheme (thus with $\velocitynumber - \consmomentsnumber + 1$ stages), see \Cref{fig:TimeSpaceTransformation}. When some non-conserved moment relaxes to its equilibrium, the schemes in \eqref{eq:FDSchemeOldPaperNGeq1} involve less time steps, see \cite{bellotti2021fd} for more details. However, this question is marginal in the present work since the results we will demonstrate hold whatever the number of steps in \eqref{eq:FDSchemeOldPaperNGeq1} and thus for any value of the relaxation parameters $\relaxparletter_{\consmomentsnumber +1}, \dots, \relaxparletter_{\velocitynumber}$ for the non-conserved moments.
  Further comments on \Cref{prop:ReductionFiniteDifferenceGeneralOld} and \Cref{prop:ReductionFiniteDifferenceNGeq1Old} are postponed to the following Section.

  \subsection{A more compact form of corresponding \fd schemes}

  Although the asymptotic analysis we shall develop in \Cref{sec:DetailedProofs} can be carried on the formulations from \Cref{prop:ReductionFiniteDifferenceGeneralOld} and \Cref{prop:ReductionFiniteDifferenceNGeq1Old} previously introduced in \cite{bellotti2021fd}, we propose a different formalism based on shift operators in time.
  Having utilized both approaches, the advantage of this new standpoint -- which shall be adopted in this paper -- is to easily deal with the asymptotic analysis of the coefficients of the characteristic polynomial and of the powers of the matrix $\schememoments$ on the right hand side of \eqref{eq:FDSchemeOldPaper} or \eqref{eq:FDSchemeOldPaperNGeq1}.
  In particular, this allows for the straightforward generalization of the procedure above second-order.
  Furthermore, the links with other asymptotic analysis of \lbm schemes from the literature -- which we shall develop in \Cref{sec:LinkExistingApproaches} -- become noticeably more transparent.
  To this end, we introduce the following Definition.
  \begin{definition}[Shift operator in time]
      Let $\genericfunction : \timelattice \to \reals$ be any function defined on the time lattice, then the time shift operator $\timeshiftoperator$ acts as
      \begin{equation*}
          (\timeshiftoperator \genericfunction) (\timevariable) = \genericfunction (\timevariable + \timestep), \qquad \forall \timevariable \in \timelattice.
      \end{equation*}
  \end{definition}
  With this, the scheme \eqref{eq:SchemeAB} can be recast under the fully-operatorial form: for any $\timevariable \in \timelattice$ and for any $ \vectorial{\spacevariable} \in \lattice$
  \begin{equation}\label{eq:SchemeFullyOperator}
      (\timeshiftoperator \matricial{\identity} - \schememoments) \vectorial{\momentletter}(\timevariable, \vectorial{\spacevariable}) = \schemeequil \vectorial{\momentletter}^{\atequilibrium}(\timevariable, \vectorial{\spacevariable}),
  \end{equation}
  which corresponds to taking the $Z$-transform \cite{jury1964theory} of the scheme in the variable $\timeshiftoperator$. Here, the inverse of the resolvent associated with $\schememoments$, namely $\timeshiftoperator \matricial{\identity} - \schememoments \in \matrixspace{\velocitynumber}{\discretetimespaceoperators}$, where $\discretetimespaceoperators \isomorphic \polynomialring{\reals}{\timeshiftoperator, \basicshift_{1}, \basicshift_{1}^{-1}, \dots, \basicshift_{\spatialdimensionality}, \basicshift_{\spatialdimensionality}^{-1}}$, with $\tensorProductRing{\reals}$ indicating the tensor product of $\reals$-algebras (see Chapter 16 in \cite{lang02algebra} or Chapter 2 in \cite{kassel95quantum}), forms a commutative ring.
  In the sequel, we shall drop the time and the space variables when not strictly needed for the sake of readability, because the system given by \eqref{eq:SchemeFullyOperator} is intrinsically time and space invariant thanks to \Cref{ass:FDAssumptionsDependencies} and since we work on an unbounded domain, without considering the initial conditions.

  \theoremstyle{plain} 
  \newtheorem*{T1}{\Cref{prop:ReductionFiniteDifferenceGeneralOld} bis}
  \newtheorem*{T2}{\Cref{prop:ReductionFiniteDifferenceNGeq1Old} bis}

  \begin{T1}[Corresponding \fd scheme for $\consmomentsnumber = 1$]\label{prop:ReductionFiniteDifferenceGeneralAdjoint}
    Consider $\consmomentsnumber = 1$. Then the \lbm scheme given by \eqref{eq:SchemeAB} or \eqref{eq:SchemeFullyOperator} corresponds to a multi-step explicit macroscopic \fd scheme on the conserved moment $\momentletter_1$ under the form
    \begin{equation}\label{eq:FDSchemeAdjoint}
      \determinant (\timeshiftoperator \matricial{\identity} - \schememoments) \momentletter_{1} = \left ( \adjugate (\timeshiftoperator \matricial{\identity} - \schememoments) \schemeequil \vectorial{\momentletter}^{\atequilibrium} \right )_1,
    \end{equation}
    where $\adjugate(\cdot)$ indicates the adjugate matrix,\footnote{It is worthwhile observing that the determinant and the adjugate matrix are defined for any square matrix with elements in a commutative ring.} also known as classical adjoint, which is the transpose of the cofactor matrix \cite{horn2012matrix}.

    Up to a temporal shift of the whole scheme, the corresponding multi-step explicit \fd scheme by \eqref{eq:FDSchemeAdjoint} equals the one from \eqref{eq:FDSchemeOldPaper}.
  \end{T1}
  \begin{proof}
    The proof can be done starting from \Cref{prop:ReductionFiniteDifferenceGeneralOld}. Alternatively, using the fundamental relation between adjugate and determinant, see Chapter 0 in \cite{horn2012matrix}, which is a consequence of the Laplace formula, we have that for any $\genericmatrix \in \matrixspace{\velocitynumber}{\genericcommutativering}$ where $\genericcommutativering$ is any commutative ring
    \begin{equation}\label{eq:FundamentalRelationAdjugate}
      \genericmatrix \adjugate(\genericmatrix) = \adjugate(\genericmatrix) \genericmatrix = \determinant(\genericmatrix) \matricial{\identity}.
    \end{equation}
    Hence, multiplying \eqref{eq:SchemeFullyOperator} by $\adjugate (\timeshiftoperator \matricial{\identity} - \schememoments)$ yields $\determinant (\timeshiftoperator \matricial{\identity} - \schememoments) \vectorial{\momentletter} =  \adjugate (\timeshiftoperator \matricial{\identity} - \schememoments) \schemeequil \vectorial{\momentletter}^{\atequilibrium}$. Selecting the first row gives \eqref{eq:FDSchemeAdjoint}.
  \end{proof}

  \begin{remark}[From kinetic to macroscopic]\label{rem:FromKineticToMacroscopic}
    We observe the following facts:
    \begin{itemize}
      \item The procedure can be reversed -- when keeping all the lines in $\determinant (\timeshiftoperator \matricial{\identity} - \schememoments) \vectorial{\momentletter} =  \adjugate (\timeshiftoperator \matricial{\identity} - \schememoments) \schemeequil \vectorial{\momentletter}^{\atequilibrium}$ -- using a multiplication by $\timeshiftoperator \matricial{\identity} - \schememoments$ and then dividing by the polynomial $\determinant (\timeshiftoperator \matricial{\identity} - \schememoments)$. In this way, one comes back to the \lbm scheme by \eqref{eq:SchemeFullyOperator}. This can be done as long as one does not select and store only the first row as in \eqref{eq:FDSchemeAdjoint}.
      Contrarily, if this selection is performed, the irreversible passage from the kinetic to the macroscopic formulation is accomplished. The non-conserved moments $\momentletter_2, \dots, \momentletter_{\velocitynumber}$ are no longer defined and they cannot be recovered from \eqref{eq:FDSchemeAdjoint}.
      This fact has been observed by \cite{dellacherie2014construction}: the same macroscopic \fd scheme can correspond to distinct \lbm schemes which can have different evolution equations for the non-conserved moments $\momentletter_2, \dots, \momentletter_{\velocitynumber}$.
      This is not surprising, since for a given monic polynomial, one can find an infinite number of matrices of which it is the characteristic polynomial.
      \item Though -- as previously emphasized -- the non-conserved moments are no longer present in the macroscopic \fd scheme by \eqref{eq:FDSchemeAdjoint}, there is a residual shadow of their presence, namely the multi-step nature of the \fd scheme, see \Cref{fig:TimeSpaceTransformation}. In particular, each non-conserved moment $\momentletter_{\indicemoments}$ relaxing away from the equilibrium, namely with $\relaxparletter_{\indicemoments} \neq 1$, for $\indicemoments \in \integerinterval{2}{\velocitynumber}$, adds a time step to the corresponding \fd scheme solely acting on the conserved moment $\momentletter_1$.
    \end{itemize}
    
  \end{remark}
  \begin{remark}[Adjugate and characteristic polynomial]
    A time shift and a change of variable in \eqref{eq:FDSchemeOldPaper} allows to express $\adjugate (\timeshiftoperator \matricial{\identity} - \schememoments)$ as a polynomial in $\timeshiftoperator$ of degree $\velocitynumber - 1$ computed from the characteristic polynomial. This relation is indeed classical and reads
    \begin{equation*}
      \adjugate (\timeshiftoperator \matricial{\identity} - \schememoments) = \sum_{\indicepolynomials = 0}^{\velocitynumber - 1} \left (\sum_{\ell = 0}^{\velocitynumber - 1 - \indicepolynomials} \coeffcharact_{\indicepolynomials + \ell + 1} \schememoments^{\ell} \right ) \timeshiftoperator^{\indicepolynomials}, \qquad \text{where} \qquad \determinant (\timeshiftoperator \matricial{\identity} - \schememoments) = \sum_{\indicepolynomials = 0}^{\velocitynumber} \coeffcharact_{\indicepolynomials} \timeshiftoperator^{\indicepolynomials}.
    \end{equation*}
  \end{remark}

  In the same way, we can restate \Cref{prop:ReductionFiniteDifferenceNGeq1Old} using the new formalism.
  \begin{T2}[Corresponding \fd schemes for $\consmomentsnumber \geq 1$]\label{prop:ReductionFiniteDifferenceNGeq1}
    Consider $\consmomentsnumber \geq 1$. Then the \lbm scheme given by \eqref{eq:SchemeAB}  or \eqref{eq:SchemeFullyOperator} corresponds to a family of multi-step explicit macroscopic \fd schemes on the conserved moments $\momentletter_{1}, \dots, \momentletter_{\consmomentsnumber}$.
    This is, for any $\indiceconserved \in \integerinterval{1}{\consmomentsnumber}$
    \begin{equation}\label{eq:FDSchemeAdjointGeneral}
      \determinant (\timeshiftoperator \matricial{\identity} - \schememoments_{\indiceconserved}) \momentletter_{\indiceconserved} =  \left ( \adjugate (\timeshiftoperator \matricial{\identity} - \schememoments_{\indiceconserved}) \schememomentsother_{\indiceconserved} \vectorial{\momentletter} \right )_{\indiceconserved} +  \left ( \adjugate (\timeshiftoperator \matricial{\identity} - \schememoments_{\indiceconserved}) \schemeequil \vectorial{\momentletter}^{\atequilibrium} \right )_{\indiceconserved}.
    \end{equation}

    Up to a temporal shift of the whole scheme, the corresponding multi-step explicit \fd scheme by \eqref{eq:FDSchemeAdjointGeneral} equals the one from \eqref{eq:FDSchemeOldPaperNGeq1}.

  \end{T2}

  We could call the form of \fd scheme from \Cref{prop:ReductionFiniteDifferenceNGeq1Old}  and \Cref{prop:ReductionFiniteDifferenceNGeq1Old} bis ``canonical'' since we shall prove in \Cref{sec:InvarianceChoiceRelaxationParamtersConserved} that it guarantees that the \fd schemes do not depend on the choice of relaxation parameters for the conserved variables, which do not play any role in the original \lbm scheme either, as previously discussed.

  \begin{remark}[Lack of scaling assumption]\label{rem:LackOfScaling}
    The results in \Cref{prop:ReductionFiniteDifferenceGeneralOld}, \Cref{prop:ReductionFiniteDifferenceNGeq1Old}, \Cref{prop:ReductionFiniteDifferenceGeneralOld} bis, \Cref{prop:ReductionFiniteDifferenceNGeq1Old} bis are fully discrete and do not make any assumption on the particular scaling between the time-step $\timestep$ and the space-step $\spacestep$.
  \end{remark}
  The previous \Cref{rem:LackOfScaling} signifies that the corresponding \fd schemes can be utilized to assess the consistency of the underlying \lbm scheme with respect to the macroscopic equations for any particular scaling between time and space discretizations, as we will showcase in \Cref{sec:MainResults}.

  \subsection{On the choice of relaxation parameters for the conserved moments}\label{sec:InvarianceChoiceRelaxationParamtersConserved}

  In \Cref{sec:LBMSchemes}, we have observed that the choice of relaxation parameters for the conserved moments, namely $\relaxparletter_{1}, \dots, \relaxparletter_{\consmomentsnumber}$, does not change the  \lbm scheme \eqref{eq:CollisionPhase}.
  However, it could be argued that different choices for $\relaxparletter_{1}, \dots, \relaxparletter_{\consmomentsnumber}$ can affect the formulations of the corresponding \fd schemes resulting from \Cref{prop:ReductionFiniteDifferenceGeneralOld} bis and \Cref{prop:ReductionFiniteDifferenceNGeq1Old} bis.
  We now show that, as one could hope, this is not the case for the \fd schemes given by \Cref{prop:ReductionFiniteDifferenceNGeq1Old} bis.

  \begin{proposition}
    The multi-step explicit macroscopic \fd schemes given by \eqref{eq:FDSchemeAdjointGeneral} in \Cref{prop:ReductionFiniteDifferenceNGeq1Old} bis do not depend on the choice of $\relaxparletter_{1}, \dots, \relaxparletter_{\consmomentsnumber}$, the relaxation parameters of the conserved moments.
  \end{proposition}
  \begin{proof}
    Fix the indices of the conserved moment $\indiceconserved \in \integerinterval{1}{\consmomentsnumber}$. Let us decompose $\schemeequil$, the part of the \lbm scheme dealing with the equilibria, as follows: $\schemeequil = \vectorequilibrium_{\indiceconserved} \otimes \canonicalbasisvector_{\indiceconserved} + \schemeequil |_{\relaxparletter_{\indiceconserved = 0}}$ where $ \vectorequilibrium_{\indiceconserved}  = \schemeequil_{\cdot, \indiceconserved}$ is the $\indiceconserved$-th column of $\schemeequil$.
On the one hand, the dependency of $\schemeequil$ on the choice of $\relaxparletter_{\indiceconserved}$ is now fully contained in $\vectorequilibrium_{\indiceconserved} $. 
On the other hand $\schemeequil |_{\relaxparletter_{\indiceconserved = 0}}$ does not depend on it.
The \fd scheme from \Cref{prop:ReductionFiniteDifferenceNGeq1Old} bis can be therefore recast, upon rearranging and using well-known properties of the external product $\otimes$, as
\begin{equation}\label{eq:SchemeIndependence}
  \left ( \determinant (\timeshiftoperator \matricial{\identity} - \schememoments_{\indiceconserved})  -\transpose{\canonicalbasisvector_{\indiceconserved}} \adjugate (\timeshiftoperator \matricial{\identity} - \schememoments_{\indiceconserved}) \vectorequilibrium_{\indiceconserved}  \right )\momentletter_{\indiceconserved} = \left ( \adjugate (\timeshiftoperator \matricial{\identity} - \schememoments_{\indiceconserved}) \schememomentsother_{\indiceconserved} \vectorial{\momentletter} \right )_{\indiceconserved} +  \left ( \adjugate (\timeshiftoperator \matricial{\identity} - \schememoments_{\indiceconserved}) \schemeequil|_{\relaxparletter_{\indiceconserved} = 0} \vectorial{\momentletter}^{\atequilibrium} \right )_{\indiceconserved}.
\end{equation}
The left hand side does not depend on $\relaxparletter_{\indicescolumns}$ for $\indicescolumns \in \integerinterval{1}{\consmomentsnumber} \smallsetminus \{ \indiceconserved \}$ by construction of $\schememoments_{\indiceconserved}$ and $\vectorequilibrium_{\indiceconserved}$. On the other hand, the right hand side does not depend on $\relaxparletter_{\indicescolumns}$ for $\indicescolumns \in \integerinterval{1}{\consmomentsnumber} \smallsetminus \{ \indiceconserved \}$, because $(\schememomentsother_{\indiceconserved})_{\cdot, \indicescolumns} + (\schemeequil |_{\relaxparletter_{\indiceconserved} = 0})_{\cdot, \indicescolumns} = (\schememomentsother_{\indiceconserved}|_{\relaxparletter_{\indicescolumns} = 0})_{\cdot, \indicescolumns}$, where we have used \eqref{eq:CollisionPhase} and \eqref{eq:EquilibriumConservedMoments}.
We are left to discuss the possible dependency of \eqref{eq:SchemeIndependence} on $\relaxparletter_{\indiceconserved}$.
For the left hand side, we need the following result concerning the determinant of matrices under rank-one updates, whose proof is analogous to that in \cite{ding2007eigenvalues}.
\begin{lemma}[Matrix determinant]\label{lemma:MatrixDeterminant}
  Let $\genericcommutativering$ be a commutative ring, $\genericmatrix \in \matrixspace{\velocitynumber}{\genericcommutativering}$ and $\vectorial{u}, \vectorial{v} \in \genericcommutativering^{\velocitynumber}$, then $\determinant(\genericmatrix + \vectorial{u} \otimes \vectorial{v}) = \determinant(\genericmatrix) + \transpose{\vectorial{v}}\adjugate(\genericmatrix) \vectorial{u}$.
\end{lemma}
By this Lemma, we deduce that \eqref{eq:SchemeIndependence} now reads
\begin{equation}\label{eq:SchemeIndependence2}
  \determinant (\timeshiftoperator \matricial{\identity} - \left ( \schememoments_{\indiceconserved} +   \vectorequilibrium_{\indiceconserved} \otimes \canonicalbasisvector_{\indiceconserved}\right ))   \momentletter_{\indiceconserved} = \left ( \adjugate (\timeshiftoperator \matricial{\identity} - \schememoments_{\indiceconserved}) \schememomentsother_{\indiceconserved} \vectorial{\momentletter} \right )_{\indiceconserved} +  \left ( \adjugate (\timeshiftoperator \matricial{\identity} - \schememoments_{\indiceconserved}) \schemeequil|_{\relaxparletter_{\indiceconserved} = 0} \vectorial{\momentletter}^{\atequilibrium} \right )_{\indiceconserved}.
\end{equation}
Observe that $\schememoments_{\indiceconserved} + \vectorequilibrium_{\indiceconserved} \otimes \canonicalbasisvector_{\indiceconserved} = \schememoments_{\indiceconserved}|_{\relaxparletter_{\indiceconserved} = 0}$, thus the left hand side of \eqref{eq:SchemeIndependence2} does not not depend on $\relaxparletter_{\indiceconserved}$.
The right hand side of \eqref{eq:SchemeIndependence2} is independent of $\relaxparletter_{\indiceconserved}$ because $\schememomentsother_{\indiceconserved}$ does not depend on it and since the $\indiceconserved$-th row of $ \adjugate (\timeshiftoperator \matricial{\identity} - \schememoments_{\indiceconserved})$ -- the transpose of the cofactor matrix of $\timeshiftoperator \matricial{\identity} - \schememoments_{\indiceconserved}$ -- cannot depend on $\relaxparletter_{\indiceconserved}$, because only the $\indiceconserved$-th column of $\timeshiftoperator \matricial{\identity} - \schememoments_{\indiceconserved}$ depends on $\relaxparletter_{\indiceconserved}$.
This concludes the proof.
  \end{proof}

  We have thus shown that the \fd schemes from \Cref{prop:ReductionFiniteDifferenceNGeq1Old} bis are the same regardless of the choice of relaxation parameters for the conserved moments and so that we are allowed to take them equal to zero or any other value of specific convenience without loss of generality.
  In particular, the choice of taking $\relaxparletter_{\indicemoments} = 0$ for $\indicemoments \in \integerinterval{1}{\consmomentsnumber}$ offers interesting simplifications in the computations to come in \Cref{sec:DetailedProofs}, in a way that shall be clearer by looking at the details.
  Moreover, this choice has the advantage of showing which moments are conserved at a glance.

\section{Main results}\label{sec:MainResults}
Everything is in place to start the standard consistency analysis \cite{strikwerda2004finite, allaire2007numerical} and computation of the modified equations \cite{warming1974modified, carpentier1997derivation} of \fd schemes.
We stress the fact that we aim at studying these features for \eqref{eq:FDSchemeAdjoint} and \eqref{eq:FDSchemeAdjointGeneral} without explicitly writing these schemes down.
We start from the assumptions allowing us to identify each term once developing in formal power series of $\spacestep$, \emph{i.e.} performing Taylor expansions.
Observe that for any time-space numerical scheme at hand, the time step $\timestep$ and the space step $\spacestep$ are linked (scaling) when the grids are refined. Therefore, we decide to take $\spacestep$ as discretization parameter tending to zero. Specific bonds between these two parameters will be given in the following pages.
\begin{assumptions}[General assumptions]\label{ass:ScalingAssumptions} 
    Assume that the change of basis $\momentsmatrix$ and the relaxation matrix $\relaxationmatrix$ are fixed as $\spacestep \to 0$.
\end{assumptions}
We also introduce the spaces of differential operators which shall be obtained by taking the limit $\spacestep \to 0$ as well as other tightly associated concepts.
\begin{definition}[Time-space differential operators]\label{def:TimeSpaceDifferentialOperators}
  We define.
  \begin{itemize}
    \item The commutative ring of time-space differential operators $\timespacedifferentialoperators \definitionequality \polynomialring{\reals}{\partial_{\timevariable}} \tensorProductRing{\reals} \polynomialring{\reals}{\partial_{\spacevariable_1}, \dots, \partial_{\spacevariable_{\spatialdimensionality}}} \isomorphic \polynomialring{\reals}{\partial_{\timevariable}, \partial_{\spacevariable_1}, \dots, \partial_{\spacevariable_{\spatialdimensionality}}}$.
    \item We consider the commutative ring of formal power series \cite{niven1969formal, monforte2013formal} $\taylorseries \definitionequality \formalpolynomialring{\timespacedifferentialoperators}{\spacestep}$.
    \item For any $\generictaylorseries = \sum_{\perturbationorderindices = 0}^{\perturbationorderindices = +\infty} \spacestep^{\perturbationorderindices} \termatorder{\generictaylorseries}{\perturbationorderindices}  \in \taylorseries $, we indicate $\generictaylorseries = \bigO{\spacestep^{{\perturbationorderindices_{\circ}}}}$ for some ${\perturbationorderindices_{\circ}} \in \naturals$ if $\termatorder{\generictaylorseries}{\perturbationorderindices} = 0$ for $\perturbationorderindices \in \integerinterval{0}{\perturbationorderindices_{\circ} - 1}$ and $\termatorder{\generictaylorseries}{\perturbationorderindices_{\circ}} \neq 0$.
    The integer $\perturbationorderindices_{\circ}$ is called ``order'' of the formal power series $\generictaylorseries$, see Chapter 1 in \cite{roman2005umbral}.
    \item Finally, let $\generictimespacediscreteop \in \discretetimespaceoperators$ and $\generictaylorseries \in \taylorseries$, then we indicate $\generictimespacediscreteop \asymtoticequivalence \generictaylorseries$, called ``asymptotic equivalence'' of $\generictimespacediscreteop$ and $\generictaylorseries$, if for any smooth function of the time and space variables $\genericfunction : \reals \times \reals^{\spatialdimensionality} \to \reals$, we have
    \begin{equation*}
      (\generictimespacediscreteop \genericfunction) (\timevariable, \vectorial{\spacevariable}) = \sum_{\perturbationorderindices = 0}^{+\infty} \spacestep^{\perturbationorderindices} (\termatorder{\generictaylorseries}{\perturbationorderindices} \genericfunction) (\timevariable, \vectorial{\spacevariable}), \qquad \forall (\timevariable, \vectorial{\spacevariable})  \in \reals \times \reals^{\spatialdimensionality}, \qquad \text{as} \qquad \spacestep \to 0.
    \end{equation*}
  \end{itemize}
\end{definition}
The previous $\bigO{\cdot}$ notation and the notion of asymptotic equivalence are effortlessly extended to vectors and matrices in an entry-wise fashion. 
It shall be common and harmless not to distinguish between $\matrixspace{\velocitynumber}{\taylorseries}$ and $\formalpolynomialring{(\matrixspace{\velocitynumber}{\timespacedifferentialoperators})}{\spacestep}$.

The momentum-velocity operator matrix $\duboisoperatormatrix \in \matrixspace{\velocitynumber}{\timespacedifferentialoperators}$, introduced by \cite{dubois2019nonlinear} with slightly different notations, is defined as follows.
It is indeed closely linked to the moment-stream matrix $\streammoments \in \matrixspace{\velocitynumber}{\setfinitedifferenceoperators}$ that we have previously introduced.
\begin{definition}[Momentum-velocity operator matrix]
    The momentum-velocity operator matrix made up of first-order differential operators in space is given by
    \begin{equation*}
        \duboisoperatormatrix \definitionequality 
        \momentsmatrix \left ( \sum_{\multiindicemodule{\multiindice} = 1} \diagmatrix \left (\vectorial{\vectorial{\normalizedvelocityletter}_1}^{\multiindice}, \dots, \vectorial{\vectorial{\normalizedvelocityletter}_{\velocitynumber}}^{\multiindice}\right ) \multiindicederivative{\multiindice} \right ) \momentsmatrix^{-1} \in \matrixspace{\velocitynumber}{\timespacedifferentialoperators}, 
    \end{equation*}
    where the multi-index notation is employed.
\end{definition}
This momentum-velocity operator matrix can be partitioned in four blocks with different meanings according to the different nature (conserved or not) of the corresponding moments, as for Equation (8) in \cite{dubois2019nonlinear}.

As previously announced, one needs to specify the used scaling between $\timestep$ and $\spacestep$ in order to perform the consistency analysis and also to recover the modified equations. We start by the acoustic scaling, see for example \cite{dubois2008equivalent, dubois2019nonlinear, yong2016theory}, where $\timestep \sim \spacestep$.
\begin{assumptions}[Acoustic scaling]\label{ass:Acoustic} The assumptions when considering schemes with the acoustic scaling are:
    \begin{enumerate}
        \item $\latticevelocity > 0$ is a fixed real number as $\spacestep \to 0$.
        \item The moments at equilibrium $\vectorial{\momentletter}^{\atequilibrium}$ are fixed as $\spacestep \to 0$.
    \end{enumerate}
\end{assumptions}
For the diffusive scaling, see \cite{zhao2017maxwell, zhang2019lattice}, where $\timestep \sim \spacestep^2$, we have:
\begin{assumptions}[Diffusive scaling]\label{ass:Diffusive} The assumptions when considering schemes with the diffusive scaling are:
    \begin{enumerate}
        \item $\latticevelocity = \diffusivityScaling/\spacestep$ where $\diffusivityScaling > 0$ is a fixed real number as $\spacestep \to 0$.
        \item $\duboisoperatormatrixentry_{\indiceslines \indicescolumns} = 0$ for $\indiceslines, \indicescolumns \in \integerinterval{1}{\consmomentsnumber}$.
        \item $\momentletter_{\indicemoments}^{\atequilibrium} = \spacestep \hat{\momentletter}_{\indicemoments}^{\atequilibrium}$ where $\hat{\momentletter}_{\indicemoments}^{\atequilibrium}$ are fixed, for $\indicemoments \in \indiceMomentsTransport \definitionequality \{ \indicescolumns \in \integerinterval{1}{\velocitynumber} ~ : ~ \duboisoperatormatrix_{\ell \indicescolumns} \neq 0 \text{ for some }\ell \in \integerinterval{1}{\consmomentsnumber}\}$, as $\spacestep \to 0$.
        \item $\momentletter_{\indicemoments}^{\atequilibrium}$  for $\indicemoments \notin \indiceMomentsTransport$ are fixed as $\spacestep \to 0$.
    \end{enumerate}
\end{assumptions}

\begin{remark}
  These assumptions are needed to state the general results to come. However, there are examples in the literature \cite{boghosian2018curious} where they are violated, in particular because the relaxation parameters depend on $\spacestep$. This does not prevent from writing the corresponding Finite Difference schemes \eqref{eq:FDSchemeAdjoint} or \eqref{eq:FDSchemeAdjointGeneral} for the \lbm scheme at hand and then recover their modified equations, but introduces a difficulty to directly obtain the modified equations without explicitly write \eqref{eq:FDSchemeAdjoint} or \eqref{eq:FDSchemeAdjointGeneral} down.
\end{remark}

We are now ready to state and then prove the main results of the present contribution.
The Taylor expansions are applied to the solution of the corresponding \fd schemes given by \Cref{prop:ReductionFiniteDifferenceNGeq1Old} or \Cref{prop:ReductionFiniteDifferenceNGeq1Old} bis, where non-conserved moments have been removed yielding purely macroscopic discrete equations.
\begin{theorem}[Acoustic scaling]\label{thm:AcousticScaling}
  Under \Cref{ass:ScalingAssumptions}, \Cref{ass:Acoustic} and in the limit $\spacestep \to 0$, the corresponding macroscopic \fd schemes given by \Cref{prop:ReductionFiniteDifferenceNGeq1Old} or \Cref{prop:ReductionFiniteDifferenceNGeq1Old} bis are consistent with the target PDEs
  \begin{equation}\label{eq:targetPDEAcoustic}
    \partial_{\timevariable} \tilde{\momentletter}_{\indiceconserved} + \latticevelocity \sum_{\indicescolumns = 1}^{\consmomentsnumber} \duboisoperatormatrixentry_{\indiceconserved \indicescolumns} \tilde{\momentletter}_{\indicescolumns} + \latticevelocity \sum_{\indicescolumns = \consmomentsnumber + 1}^{\velocitynumber} \duboisoperatormatrixentry_{\indiceconserved \indicescolumns} \momentletter_{\indicescolumns}^{\atequilibrium} (\tilde{\momentletter}_1, \dots, \tilde{\momentletter}_{\consmomentsnumber}) = 0,
  \end{equation}
  for $\indiceconserved \in \integerinterval{1}{\consmomentsnumber}$. 
  For smooth solutions $\tilde{\momentletter}_1, \dots, \tilde{\momentletter}_{\consmomentsnumber}:  \reals^{+} \times \reals \to \reals$ of \eqref{eq:targetPDEAcoustic}, the truncation error is given by
  \begin{align*}
    \tau_{\indiceconserved} = &\latticevelocity \spacestep \sum_{\indicescolumns = \consmomentsnumber + 1}^{\velocitynumber} \left (\frac{1}{\relaxparletter_{\indicescolumns}} - \frac{1}{2} \right ) \duboisoperatormatrixentry_{\indiceconserved \indicescolumns} \left ( \sum_{\ell = 1}^{\consmomentsnumber}\duboisoperatormatrixentry_{\indicescolumns \ell} \tilde{\momentletter}_{\ell} + \sum_{\ell = \consmomentsnumber + 1}^{\velocitynumber} \duboisoperatormatrixentry_{\indicescolumns \ell} \momentletter_{\ell}^{\atequilibrium} (\tilde{\momentletter}_1, \dots, \tilde{\momentletter}_{\consmomentsnumber}) - \frac{1}{\latticevelocity}\sum_{\ell = 1}^{\consmomentsnumber} \straightderivative{\momentletter_{\indicescolumns}^{\atequilibrium}}{\momentletter_{\ell}}(\tilde{\momentletter}_1, \dots, \tilde{\momentletter}_{\consmomentsnumber}) \gammaDubois{1}{\ell}(\tilde{\momentletter}_1, \dots, \tilde{\momentletter}_{\consmomentsnumber}) \right ) \\
    &+ \bigO{\spacestep^2},
  \end{align*}
  where $ \gammaDubois{1}{\indiceconserved}(\tilde{\momentletter}_1, \dots, \tilde{\momentletter}_{\consmomentsnumber}) \definitionequality \latticevelocity \sum_{\indicescolumns = 1}^{\indicescolumns = \consmomentsnumber} \duboisoperatormatrixentry_{\indiceconserved \indicescolumns} \tilde{\momentletter}_{\indicescolumns} + \latticevelocity \sum_{\indicescolumns = \consmomentsnumber + 1}^{\indicescolumns = \velocitynumber} \duboisoperatormatrixentry_{\indiceconserved \indicescolumns} \momentletter_{\indicescolumns}^{\atequilibrium} (\tilde{\momentletter}_1, \dots, \tilde{\momentletter}_{\consmomentsnumber})$.
  Therefore, the modified equations up to second order read
  \begin{align*}
    &\partial_{\timevariable} {\momentletter}_{\indiceconserved} + \gammaDubois{1}{\indiceconserved}({\momentletter}_1, \dots, {\momentletter}_{\consmomentsnumber}) \\
    &- \latticevelocity \spacestep \sum_{\indicescolumns = \consmomentsnumber + 1}^{\velocitynumber} \left (\frac{1}{\relaxparletter_{\indicescolumns}} - \frac{1}{2} \right ) \duboisoperatormatrixentry_{\indiceconserved \indicescolumns} \left ( \sum_{\ell = 1}^{\consmomentsnumber}\duboisoperatormatrixentry_{\indicescolumns \ell} {\momentletter}_{\ell} + \sum_{\ell = \consmomentsnumber + 1}^{\velocitynumber} \duboisoperatormatrixentry_{\indicescolumns \ell} \momentletter_{\ell}^{\atequilibrium} ({\momentletter}_1, \dots, {\momentletter}_{\consmomentsnumber}) - \frac{1}{\latticevelocity}\sum_{\ell = 1}^{\consmomentsnumber} \straightderivative{\momentletter_{\indicescolumns}^{\atequilibrium}}{\momentletter_{\ell}}({\momentletter}_1, \dots, {\momentletter}_{\consmomentsnumber}) \gammaDubois{1}{\ell}({\momentletter}_1, \dots, {\momentletter}_{\consmomentsnumber}) \right ) = \bigO{\spacestep^2}.
  \end{align*}
\end{theorem}

The first term in $\gammaDubois{1}{\indiceconserved}$ represents the derivatives of fluxes of the conserved variables, which are necessarily linear, while the second one represents the derivatives of the fluxes given by the equilibria of the non-conserved moments, which can be non-linear.
In the numerical diffusion terms, the so-called H\'enon's parameters \cite{henon1987viscosity} of type $1/\relaxparletter_{\indicescolumns} - 1/2$ appear.
These terms are proportional to $\spacestep$.
This is not surprising, since the only way of having a stable explicit \fd scheme to simulate the heat equation under the acoustic scaling is to consider a diffusion coefficient proportional to $\spacestep$, in order to constrain the speed of propagation of information to remain finite in the limit $\spacestep \to 0$, see for instance Theorem 6.3.1 in \cite{strikwerda2004finite}.

Let us provide two examples for specific \lbm schemes taken from the literature and employed with the acoustic scaling.
\begin{example}[\scheme{1}{3} with one conserved moment - acoustic scaling]\label{ex:D1Q3Acoustic}
  We consider the \scheme{1}{3} scheme presented in \cite{dubois2020notion,bellotti2021fd}, for which $\spatialdimensionality = 1$, $\velocitynumber = 3$, $\normalizedvelocityletter_{1} = 0$, $\normalizedvelocityletter_{2} = 1$, $\normalizedvelocityletter_{3} = -1$ and $\consmomentsnumber = 1$.
  The moment matrix is
  \begin{equation*}
    \momentsmatrix = 
    \begin{pmatrix}
      1   & 1 &  1 \\
      0   & 1 &  -1\\
      -2  & 1 &  1         
    \end{pmatrix}, \qquad \text{hence} \qquad 
    \duboisoperatormatrix = 
    \begin{pmatrix}
      0 & \partial_{\spacevariable_{1}} & 0 \\
      \frac{2}{3} \partial_{\spacevariable_{1}} & 0 & \frac{1}{3}\partial_{\spacevariable_{1}}\\
      0 & \partial_{\spacevariable_{1}} & 0
    \end{pmatrix}.
  \end{equation*}
  \Cref{thm:AcousticScaling} immediately gives the modified equation for the acoustic scaling, which reads
  \begin{equation*}
    \partial_{\timevariable} \momentletter_1 + \latticevelocity \partial_{\spacevariable_1} \momentletter_2^{\atequilibrium}(\momentletter_1) - \latticevelocity \spacestep \Biggl ( \frac{1}{\relaxparletter_2} - \frac{1}{2} \Biggr ) \partial_{\spacevariable_1} \Biggl ( \frac{2}{3} \partial_{\spacevariable_1} \momentletter_1 + \frac{1}{3} \partial_{\spacevariable_1} \momentletter_3^{\atequilibrium}(\momentletter_1) - \straightderivative{\momentletter_2^{\atequilibrium}(\momentletter_1)}{\momentletter_1} \partial_{\spacevariable_1}\momentletter_2^{\atequilibrium}(\momentletter_1) \Biggr ) = \bigO{\spacestep^2}.
  \end{equation*}
\end{example}

\begin{example}[\scheme{2}{9} with three conserved moments - acoustic scaling]\label{ex:D2Q9N3Acoustic}
  We consider the \scheme{2}{9} scheme presented in \cite{lallemand2000theory}, for which $\spatialdimensionality = 2$, $\velocitynumber = 9$, with 
  \begin{equation*}
    \vectorial{\normalizedvelocityletter}_{\indicescolumns} = 
    \begin{cases}
      \transpose{(0, 0)}, \qquad &\text{if} \quad \indicescolumns = 1, \\
      \transpose{(\cos{((\indicescolumns-2)\pi/2}), \sin{((\indicescolumns-2)\pi/2}))},  \qquad &\text{if} \quad \indicescolumns \in \integerinterval{2}{5}, \\
      \sqrt{2}\transpose{(\cos{((2\indicescolumns-3)\pi/4}), \sin{((2\indicescolumns-3)\pi/4}))},  \qquad &\text{if} \quad \indicescolumns \in \integerinterval{6}{9},
    \end{cases}
  \end{equation*}
  and $\consmomentsnumber = 3$.
  The moment matrix is taken to be (we just permute rows compared to \cite{lallemand2000theory} to start with the conserved moments)
  \begin{footnotesize}
  \begin{equation*}
    \momentsmatrix = 
    \begin{pmatrix}
      1 &  1  & 1  & 1  & 1  & 1 & 1 &  1  & 1 \\
      0 &  1  & 0  & -1 & 0  & 1 & -1&  -1 & 1 \\
      0 &  0  & 1  & 0  & -1 & 1 & 1 &  -1 & -1\\
      -4&  -1 & -1 & -1 & -1 & 2 & 2 &  2  & 2 \\
      4 &  -2 & -2 & -2 & -2 & 1 & 1 &  1  & 1 \\
      0 &  -2 & 0  & 2  & 0  & 1 & -1&  -1 & 1 \\
      0 &  0  & -2 & 0  & 2  & 1 & 1 &  -1 & -1\\
      0 &  1  & -1 & 1  & -1 & 0 & 0 &  0  & 0 \\
      0 &  0  & 0  & 0  & 0  & 1 & -1&  1  & -1
    \end{pmatrix}, \quad \text{hence} \quad
    \duboisoperatormatrix = 
    \begin{pmatrix}
      0 & \partial_{\spacevariable_1} & \partial_{\spacevariable_2} & 0 & 0 & 0 &0 & 0 &0 \\
      \frac{2}{3} \partial_{\spacevariable_1} & 0 & 0 & \frac{1}{6} \partial_{\spacevariable_1} & 0 & 0 & 0 & \frac{1}{2}\partial_{\spacevariable_1} & \partial_{\spacevariable_2} \\
      \frac{2}{3} \partial_{\spacevariable_2} & 0 & 0 & \frac{1}{6} \partial_{\spacevariable_2} & 0 & 0 & 0 & -\frac{1}{2}\partial_{\spacevariable_2} & \partial_{\spacevariable_1} \\
      0 &  \partial_{\spacevariable_1} &  \partial_{\spacevariable_2} & 0 & 0 &  \partial_{\spacevariable_1} &  \partial_{\spacevariable_2} & 0 & 0 \\
      0 & 0 & 0 &0&0& \partial_{\spacevariable_1} & \partial_{\spacevariable_2} & 0 & 0 \\
      0 & 0 & 0 & \frac{1}{3}  \partial_{\spacevariable_1} & \frac{1}{3}  \partial_{\spacevariable_1}  & 0 & 0 & - \partial_{\spacevariable_1} &  \partial_{\spacevariable_2} \\
      0 & 0 & 0 & \frac{1}{3}  \partial_{\spacevariable_2} & \frac{1}{3}  \partial_{\spacevariable_2}  & 0 & 0 &  \partial_{\spacevariable_2} &  \partial_{\spacevariable_1} \\
      0 & \frac{1}{3} \partial_{\spacevariable_1} & -\frac{1}{3} \partial_{\spacevariable_2} & 0 & 0 & -\frac{1}{3} \partial_{\spacevariable_1} & \frac{1}{3} \partial_{\spacevariable_2} & 0 & 0\\
      0 & \frac{2}{3} \partial_{\spacevariable_2} & \frac{2}{3} \partial_{\spacevariable_1} & 0 & 0 & \frac{1}{3} \partial_{\spacevariable_2} & \frac{1}{3} \partial_{\spacevariable_1} & 0 & 0
    \end{pmatrix}.
  \end{equation*}
\end{footnotesize}
The equilibria defining the modified equations under acoustic scaling at second-order are taken as in \cite{lallemand2000theory, dubois2019nonlinear}, that is
\begin{align*}
  \momentletter_4^{\atequilibrium} &= -2\momentletter_1 + 3(\momentletter_2^2 + \momentletter_3^2)/\momentletter_1, \qquad \momentletter_6^{\atequilibrium} = -\momentletter_2 + 3\momentletter_2(\momentletter_2^2 + \momentletter_3^2)/\momentletter_1^2, \qquad \momentletter_7^{\atequilibrium} = -\momentletter_3 + 3\momentletter_3(\momentletter_2^2 + \momentletter_3^2)/\momentletter_1^2, \\
  \momentletter_8^{\atequilibrium} &= (\momentletter_2^2 - \momentletter_3^2)/\momentletter_1, \qquad \momentletter_9^{\atequilibrium} = \momentletter_2 \momentletter_3/\momentletter_1.
\end{align*}
Is is well-known \cite{fevrier2014extension, dubois2019nonlinear} that the $\bigO{\spacestep}$ terms for the second and third modified equations contain spurious third-order contributions in $\momentletter_2, \momentletter_{3}$. We shall neglect these terms considering that they are small (low-speed flow).
Furthermore, we consider that $\momentletter_1$ varies slowly as far as the $\bigO{\spacestep}$ term is concerned, thus we neglect its derivatives.
Moreover, we take $\relaxparletter_9 = \relaxparletter_8$, see \cite{lallemand2000theory,dubois2019nonlinear}.
Under these assumptions, the modified equations from \Cref{thm:AcousticScaling} read
\begin{equation*}
  \partial_{\timevariable} \momentletter_1 +  \partial_{\spacevariable_1} \overline{\momentletter}_2 + \partial_{\spacevariable_2} \overline{\momentletter}_3  = \bigO{\spacestep^2},
\end{equation*}
\begin{align*}
  \partial_{\timevariable} \overline{\momentletter}_2 &+ \partial_{\spacevariable_1} \Biggl ( \frac{\overline{\momentletter}_2^2}{\momentletter_1} + \frac{\latticevelocity^2}{3}\momentletter_1\Biggr ) + \partial_{\spacevariable_2} \Biggl ( \frac{\overline{\momentletter}_2\overline{\momentletter}_3}{\momentletter_1} \Biggr ) \\
  &- \frac{\latticevelocity}{3} \spacestep \Biggl ( \partial_{\spacevariable_1} \Biggl ( 2 \Biggl ( \frac{1}{\relaxparletter_8} - \frac{1}{2} \Biggr ) \partial_{\spacevariable_1} \overline{\momentletter}_2 + \Biggl ( \frac{1}{\relaxparletter_4} - \frac{1}{\relaxparletter_8} \Biggr ) (\partial_{\spacevariable_1} \overline{\momentletter}_2 + \partial_{\spacevariable_2} \overline{\momentletter}_3) \Biggr ) + \partial_{\spacevariable_2} \Biggl ( \Biggl ( \frac{1}{\relaxparletter_8} - \frac{1}{2} \Biggr )  (\partial_{\spacevariable_2} \overline{\momentletter}_2 + \partial_{\spacevariable_1} \overline{\momentletter}_3)  \Biggr ) \Biggr ) = \bigO{\spacestep^2},
\end{align*}
\begin{align*}
  \partial_{\timevariable} \overline{\momentletter}_3 &+ \partial_{\spacevariable_1} \Biggl ( \frac{\overline{\momentletter}_2\overline{\momentletter}_3}{\momentletter_1} \Biggr ) + \partial_{\spacevariable_2} \Biggl ( \frac{\overline{\momentletter}_3^2}{\momentletter_1} + \frac{\latticevelocity^2}{3}\momentletter_1\Biggr )  \\
  &- \frac{\latticevelocity}{3} \spacestep \Biggl ( \partial_{\spacevariable_1} \Biggl ( \Biggl (  \frac{1}{\relaxparletter_8} - \frac{1}{2} \Biggr )  (\partial_{\spacevariable_2} \overline{\momentletter}_2 + \partial_{\spacevariable_1} \overline{\momentletter}_3)  \Biggr ) + \partial_{\spacevariable_2} \Biggl ( 2 \Biggl ( \frac{1}{\relaxparletter_8} - \frac{1}{2} \Biggr ) \partial_{\spacevariable_2} \overline{\momentletter}_3 + \Biggl ( \frac{1}{\relaxparletter_4} - \frac{1}{\relaxparletter_8} \Biggr ) (\partial_{\spacevariable_1} \overline{\momentletter}_2 + \partial_{\spacevariable_2} \overline{\momentletter}_3) \Biggr )  \Biggr ) = \bigO{\spacestep^2},
\end{align*}
were we have used $\overline{\momentletter}_2 \definitionequality \latticevelocity \momentletter_2$ and $\overline{\momentletter}_3 \definitionequality \latticevelocity \momentletter_3$. 
The first equation enforces the conservation of the density $\momentletter_1$ in the Euler system, discretized with a second-order scheme. The momentum along the first axis (respectively, second) is $\overline{\momentletter}_2$ (respectively, $\overline{\momentletter}_3$).
The second equation represents -- at leading order -- the conservation of momentum along the first axis in the Euler system. The pressure law is linear and prescribes that the pressure is equal to $\latticevelocity^2 \momentletter_1/3$, hence the speed of the sound is $\latticevelocity/\sqrt{3}$. The numerical diffusion at order $\bigO{\spacestep}$ is what makes up the terms that are usually recognized (except for the previously described pressure) as the stress tensor from the Navier-Stokes system.
Recalling that we have assumed slow variations of $\momentletter_1$ (weakly compressible flow), we have a first bulk viscosity (also known as shear or dynamic viscosity) which equals $\mu = \latticevelocity \spacestep (1/\relaxparletter_8 - 1/2)\momentletter_1 / 3$ (not linked $\diffusivityScaling$ in \Cref{ass:Diffusive}) and a second bulk viscosity (also known as volume viscosity) given by $\kappa = \latticevelocity \spacestep (3(1/\relaxparletter_4 - 1/2) - (1/\relaxparletter_8 - 1/2)) \momentletter_1 / 9$.
Hence, for this kind of system, the viscosity is modeled using numerical diffusion and is proportional to $\spacestep$, thus vanishing when going towards convergence. The same remarks hold for the last equation.
\end{example}

\begin{theorem}[Diffusive scaling]\label{thm:DiffusiveScaling}
  Under \Cref{ass:ScalingAssumptions}, \Cref{ass:Diffusive} and in the limit $\spacestep \to 0$, the corresponding macroscopic \fd schemes given by \Cref{prop:ReductionFiniteDifferenceNGeq1Old} or \Cref{prop:ReductionFiniteDifferenceNGeq1Old} bis are consistent with the target PDEs
  \begin{align}
    \partial_{\timevariable} \tilde{\momentletter}_{\indiceconserved} &+ \diffusivityScaling \sum_{\substack{\indicescolumns = \consmomentsnumber + 1 \\ \indicescolumns \in \indiceMomentsTransport}}^{\velocitynumber} \duboisoperatormatrixentry_{\indiceconserved \indicescolumns} \hat{\momentletter}_{\indicescolumns}^{\atequilibrium} (\tilde{\momentletter}_1, \dots, \tilde{\momentletter}_{\consmomentsnumber}) 
    - \diffusivityScaling \sum_{\indicescolumns = \consmomentsnumber + 1}^{\velocitynumber} \Biggl ( \frac{1}{\relaxparletter_{\indicescolumns}} - \frac{1}{2}\Biggr ) \duboisoperatormatrixentry_{\indiceconserved \indicescolumns}  \Biggl ( \sum_{\ell = 1}^{\consmomentsnumber} \duboisoperatormatrixentry_{\indicescolumns \ell} \tilde{\momentletter}_{\ell} + \sum_{\substack{\ell = \consmomentsnumber + 1 \\ \ell \notin \indiceMomentsTransport}}^{\velocitynumber}  \duboisoperatormatrixentry_{\indicescolumns \ell} {\momentletter}_{\ell}^{\atequilibrium} (\tilde{\momentletter}_1, \dots, \tilde{\momentletter}_{\consmomentsnumber}) \Biggr )= 0, \label{eq:targetPDEParabolic} 
  \end{align}
  for $\indiceconserved \in \integerinterval{1}{\consmomentsnumber}$. 
  For smooth solutions $\tilde{\momentletter}_1, \dots, \tilde{\momentletter}_{\consmomentsnumber}:  \reals^{+} \times \reals \to \reals$ of \eqref{eq:targetPDEParabolic}, the truncation error is given by $\tau_{\indiceconserved} = \bigO{\spacestep}$.
\end{theorem}
We can \emph{a posteriori} explain the meaning of some \Cref{ass:Diffusive} which were less clear before stating \Cref{thm:DiffusiveScaling}. The second assumption avoids to deal with terms which shall naturally appear at order $\bigO{\spacestep}$ but which, since pertaining to the conserved moments, cannot be transformed into terms $\bigO{\spacestep^2}$. Quite the opposite, the third assumption allows to rise to $\bigO{\spacestep^2}$ those terms which contributed to the leading order in \Cref{thm:AcousticScaling}. This is achieved by a rescaling of the equilibria using $\hat{\momentletter}^{\atequilibrium}$.
We therefore see that \lbm schemes can be used to simulate non-linear transport/diffusion equations when using a diffusive scaling.

We also give two examples for specific \lbm schemes considered under diffusive scaling.
\begin{example}[\scheme{1}{3} with one conserved moment - diffusive scaling]
  We come back to the setting of \Cref{ex:D1Q3Acoustic} except that we consider a diffusive scaling.
  Thus we have to take $\momentletter_{2}^{\atequilibrium}(\momentletter_1) = \spacestep \hat{\momentletter}_{2}^{\atequilibrium}(\momentletter_1)$ to comply with \Cref{ass:Diffusive}. This  yields the modified equation
  \begin{equation*}
    \partial_{\timevariable} \momentletter_1 + \diffusivityScaling \partial_{\spacevariable_1} \hat{\momentletter}_2^{\atequilibrium}(\momentletter_1) - \diffusivityScaling \Biggl ( \frac{1}{\relaxparletter_2} - \frac{1}{2} \Biggr ) \partial_{\spacevariable_1} \Biggl ( \frac{2}{3} \partial_{\spacevariable_1} \momentletter_1 + \frac{1}{3} \partial_{\spacevariable_1} \momentletter_3^{\atequilibrium}(\momentletter_1)  \Biggr ) = \bigO{\spacestep}.
  \end{equation*}
  The scheme allows to simulate non-linear transport phenomena using $\hat{\momentletter}_2^{\atequilibrium}$ as well as linear and non-linear diffusion using $\momentletter_3^{\atequilibrium}$.
\end{example}

\begin{example}[\scheme{2}{9} with one conserved moment - diffusive scaling]
  We consider the same scheme as \Cref{ex:D2Q9N3Acoustic} except for the fact that only one conserved moment $\consmomentsnumber = 1$ is present and that the equilibria are general, with $\momentletter_{2}^{\atequilibrium} (\momentletter_1) = \spacestep \hat{\momentletter}_{2}^{\atequilibrium} (\momentletter_1)$ and $\momentletter_{3}^{\atequilibrium} (\momentletter_1) = \spacestep \hat{\momentletter}_{3}^{\atequilibrium}(\momentletter_1) $, to fulfill \Cref{ass:Diffusive}.
  This is the setting introduced in \cite{zhang2019lattice}. The modified equation reads
  \begin{align*}
    \partial_{\timevariable} \momentletter_1 &+ \diffusivityScaling \partial_{\spacevariable_1} \hat{\momentletter}_2^{\atequilibrium}(\momentletter_1) + \diffusivityScaling \partial_{\spacevariable_2} \hat{\momentletter}_3^{\atequilibrium}(\momentletter_1) - \frac{2\diffusivityScaling}{3} \Biggl ( \frac{1}{\relaxparletter_2} - \frac{1}{2}\Biggr ) \partial_{\spacevariable_1 \spacevariable_1} \momentletter_1 - \frac{2\diffusivityScaling}{3} \Biggl ( \frac{1}{\relaxparletter_3} - \frac{1}{2}\Biggr ) \partial_{\spacevariable_2 \spacevariable_2} \momentletter_1 \\
    &- \frac{\diffusivityScaling}{6} \Biggl ( \Biggl ( \frac{1}{\relaxparletter_2} - \frac{1}{2}\Biggr ) \partial_{\spacevariable_1 \spacevariable_1} + \Biggl ( \frac{1}{\relaxparletter_3} - \frac{1}{2}\Biggr ) \partial_{\spacevariable_2 \spacevariable_2} \Biggr) \momentletter_4^{\atequilibrium}(\momentletter_1) - \frac{\diffusivityScaling}{2} \Biggl ( \Biggl ( \frac{1}{\relaxparletter_2} - \frac{1}{2}\Biggr ) \partial_{\spacevariable_1 \spacevariable_1} - \Biggl ( \frac{1}{\relaxparletter_3} - \frac{1}{2}\Biggr ) \partial_{\spacevariable_2 \spacevariable_2} \Biggr) \momentletter_8^{\atequilibrium}(\momentletter_1) \\
    &- \diffusivityScaling \Biggl( \frac{1}{\relaxparletter_2} + \frac{1}{\relaxparletter_3} - 1\Biggr) \partial_{\spacevariable_1 \spacevariable_2} \momentletter_9^{\atequilibrium}(\momentletter_1) = \bigO{\spacestep}.
  \end{align*}
  Therefore, the scheme allows to simulate non-linear transport phenomena using $\hat{\momentletter}_2^{\atequilibrium}$ and $\hat{\momentletter}_3^{\atequilibrium}$ as well as linear and non-linear diffusion with crossed terms \emph{via} $\momentletter_4^{\atequilibrium}$, $\momentletter_8^{\atequilibrium}$ and $\momentletter_9^{\atequilibrium}$.
\end{example}

In this contribution, we have deliberately neglected the behavior of the schemes close to the initial time $\timevariable = 0$.
It is dictated by the choice of initial datum for the non-conserved moments, which is not unique for \lbm schemes since $\velocitynumber > \consmomentsnumber$ but one only knows the $\consmomentsnumber$ conserved moments at $\timevariable = 0$, being the initial datum of the target PDEs to be solved.
The interested reader can consult \cite{van2009smooth, rheinlander2007analysis} for more information on this topic.

Let us sketch the main ideas of the proofs of \Cref{thm:AcousticScaling} and \Cref{thm:DiffusiveScaling}:
\begin{itemize}
  \item The result of \Cref{prop:ReductionFiniteDifferenceNGeq1Old} has allowed to eliminate the non-conserved moments from the discrete scheme, thus has completed the step represented by a vertical arrow in \Cref{fig:PlanOfTheWork}. Contrarily to the existing approaches, we do not need (and we cannot, see \Cref{rem:FromKineticToMacroscopic}) to estimate the Taylor expansions of the non-conserved moments.
  \item We benefit from the clever formulation from \Cref{prop:ReductionFiniteDifferenceNGeq1Old} bis instead of that of \Cref{prop:ReductionFiniteDifferenceNGeq1Old}. Indeed, considering $\asymptotictimeshiftoperator \matricial{\identity} - \asymptoticschememoments_{\indiceconserved} \asymtoticequivalence \timeshiftoperator \matricial{\identity} - \schememoments_{\indiceconserved}$, we are allowed to write, for every $\indiceconserved \in \integerinterval{1}{\consmomentsnumber}$
  \begin{equation*}
    \determinant (\asymptotictimeshiftoperator \matricial{\identity} - \asymptoticschememoments_{\indiceconserved}) \momentletter_{\indiceconserved} =  \left ( \adjugate (\asymptotictimeshiftoperator \matricial{\identity} - \asymptoticschememoments_{\indiceconserved}) \asymptoticschememomentsother_{\indiceconserved} \vectorial{\momentletter} \right )_{\indiceconserved} +  \left ( \adjugate (\asymptotictimeshiftoperator \matricial{\identity} - \asymptoticschememoments_{\indiceconserved}) \asymptoticschemeequil \vectorial{\momentletter}^{\atequilibrium} \right )_{\indiceconserved},
  \end{equation*}
  obtained applying the scheme to smooth functions ${\momentletter}_1, \dots, {\momentletter}_{\consmomentsnumber}:  \reals^{+} \times \reals \to \reals$ and by replacing matrices with entries in the ring $\discretetimespaceoperators$ of discrete operators by their asymptotic equivalents in the ring $\taylorseries$. Here, for example, $ \determinant (\asymptotictimeshiftoperator \matricial{\identity} - \asymptoticschememoments_{\indiceconserved}) \in \taylorseries$, and the expression perfectly makes sense because the determinant and the adjugate are well-defined polynomial functions of any square matrix on a commutative ring, like $\taylorseries$.
  Since the determinant and the adjugate are non-linear functions and thus mix different orders in the expansion $\asymptotictimeshiftoperator \matricial{\identity} - \asymptoticschememoments_{\indiceconserved}$, if we want to recover a closed-form result at a given order of accuracy, we are compelled to utilize the Taylor expansions of the determinant and the adjugate. However, these expansions are well-known and can be computed at any order of accuracy.

  Quite the opposite, if we want to exploit the formulation of \cite{bellotti2021fd} stated in \Cref{prop:ReductionFiniteDifferenceNGeq1Old}, we should characterize the asymptotic equivalents of any coefficient of the characteristic polynomial of $\schememoments_{\indiceconserved}$ and then combine them with the asymptotic equivalents of the time shifts $\timeshiftoperator$ alone and the terms on the right hand side of \eqref{eq:FDSchemeOldPaperNGeq1}. Though this is actually feasible and we firstly did it, the computations are extremely involved\footnote{Probably, a deeper mastery of the elementary symmetric polynomials, the Newton's identities, the Bell polynomials and the Feddeev-Leverrier algorithm could simplify many reasonings.} and very hard to generalize above second-order.

  This justifies the use of the formulation from \Cref{prop:ReductionFiniteDifferenceNGeq1Old} bis to achieve the step denoted by an horizontal arrow in \Cref{fig:PlanOfTheWork}.
\end{itemize}

\section{Detailed proofs}\label{sec:DetailedProofs}

The vast majority of rest of this work is devoted to the detailed proof of \Cref{thm:AcousticScaling} for the scalar case $\consmomentsnumber = 1$.
This choice has been adopted to keep the presentation and the involved notations as simple as possible.
The idea behind the generalization to $\consmomentsnumber > 1$ is eventually given in \Cref{sec:ExtansionSeveralConservedMoments} and is straightforward except for the more involved notations.
The proof of \Cref{thm:DiffusiveScaling} follows exactly the same path of \Cref{thm:AcousticScaling} and is therefore omitted.

Let us start by finding, for each shift operator from \Cref{def:ShiftandFDOperators}, its asymptotically equivalent formal power series in $\spacestep$, see for instance \cite{yong2016theory, dubois2019nonlinear}.
This is formalized by the following Lemma.
\begin{lemma}[Series expansion of a shift operator in space]\label{lemma:ExpansionShift}
    Let $\vectorial{z} \in \relatives^{\spatialdimensionality}$, then the associated shift operator in space $\shiftoperator{\vectorial{z}}  \in \discretetimespaceoperators$ is asymptotically equivalent, in the limit of $\spacestep \to 0$, to the formal power series of differential operators of the form
    \begin{equation*}
        \shiftoperator{\vectorial{z}} \asymtoticequivalence \sum_{\multiindicemodule{\multiindice} \geq 0} \frac{(-\spacestep)^{\multiindicemodule{\multiindice}} \vectorial{z}^{\multiindice}}{\multiindice!} \multiindicederivative{\multiindice}  \in \taylorseries.
    \end{equation*}
  \end{lemma}
\begin{proof}
    Let $\genericfunction : \reals^{\spatialdimensionality} \to \reals$ be a smooth function of the spatial variable. Then performing a Taylor expansion for $\spacestep \to 0$ yields
    \begin{equation*}
        (\shiftoperator{\vectorial{z}} \genericfunction)(\vectorial{\spacevariable}) =  \genericfunction (\vectorial{\spacevariable} - \vectorial{z} \spacestep) = \sum_{\multiindicemodule{\multiindice} \geq 0} \frac{(-\spacestep)^{\multiindicemodule{\multiindice}} \vectorial{z}^{\multiindice}}{\multiindice!} \multiindicederivative{\multiindice} \genericfunction(\vectorial{\spacevariable}), \qquad \vectorial{\spacevariable} \in \reals^{\spatialdimensionality}.
    \end{equation*}
\end{proof}
The extension of \Cref{lemma:ExpansionShift} to any \fd operator in $\setfinitedifferenceoperators$ according to \Cref{def:ShiftandFDOperators} is done by linearity. 
With this in mind, recalling the definition of $\streammoments \in \matrixspace{\velocitynumber}{\setfinitedifferenceoperators}$, the moments-stream matrix and using \Cref{ass:ScalingAssumptions}, we have that
\begin{equation}\label{eq:ExpansionsStreamMatrix} 
    \streammoments \definitionequality \momentsmatrix \diagmatrix(\shiftoperator{\normalizedvelocityletter_1}, \dots, \shiftoperator{\normalizedvelocityletter_{\velocitynumber}}) \momentsmatrix^{-1} \asymtoticequivalence \momentsmatrix \Biggl ( \sum_{\multiindicemodule{\multiindice} \geq 0} \frac{(-\spacestep)^{\multiindicemodule{\multiindice}}}{\multiindice!}\diagmatrix \left (\vectorial{\normalizedvelocityletter}_1^{\multiindice}, \dots, \vectorial{\normalizedvelocityletter}_{\velocitynumber}^{\multiindice}\right ) \multiindicederivative{\multiindice} \Biggr ) \momentsmatrix^{-1} =: \asymptoticstreammoments \in \matrixspace{\velocitynumber}{\taylorseries}.
\end{equation}
Accordingly, we introduce $\asymptoticschememoments \definitionequality \asymptoticstreammoments (\matricial{\identity} - \relaxationmatrix) \in \matrixspace{\velocitynumber}{\taylorseries}$ and $\asymptoticschemeequil \definitionequality \asymptoticstreammoments \relaxationmatrix \in \matrixspace{\velocitynumber}{\taylorseries}$ such that $\asymptoticschememoments \asymtoticequivalence \schememoments$ and $\asymptoticschemeequil \asymtoticequivalence \schemeequil$.
The tight bond between the momentum-velocity operator matrix $\duboisoperatormatrix \in \matrixspace{\velocitynumber}{\timespacedifferentialoperators}$ from \cite{dubois2019nonlinear} and our moments-stream matrix $\streammoments \in \matrixspace{\velocitynumber}{\setfinitedifferenceoperators}$ and its asymptotic equivalent matrix $\asymptoticstreammoments \in \matrixspace{\velocitynumber}{\taylorseries}$ is given by the following Lemma.
\begin{lemma}[Link between $\duboisoperatormatrix$ and $\termatorder{\asymptoticstreammoments}{\perturbationorderindices}$]\label{lemma:LinkBetweenWeandDubois}
  For any order $\perturbationorderindices \in \naturals$, the matrix $\termatorder{\asymptoticstreammoments}{\perturbationorderindices} \in \matrixspace{\velocitynumber}{\timespacedifferentialoperators}$ is linked to $\duboisoperatormatrix \in \matrixspace{\velocitynumber}{\timespacedifferentialoperators}$ by
  \begin{equation*}
      \termatorder{\asymptoticstreammoments}{\perturbationorderindices} = \frac{(-1)^{\perturbationorderindices}}{ \perturbationorderindices !} \duboisoperatormatrix^{\perturbationorderindices}.
  \end{equation*}
  Moreover, using the \Cref{ass:ScalingAssumptions}, we also have
  \begin{equation*}
    \termatorder{\asymptoticschememoments}{\perturbationorderindices} = \frac{(-1)^{\perturbationorderindices}}{ \perturbationorderindices !} \duboisoperatormatrix^{\perturbationorderindices} (\matricial{\identity} - \relaxationmatrix), \qquad \termatorder{\asymptoticschemeequil}{\perturbationorderindices} = \frac{(-1)^{\perturbationorderindices}}{\perturbationorderindices !} \duboisoperatormatrix^{\perturbationorderindices}  \relaxationmatrix.
  \end{equation*}
\end{lemma}
\begin{proof}
  By (21) in \cite{dubois2019nonlinear}, we have that $\streammoments \asymtoticequivalence \asymptoticstreammoments =  \text{exp} (-\spacestep \duboisoperatormatrix  )$. The expansion of the exponential function yields the result.
  Using \Cref{ass:ScalingAssumptions}, one obtains that $\matricial{\identity} - \relaxationmatrix$ and $\relaxationmatrix$ do not perturb the orders of the expansion.
\end{proof}
As far as the time variable is concerned, we can complete by the development of the time shift operator $\timeshiftoperator$ in order to provide the overall expansion of the inverse of the resolvent $\timeshiftoperator \matricial{\identity} - \schememoments \in \matrixspace{\velocitynumber}{\discretetimespaceoperators}$.
\begin{lemma}[Expansion of the inverse of the resolvent]\label{lemma:ExpansionResolvent}
  Under \Cref{ass:ScalingAssumptions}, \Cref{ass:Acoustic} and in the limit of $\spacestep \to 0$, the inverse of the resolvent $\timeshiftoperator \matricial{\identity} - \schememoments  \in \matrixspace{\velocitynumber}{\discretetimespaceoperators}$ is asymptotically equivalent to $ \asymptotictimeshiftoperator \matricial{\identity} - \asymptoticschememoments \in \matrixspace{\velocitynumber}{\taylorseries}$, where
  \begin{equation}\label{eq:ExpansionResolvent}
    \asymptotictimeshiftoperator \matricial{\identity} - \asymptoticschememoments = \sum_{\perturbationorderindices = 0}^{+\infty} \frac{\spacestep^{\perturbationorderindices}}{\perturbationorderindices ! } \left ( \frac{1}{\latticevelocity^{\perturbationorderindices} }\partial_t^{\perturbationorderindices} \matricial{\identity} - (-1)^{\perturbationorderindices} \duboisoperatormatrix^{\perturbationorderindices} (\matricial{\identity} - \relaxationmatrix)\right ) = \relaxationmatrix + \spacestep \left ( \frac{1}{\latticevelocity} \partial_t \matricial{\identity} + \duboisoperatormatrix (\matricial{\identity} - \relaxationmatrix)\right ) + \frac{\spacestep^2}{2} \left ( \frac{1}{\latticevelocity^2}\partial_{tt} \matricial{\identity} - \duboisoperatormatrix^2 (\matricial{\identity} - \relaxationmatrix)\right ) + \bigO{\spacestep^3}.
  \end{equation}
\end{lemma}
\begin{proof}
  The standard Taylor expansion of $\timeshiftoperator$, using the assumption on the acoustic scaling, gives the claim.
\end{proof}

The consistency analysis of the \fd schemes from \Cref{prop:ReductionFiniteDifferenceNGeq1Old} bis could be carried on infinite formal power series of differential operators $\taylorseries$ on the formulation
  \begin{equation}\label{eq:DevelopedEquation}
    \determinant (\asymptotictimeshiftoperator \matricial{\identity} - \asymptoticschememoments_{\indiceconserved}) \momentletter_{\indiceconserved} =  \left ( \adjugate (\asymptotictimeshiftoperator \matricial{\identity} - \asymptoticschememoments_{\indiceconserved}) \asymptoticschememomentsother_{\indiceconserved} \vectorial{\momentletter} \right )_{\indiceconserved} +  \left ( \adjugate (\asymptotictimeshiftoperator \matricial{\identity} - \asymptoticschememoments_{\indiceconserved}) \asymptoticschemeequil \vectorial{\momentletter}^{\atequilibrium} \right )_{\indiceconserved},
  \end{equation}
  for each $\indiceconserved \in \integerinterval{1}{\consmomentsnumber}$, because the determinant and the adjugate perfectly make sense for any square matrix on a commutative ring, like $\taylorseries$.
  However, in order to prove \Cref{thm:AcousticScaling}, where formal power series are truncated at a certain order, we shall need \eqref{eq:ExpansionResolvent} from \Cref{lemma:ExpansionResolvent} as well as the Taylor expansions of the determinant and the adjugate matrix around a given matrix.
  Indeed, these are non-linear functions and thus mix different orders in the expansions $\asymptotictimeshiftoperator \matricial{\identity} - \asymptoticschememoments \in \matrixspace{\velocitynumber}{\taylorseries}$.
  Since the product of the relaxation parameters for the non-conserved moments is a quantity which shall frequently appear in the computations to come, we fix a special notation for it, namely setting $\productrelaxation \definitionequality \prod_{\indicemoments = 2}^{\indicemoments = \velocitynumber} \relaxparletter_{\indicemoments} \neq 0$.

  \subsection{Determinant}
  We start by studying the expansion of the determinant up to second-order in the perturbation.
  For this, we need to characterize its derivatives.
  The expansion can be carried at higher order by employing the very same strategy.

  \begin{lemma}[Derivatives and expansion of the determinant function]\label{lemma:DerivativesDeterminant}
    Let $\genericmatrix \in \lineargroup{\velocitynumber} {\genericcommutativering}$ and $\genericmatrixtwo, \genericmatrixthree \in \matrixspace{\velocitynumber} {\genericcommutativering}$, where $\genericcommutativering$ is a commutative ring.
    Then the determinant function 
    \begin{align*}
      \determinant \colon \matrixspace{\velocitynumber}{\genericcommutativering} &\to \genericcommutativering \\
      \genericmatrix &\mapsto \determinant (\genericmatrix),
    \end{align*}
    has the following derivatives.
    \begin{align}
      \differentialMatrixOperator{\genericmatrix}{\determinant (\genericmatrix)}{\genericmatrixtwo} &= \determinant (\genericmatrix) \trace(\genericmatrix^{-1} \genericmatrixtwo), \label{eq:JacobiFormula}\\
      \secondDifferentialMatrixOperator{\genericmatrix}{\determinant (\genericmatrix)}{\genericmatrixtwo}{\genericmatrixthree} &= \determinant (\genericmatrix) \left ( \trace(\genericmatrix^{-1} \genericmatrixthree) \trace(\genericmatrix^{-1} \genericmatrixtwo) -\trace (\genericmatrix^{-1} \genericmatrixthree \genericmatrix^{-1}\genericmatrixtwo) \right ), \label{eq:DeterminantSecondDerivative}
    \end{align}
    where $\trace (\cdot)$ indicates the trace, \emph{i.e.} the sum of the diagonal entries.
    \eqref{eq:JacobiFormula} is known as Jacobi formula.
    Moreover, the second-order Taylor expansion of the determinant function reads
  \begin{equation*}
    \determinant (\genericmatrix + \genericmatrixtwo) = \determinant (\genericmatrix) + \differentialMatrixOperator{\genericmatrix}{\determinant (\genericmatrix)}{\genericmatrixtwo}  + \frac{1}{2} \secondDifferentialMatrixOperator{\genericmatrix}{\determinant (\genericmatrix)}{\genericmatrixtwo}{\genericmatrixtwo} + \bigO{\lVert \genericmatrixtwo \rVert^{3}}, 
  \end{equation*}
  where the derivatives are given by \eqref{eq:JacobiFormula} and \eqref{eq:DeterminantSecondDerivative}.
  \end{lemma}
  \begin{proof}
    The Jacobi formula \eqref{eq:JacobiFormula} is a standard result, see Chapter 0 in \cite{horn2012matrix} or Chapter 5 in \cite{zwillinger2018crc}.
    Let us prove \eqref{eq:DeterminantSecondDerivative}.
    \begin{align*}
      \secondDifferentialMatrixOperator{\genericmatrix}{\determinant (\genericmatrix)}{\genericmatrixtwo}{\genericmatrixthree} &\definitionequality \differentialMatrixOperator{\genericmatrix}{\differentialMatrixOperator{\genericmatrix}{\determinant (\genericmatrix)}{\genericmatrixtwo}}{\genericmatrixthree} = \differentialMatrixOperator{\genericmatrix}{\determinant (\genericmatrix) \trace(\genericmatrix^{-1} \genericmatrixtwo)}{\genericmatrixthree}, \\
      &= \differentialMatrixOperator{\genericmatrix}{\determinant (\genericmatrix) }{\genericmatrixthree} \trace(\genericmatrix^{-1} \genericmatrixtwo) + \determinant (\genericmatrix) \differentialMatrixOperator{\genericmatrix}{\trace(\genericmatrix^{-1} \genericmatrixtwo)}{\genericmatrixthree}, \\
      &= \determinant (\genericmatrix) \trace(\genericmatrix^{-1} \genericmatrixthree) \trace(\genericmatrix^{-1} \genericmatrixtwo) + \determinant (\genericmatrix) \trace(\differentialMatrixOperator{\genericmatrix}{\genericmatrix^{-1} \genericmatrixtwo}{\genericmatrixthree}), \\
      &= \determinant (\genericmatrix) \trace(\genericmatrix^{-1} \genericmatrixthree) \trace(\genericmatrix^{-1} \genericmatrixtwo) - \determinant (\genericmatrix) \trace(\genericmatrix^{-1} \genericmatrixthree \genericmatrix^{-1} \genericmatrixtwo),
    \end{align*}
    where we have used, in this order, the product rule for derivatives, the Jacobi formula \eqref{eq:JacobiFormula}, the linearity of the trace and the fact that $\differentialMatrixOperator{\genericmatrix}{\genericmatrix^{-1}}{\genericmatrixtwo} = -\genericmatrix^{-1} \genericmatrixtwo \genericmatrix^{-1}$, see Chapter 5 in \cite{zwillinger2018crc}.
  \end{proof}
  \begin{remark}[On the invertibility assumption]\label{rem:Invertibility}
    There exists a form of the Jacobi formula \eqref{eq:JacobiFormula} for general $\genericmatrix \in \matrixspace{\velocitynumber}{\genericcommutativering}$ without assuming invertibility, under the form $\differentialMatrixOperator{\genericmatrix}{\determinant (\genericmatrix)}{\genericmatrixtwo} = \trace (\adjugate (\genericmatrix) \genericmatrixtwo)$. This is equivalent to \eqref{eq:JacobiFormula}, since \eqref{eq:FundamentalRelationAdjugate} holds.
    Nevertheless, we decided to state \Cref{lemma:DerivativesDeterminant} using the invertibility assumption. 
    This is done, as we shall see, without loss of generality by taking advantage of some invertible approximation of real matrices and allows to easily find the formul\ae~for higher order derivatives and expansions \emph{via} basic differential calculus, as illustrated in the previous proof.
  \end{remark}
In the sequel, we shall take $\genericcommutativering = \taylorseries$ and $\genericmatrix = \relaxationmatrix \in \lineargroup{\velocitynumber}{\reals} \subset \lineargroup{\velocitynumber}{\taylorseries}$ and $\genericmatrixtwo = \bigO{\spacestep} \in \matrixspace{\velocitynumber}{\taylorseries}$.
To simplify the computations and relying on the findings of \Cref{sec:InvarianceChoiceRelaxationParamtersConserved}, we can consider $\relaxationmatrix$ singular by having $\relaxparletter_{1} = 0$.
To avoid the difficulties linked with singular matrices, in the spirit of \Cref{rem:Invertibility}, we take advantage of the fact that the derivatives of the determinant (and the determinant itself) around $\genericmatrix$ are smooth (indeed, polynomial) functions of $\genericmatrix$.
Thus, we introduce the non-singular approximation $\relaxationmatrix$ where $\relaxparletter_1 \neq 0$, which is such that $\relaxationmatrix \to \relaxationmatrix|_{\relaxparletter_1 = 0}$ as $\relaxparletter_1  \to 0$ for any matricial topology.

We are now ready to use the expansion given by \Cref{lemma:ExpansionResolvent} into the terms stemming from \Cref{lemma:DerivativesDeterminant} to find the leading order terms of the left hand side of \eqref{eq:FDSchemeAdjoint}, namely of $\determinant (\asymptotictimeshiftoperator \matricial{\identity} - \asymptoticschememoments) \in \taylorseries$.
This is nothing but computing the Taylor series of composite functions (see the Fa\`a di Bruno's formul\ae \cite{johnson2002curious}) or the composition of formal series
\begin{align*}
    \determinant (\asymptotictimeshiftoperator \matricial{\identity} - \asymptoticschememoments) = \determinant (\relaxationmatrix) &+ \spacestep \differentialMatrixOperator{\relaxationmatrix}{\determinant (\relaxationmatrix)}{\termatorderparenthesis{\asymptotictimeshiftoperator \matricial{\identity} - \asymptoticschememoments}{1}} \\
    &+\spacestep^2 \bigl (\differentialMatrixOperator{\relaxationmatrix}{\determinant (\relaxationmatrix)}{\termatorderparenthesis{\asymptotictimeshiftoperator \matricial{\identity} - \asymptoticschememoments}{2}} + \frac{1}{2}\secondDifferentialMatrixOperator{\relaxationmatrix}{\determinant (\relaxationmatrix)}{\termatorderparenthesis{\asymptotictimeshiftoperator \matricial{\identity} - \asymptoticschememoments}{1}}{\termatorderparenthesis{\asymptotictimeshiftoperator \matricial{\identity} - \asymptoticschememoments}{1}} \bigr ) + \bigO{\spacestep^3}.
\end{align*}

\begin{itemize}
  \item One clearly has $\determinant (\relaxationmatrix) = \relaxparletter_1 \productrelaxation$, because the matrix $\relaxationmatrix$ is diagonal.
  Thus, the Taylor expansion of $\determinant (\asymptotictimeshiftoperator \matricial{\identity} - \asymptoticschememoments)$ does not contain zero-order terms if $\relaxparletter_1 = 0$.
  \item Let $\genericmatrix = \relaxationmatrix \in \lineargroup{\velocitynumber}{\reals} \subset \lineargroup{\velocitynumber}{\taylorseries}$ and $\genericmatrixtwo = \spacestep \left ( \frac{1}{\latticevelocity} \partial_t \matricial{\identity} + \duboisoperatormatrix (\matricial{\identity} - \relaxationmatrix)\right ) + \frac{\spacestep^2}{2} \left ( \frac{1}{\latticevelocity^2} \partial_{tt} \matricial{\identity} - \duboisoperatormatrix^2 (\matricial{\identity} - \relaxationmatrix)\right ) + \bigO{\spacestep^3} \in \matrixspace{\velocitynumber}{\taylorseries}$ from \Cref{lemma:ExpansionResolvent}. Using \eqref{eq:JacobiFormula} from \Cref{lemma:DerivativesDeterminant} and performing elementary computations, we have
  \begin{align}
    &\differentialMatrixOperator{\genericmatrix}{\determinant (\genericmatrix)}{\genericmatrixtwo} = \spacestep \productrelaxation \Biggl ( \frac{1}{\latticevelocity}  \partial_t  + (1-\relaxparletter_1) \duboisoperatormatrixentry_{11} + \relaxparletter_1 \sum_{\indiceslines = 2}^{\velocitynumber} \frac{1}{\relaxparletter_{\indiceslines}} \Bigr ( \frac{1}{\latticevelocity }\partial_t + (1 - \relaxparletter_{\indiceslines}) \duboisoperatormatrixentry_{\indiceslines \indiceslines}\Bigr ) \Biggr )  \label{eq:ExpansionDerivativeDeterminant}  \\
    &+\frac{\spacestep^2}{2} \productrelaxation \Biggl ( \frac{1}{\latticevelocity^2} \partial_{tt} - (1-\relaxparletter_1) \duboisoperatormatrixentry_{11}\duboisoperatormatrixentry_{11} - (1-\relaxparletter_1) \sum_{\ell = 2}^{\velocitynumber} \duboisoperatormatrixentry_{1\ell} \duboisoperatormatrixentry_{\ell 1} + \relaxparletter_1 \sum_{\indiceslines = 2}^{\velocitynumber} \frac{1}{\relaxparletter_{\indiceslines}} \Biggl ( \frac{1}{\latticevelocity^2}\partial_{tt} - (1 - \relaxparletter_{\indiceslines}) \sum_{\ell = 1}^{\velocitynumber} \duboisoperatormatrixentry_{\indiceslines \ell}\duboisoperatormatrixentry_{\ell \indiceslines} \Biggr ) \Biggr ) + \bigO{\spacestep^3}.\nonumber
  \end{align}
  We keep this expression without taking the limit in $\relaxparletter_1$, for future use.
  Taking the limit for $\relaxparletter_1 \to 0$ yields the derivative around the singular matrix $\relaxationmatrix|_{\relaxparletter_1 = 0}$ instead of $\relaxationmatrix \in \lineargroup{\velocitynumber}{\reals}$ for $\relaxparletter_1 \neq 0$.
  \begin{equation}\label{eq:FirstDerivativeDeterminantOnResolvent}
    \lim_{\relaxparletter_1 \to 0} \differentialMatrixOperator{\genericmatrix}{\determinant (\genericmatrix)}{\genericmatrixtwo} =  \spacestep \productrelaxation \Bigl ( \frac{1}{\latticevelocity }\partial_t  + \duboisoperatormatrixentry_{11}\Bigr ) + \frac{\spacestep^2}{2} \productrelaxation \Biggl ( \frac{1}{\latticevelocity^2} \partial_{tt} - \duboisoperatormatrixentry_{11}\duboisoperatormatrixentry_{11} - \sum_{\ell = 2}^{\velocitynumber} \duboisoperatormatrixentry_{1\ell} \duboisoperatormatrixentry_{\ell 1} \Biggr ) + \bigO{\spacestep^3}.
  \end{equation}
  This gives all the first-order term and part of the second-order term in the series  $\determinant (\asymptotictimeshiftoperator \matricial{\identity} - \asymptoticschememoments)$.
  \item Let $\genericmatrix = \relaxationmatrix  \in \lineargroup{\velocitynumber}{\reals} \subset \lineargroup{\velocitynumber}{\taylorseries}$ and $\genericmatrixtwo = \spacestep \left ( \frac{1}{\latticevelocity} \partial_t \matricial{\identity} + \duboisoperatormatrix (\matricial{\identity} - \relaxationmatrix)\right ) + \bigO{\spacestep^2} \in \matrixspace{\velocitynumber}{\taylorseries}$ from \Cref{lemma:ExpansionResolvent}. Using \eqref{eq:DeterminantSecondDerivative} from \Cref{lemma:DerivativesDeterminant}, we have, after some algebra
  \begin{align}
    \secondDifferentialMatrixOperator{\genericmatrix}{\determinant (\genericmatrix)}{\genericmatrixtwo}{\genericmatrixtwo} = \spacestep^2  \productrelaxation \Biggl ( &2 \Bigl ( \frac{1}{\latticevelocity}\partial_t + (1-\relaxparletter_1) \duboisoperatormatrixentry_{11} \Bigr ) \sum_{\indiceslines = 2}^{\velocitynumber} \frac{1}{\relaxparletter_{\indiceslines}} \Bigl (\frac{1}{\latticevelocity} \partial_t + (1 - \relaxparletter_{\indiceslines} ) \duboisoperatormatrixentry_{\indiceslines \indiceslines} \Bigr ) + \relaxparletter_{1} \Biggl ( \sum_{\indiceslines = 2}^{\velocitynumber} \frac{1}{\relaxparletter_{\indiceslines}} \Bigl ( \frac{1}{\latticevelocity}\partial_t + (1 - \relaxparletter_{\indiceslines}  ) \duboisoperatormatrixentry_{\indiceslines \indiceslines} \Bigr )\Biggr )^2 \nonumber \\
    -&2  ( 1 - \relaxparletter_1  ) \sum_{\ell = 2}^{\velocitynumber} \left (\frac{1}{\relaxparletter_{\ell}} - 1 \right ) \duboisoperatormatrixentry_{1\ell} \duboisoperatormatrixentry_{\ell 1} - \relaxparletter_{1} \sum_{\indiceslines = 2}^{\velocitynumber} \frac{1}{\relaxparletter_{\indiceslines}^2} \Bigl ( \frac{1}{\latticevelocity}\partial_t + (1 - \relaxparletter_{\indiceslines}) \duboisoperatormatrixentry_{\indiceslines \indiceslines} \Bigr )^2 \nonumber \\
    - &\relaxparletter_1 \sum_{\indiceslines = 2}^{\velocitynumber} \sum_{\substack{\ell = 2 \\ \ell \neq \indiceslines}}^{\velocitynumber} \left (\frac{1}{\relaxparletter_{\indiceslines}} - 1 \right ) \left (\frac{1}{\relaxparletter_{\ell}} - 1 \right ) \duboisoperatormatrixentry_{\indiceslines \ell} \duboisoperatormatrixentry_{ \ell \indiceslines} \Biggr ) + \bigO{\spacestep^3}.  \label{eq:ExpansionSecondDerivativeDeterminant}
  \end{align}
  Once more, we take the limit for $\relaxparletter_1 \to 0$ in order to find the desired result on the remaining second-order terms in the development $\determinant (\asymptotictimeshiftoperator \matricial{\identity} - \asymptoticschememoments)$
  \begin{align}
    \lim_{\relaxparletter_1 \to 0} \secondDifferentialMatrixOperator{\genericmatrix}{\determinant (\genericmatrix)}{\genericmatrixtwo}{\genericmatrixtwo} = 2 \spacestep^2 \productrelaxation \Biggl (& \frac{1}{\latticevelocity^2}\partial_{tt} \sum_{\ell = 2}^{\velocitynumber} \frac{1}{\relaxparletter_{\ell}} + \frac{1}{\latticevelocity}  \duboisoperatormatrixentry_{11} \partial_{t} \sum_{\ell = 2}^{\velocitynumber} \frac{1}{\relaxparletter_{\ell}} + \frac{1}{\latticevelocity}\partial_t \sum_{\indiceslines = 2}^{\velocitynumber} \Biggl ( \frac{1}{\relaxparletter_{\indiceslines}} - 1 \Biggr ) \duboisoperatormatrixentry_{\indiceslines \indiceslines}  \nonumber\\
    &+ \duboisoperatormatrixentry_{11} \sum_{\indiceslines = 2}^{\velocitynumber} \left ( \frac{1}{\relaxparletter_{\indiceslines}} - 1 \right ) \duboisoperatormatrixentry_{\indiceslines \indiceslines}
    - \sum_{\ell = 2}^{\velocitynumber} \Biggl (\frac{1}{\relaxparletter_{\ell}} - 1 \Biggr )\duboisoperatormatrixentry_{1\ell} \duboisoperatormatrixentry_{\ell 1} \Biggr ) + \bigO{\spacestep^3}.   \label{eq:SecondDerivativeDeterminantOnResolvent} 
  \end{align}
\end{itemize}
Putting \eqref{eq:FirstDerivativeDeterminantOnResolvent} and \eqref{eq:SecondDerivativeDeterminantOnResolvent} together in \Cref{lemma:DerivativesDeterminant}, with expansion around $\relaxationmatrix$, allows to write $\determinant (\asymptotictimeshiftoperator \matricial{\identity} - \asymptoticschememoments)$ up to third order.
This is
\begin{align}
  \lim_{\relaxparletter_1 \to 0} \determinant (\asymptotictimeshiftoperator \matricial{\identity} - \asymptoticschememoments) = \spacestep \productrelaxation \Bigl ( \frac{1}{\latticevelocity }\partial_t  + \duboisoperatormatrixentry_{11}\Bigr ) 
  +\spacestep^2 \productrelaxation \Biggl (& \frac{1}{\latticevelocity^2} \Biggl (\frac{1}{2} + \sum_{\ell = 2}^{\velocitynumber} \frac{1}{\relaxparletter_{\ell}} \Biggr ) \partial_{tt}+ \frac{1}{\latticevelocity}  \duboisoperatormatrixentry_{11} \partial_{t} \sum_{\ell = 2}^{\velocitynumber} \frac{1}{\relaxparletter_{\ell}} + \frac{1}{\latticevelocity}\partial_t \sum_{\indiceslines = 2}^{\velocitynumber} \Biggl ( \frac{1}{\relaxparletter_{\indiceslines}} - 1 \Biggr ) \duboisoperatormatrixentry_{\indiceslines \indiceslines}  \label{eq:DeterminantExpansionOnResolvent}\\
    &- \frac{1}{2} \duboisoperatormatrixentry_{11}\duboisoperatormatrixentry_{11} 
    - \sum_{\ell = 2}^{\velocitynumber} \Biggl (\frac{1}{\relaxparletter_{\ell}} - \frac{1}{2} \Biggr )\duboisoperatormatrixentry_{1\ell} \duboisoperatormatrixentry_{\ell 1}  + \duboisoperatormatrixentry_{11} \sum_{\indiceslines = 2}^{\velocitynumber} \left ( \frac{1}{\relaxparletter_{\indiceslines}} - 1 \right ) \duboisoperatormatrixentry_{\indiceslines \indiceslines} \Biggr ) + \bigO{\spacestep^3}. \nonumber
\end{align}

\subsection{Adjugate}
We now switch to the formal power series of the adjugate function of the inverse of the resolvent, in order to deal with the right hand side of the corresponding \fd scheme given by \eqref{eq:FDSchemeAdjoint}.
Let us start by characterizing its derivatives.
\begin{lemma}[Derivatives and expansion of the adjugate function]\label{lemma:DerivativesAjugate}
  Let $\genericmatrix \in \lineargroup{\velocitynumber}{\genericcommutativering}$ and $\genericmatrixtwo, \genericmatrixthree \in \matrixspace{\velocitynumber}{\genericcommutativering}$, where $\genericcommutativering$ is a commutative ring.
  Then the adjugate function
  \begin{align*}
    \adjugate \colon \matrixspace{\velocitynumber}{\genericcommutativering} &\to \matrixspace{\velocitynumber}{\genericcommutativering}  \\
    \genericmatrix &\mapsto \adjugate (\genericmatrix),
  \end{align*}
  has the following derivatives.
  \begin{align}
    \differentialMatrixOperator{\genericmatrix}{\adjugate (\genericmatrix)}{\genericmatrixtwo} &= \determinant (\genericmatrix) \left ( \trace(\genericmatrix^{-1} \genericmatrixtwo) \matricial{\identity} - \genericmatrix^{-1}\genericmatrixtwo \right ) \genericmatrix^{-1}, \label{eq:DerivativeAjugate}\\
    \secondDifferentialMatrixOperator{\genericmatrix}{\adjugate (\genericmatrix)}{\genericmatrixtwo}{\genericmatrixthree} = \determinant (\genericmatrix) \Bigl ( &\left ( \trace(\genericmatrix^{-1} \genericmatrixthree) \trace(\genericmatrix^{-1} \genericmatrixtwo) -\trace (\genericmatrix^{-1} \genericmatrixthree \genericmatrix^{-1}\genericmatrixtwo) \right ) \genericmatrix^{-1} \nonumber \\
    &+\genericmatrix^{-1} \left (\genericmatrixthree \genericmatrix^{-1}\genericmatrixtwo + \genericmatrixtwo \genericmatrix^{-1}\genericmatrixthree - \trace(\genericmatrix^{-1} \genericmatrixthree)\genericmatrixtwo - \trace(\genericmatrix^{-1} \genericmatrixtwo)\genericmatrixthree  \right )\genericmatrix^{-1} \Bigr ). \label{eq:secondDerivativeAjugate}
  \end{align}
  Moreover, the second-order Taylor expansion of the adjugate function reads
  \begin{equation*}
    \adjugate (\genericmatrix + \genericmatrixtwo) = \adjugate (\genericmatrix) + \differentialMatrixOperator{\genericmatrix}{\adjugate (\genericmatrix)}{\genericmatrixtwo}  + \frac{1}{2} \secondDifferentialMatrixOperator{\genericmatrix}{\adjugate (\genericmatrix)}{\genericmatrixtwo}{\genericmatrixtwo} + \bigO{\lVert \genericmatrixtwo \rVert^{3}}, 
  \end{equation*}
  where the derivatives are given by \eqref{eq:DerivativeAjugate} and \eqref{eq:secondDerivativeAjugate}.
\end{lemma}
\begin{proof}
  Since \eqref{eq:FundamentalRelationAdjugate} holds and $\genericmatrix$ is invertible, we have that $\adjugate (\genericmatrix) = \determinant(\genericmatrix) \genericmatrix^{-1}$.
  Therefore
  \begin{align*}
    \differentialMatrixOperator{\genericmatrix}{\adjugate (\genericmatrix)}{\genericmatrixtwo} &= \differentialMatrixOperator{\genericmatrix}{\determinant(\genericmatrix) \genericmatrix^{-1}}{\genericmatrixtwo} =  \differentialMatrixOperator{\genericmatrix}{\determinant(\genericmatrix) }{\genericmatrixtwo} \genericmatrix^{-1} + \determinant(\genericmatrix) \differentialMatrixOperator{\genericmatrix}{\genericmatrix^{-1}}{\genericmatrixtwo}, \\
    &= \determinant (\genericmatrix) \trace(\genericmatrix^{-1} \genericmatrixtwo)  \genericmatrix^{-1} - \determinant(\genericmatrix) \genericmatrix^{-1} \genericmatrixtwo \genericmatrix^{-1},
  \end{align*}
  where we have used the rule for the derivative of a product, the Jacobi formula \eqref{eq:JacobiFormula} and the identity $\differentialMatrixOperator{\genericmatrix}{\genericmatrix^{-1}}{\genericmatrixtwo} = -\genericmatrix^{-1} \genericmatrixtwo \genericmatrix^{-1}$.
  For the second derivative, we have 
  \begin{align*}
    \secondDifferentialMatrixOperator{\genericmatrix}{\adjugate (\genericmatrix)}{\genericmatrixtwo}{\genericmatrixthree} &\definitionequality \differentialMatrixOperator{\genericmatrix}{\differentialMatrixOperator{\genericmatrix}{\adjugate (\genericmatrix)}{\genericmatrixtwo}}{\genericmatrixthree} = \differentialMatrixOperator{\genericmatrix}{\determinant (\genericmatrix) \left ( \trace(\genericmatrix^{-1} \genericmatrixtwo) \matricial{\identity} - \genericmatrix^{-1}\genericmatrixtwo \right ) \genericmatrix^{-1}}{\genericmatrixthree}, \\
    &= \differentialMatrixOperator{\genericmatrix}{\determinant (\genericmatrix)}{\genericmatrixthree} \left ( \trace(\genericmatrix^{-1} \genericmatrixtwo) \matricial{\identity} - \genericmatrix^{-1}\genericmatrixtwo \right ) \genericmatrix^{-1} +  \determinant (\genericmatrix) \differentialMatrixOperator{\genericmatrix}{\left ( \trace(\genericmatrix^{-1} \genericmatrixtwo) \matricial{\identity} - \genericmatrix^{-1}\genericmatrixtwo \right ) \genericmatrix^{-1}}{\genericmatrixthree}, \\
    &=\determinant (\genericmatrix) \trace(\genericmatrix^{-1} \genericmatrixthree) \left ( \trace(\genericmatrix^{-1} \genericmatrixtwo) \matricial{\identity} - \genericmatrix^{-1}\genericmatrixtwo \right ) \genericmatrix^{-1} + \determinant (\genericmatrix) \differentialMatrixOperator{\genericmatrix}{\trace(\genericmatrix^{-1} \genericmatrixtwo) \matricial{\identity} - \genericmatrix^{-1}\genericmatrixtwo}{\genericmatrixthree} \genericmatrix^{-1}\\
    &+ \determinant (\genericmatrix) \left ( \trace(\genericmatrix^{-1} \genericmatrixtwo) \matricial{\identity} - \genericmatrix^{-1}\genericmatrixtwo \right ) \differentialMatrixOperator{\genericmatrix}{\genericmatrix^{-1}}{\genericmatrixthree}, \\
    &=\determinant (\genericmatrix) \trace(\genericmatrix^{-1} \genericmatrixthree) \left ( \trace(\genericmatrix^{-1} \genericmatrixtwo) \matricial{\identity} - \genericmatrix^{-1}\genericmatrixtwo \right ) \genericmatrix^{-1} \\
    &+ \determinant (\genericmatrix)  \left ( \trace \left ( \differentialMatrixOperator{\genericmatrix}{\genericmatrix^{-1}}{\genericmatrixthree} \genericmatrixtwo \right ) \matricial{\identity} - \differentialMatrixOperator{\genericmatrix}{\genericmatrix^{-1}}{\genericmatrixthree}\genericmatrixtwo \right ) \genericmatrix^{-1} \\
    &- \determinant (\genericmatrix) \left ( \trace(\genericmatrix^{-1} \genericmatrixtwo) \matricial{\identity} - \genericmatrix^{-1}\genericmatrixtwo \right ) \genericmatrix^{-1} \genericmatrixthree \genericmatrix^{-1}, \\
    &=\determinant (\genericmatrix) \trace(\genericmatrix^{-1} \genericmatrixthree) \left ( \trace(\genericmatrix^{-1} \genericmatrixtwo) \matricial{\identity} - \genericmatrix^{-1}\genericmatrixtwo \right ) \genericmatrix^{-1} - \determinant (\genericmatrix)  \left ( \trace \left ( \genericmatrix^{-1} \genericmatrixthree  \genericmatrix^{-1}  \genericmatrixtwo \right ) \matricial{\identity} - \genericmatrix^{-1} \genericmatrixthree \genericmatrix^{-1} \genericmatrixtwo \right ) \genericmatrix^{-1} \\
    &- \determinant (\genericmatrix) \left ( \trace(\genericmatrix^{-1} \genericmatrixtwo) \matricial{\identity} - \genericmatrix^{-1}\genericmatrixtwo \right ) \genericmatrix^{-1} \genericmatrixthree \genericmatrix^{-1},
  \end{align*}
  where we have used the rule for the derivative of a product, the Jacobi formula \eqref{eq:JacobiFormula}, the linearity of the derivative and the trace and the identity $\differentialMatrixOperator{\genericmatrix}{\genericmatrix^{-1}}{\genericmatrixtwo} = -\genericmatrix^{-1} \genericmatrixtwo \genericmatrix^{-1}$.
  Upon rearrangement, this yields the result.
\end{proof}
\begin{remark}\label{rem:Recycle}
  We observe that, looking at \eqref{eq:DerivativeAjugate} and \eqref{eq:secondDerivativeAjugate} compared to \eqref{eq:JacobiFormula} and \eqref{eq:DeterminantSecondDerivative}, we have that
  \begin{align*}
    \differentialMatrixOperator{\genericmatrix}{\adjugate (\genericmatrix)}{\genericmatrixtwo} &= \differentialMatrixOperator{\genericmatrix}{\determinant (\genericmatrix)}{\genericmatrixtwo} \genericmatrix^{-1} - \determinant (\genericmatrix) \genericmatrix^{-1}\genericmatrixtwo \genericmatrix^{-1}, \\
    \secondDifferentialMatrixOperator{\genericmatrix}{\adjugate (\genericmatrix)}{\genericmatrixtwo}{\genericmatrixthree} &=\secondDifferentialMatrixOperator{\genericmatrix}{\determinant (\genericmatrix)}{\genericmatrixtwo}{\genericmatrixthree} \genericmatrix^{-1} + \determinant (\genericmatrix) \genericmatrix^{-1} \left (\genericmatrixthree \genericmatrix^{-1}\genericmatrixtwo + \genericmatrixtwo \genericmatrix^{-1}\genericmatrixthree - \trace(\genericmatrix^{-1} \genericmatrixthree)\genericmatrixtwo - \trace(\genericmatrix^{-1} \genericmatrixtwo)\genericmatrixthree  \right )\genericmatrix^{-1}.
  \end{align*}
  This implies that we can reuse the computations we did for the determinant in the current treatment of the adjugate, as far as the first terms on the right hand sides are concerned. 
  However, one must be careful that now they are multiplied by $\genericmatrix^{-1}$.
\end{remark}

If we had stopped the developments at first order, we could have used the first-order perturbation theory of the adjugate matrix as provided by Theorem 2.1 from \cite{stewart1998adjugate}.
However, to the best of our knowledge, no second-order perturbation theory for this matrix is available in the literature, thus we have been compelled to independently develop it using differential calculus.
\Cref{lemma:DerivativesAjugate} is thus a generalization of the results from \cite{stewart1998adjugate} and can therefore be used -- beyond the application presented in this contribution -- by researchers needing a second-order perturbation theory for the adjugate matrix.

Since we are ultimately interested, as one can notice from \eqref{eq:FDSchemeAdjoint}, in multiplying the formal power series $\adjugate (\asymptotictimeshiftoperator \matricial{\identity} - \asymptoticschememoments) \in \matrixspace{\velocitynumber}{\taylorseries}$ by $\asymptoticschemeequil \in \matrixspace{\velocitynumber}{\taylorseries}$ in a Cauchy-like fashion (the standard product of formal power series) and select the first row, see \Cref{prop:ReductionFiniteDifferenceGeneralOld} bis, we perform the computations only for the first row of $\adjugate (\asymptotictimeshiftoperator \matricial{\identity} - \asymptoticschememoments)$.
\begin{itemize}
  \item Using the definition of the adjugate matrix in combination with the Laplace formula or using the explicit formula for the adjugate of an upper triangular matrix, see \cite{horn2012matrix}, we have 
  \begin{equation*}
    \adjugate(\relaxationmatrix) = \productrelaxation \diagmatrix \left (1, \frac{\relaxparletter_{1}}{\relaxparletter_2}, \dots, \frac{\relaxparletter_{1}}{\relaxparletter_{\velocitynumber}} \right ), \qquad \text{thus} \qquad \lim_{\relaxparletter_1 \to 0}  \adjugate(\relaxationmatrix) = \productrelaxation \canonicalbasisvector_1 \otimes \canonicalbasisvector_1.
  \end{equation*}
  Hence, contrarily to the determinant, the zero-order term in $\adjugate (\asymptotictimeshiftoperator \matricial{\identity} - \asymptoticschememoments)$ is not zero for $\relaxparletter_1 = 0$ but  a singular one-rank diagonal matrix. 
  \item Let $\genericmatrix = \relaxationmatrix \in \lineargroup{\velocitynumber}{\reals} \subset \lineargroup{\velocitynumber}{\taylorseries}$ and $\genericmatrixtwo = \spacestep \left (\frac{1}{\latticevelocity}\partial_t \matricial{\identity} + \duboisoperatormatrix (\matricial{\identity} - \relaxationmatrix)\right ) + \frac{\spacestep^2}{2} \left ( \frac{1}{\latticevelocity^2} \partial_{tt} \matricial{\identity} - \duboisoperatormatrix^2 (\matricial{\identity} - \relaxationmatrix)\right ) + \bigO{\spacestep^3} \in \matrixspace{\velocitynumber}{\taylorseries}$ from \Cref{lemma:ExpansionResolvent}.
  We utilize the previous computations from \eqref{eq:ExpansionDerivativeDeterminant}, as suggested in \Cref{rem:Recycle}, into \eqref{eq:DerivativeAjugate}.
  \begin{align*}
    &\differentialMatrixOperator{\genericmatrix}{\adjugate (\genericmatrix)}{\genericmatrixtwo} = \Biggl ( \spacestep \productrelaxation \Biggl ( \frac{1}{\latticevelocity}  \partial_t  + (1-\relaxparletter_1) \duboisoperatormatrixentry_{11} + \relaxparletter_1 \sum_{\indiceslines = 2}^{\velocitynumber} \frac{1}{\relaxparletter_{\indiceslines}} \Bigr ( \frac{1}{\latticevelocity }\partial_t + (1 - \relaxparletter_{\indiceslines}) \duboisoperatormatrixentry_{\indiceslines \indiceslines}\Bigr ) \Biggr ) \\
    &+ \frac{\spacestep^2}{2} \productrelaxation \Biggl ( \frac{1}{\latticevelocity^2} \partial_{tt} - (1-\relaxparletter_1) \duboisoperatormatrixentry_{11}\duboisoperatormatrixentry_{11} - (1-\relaxparletter_1) \sum_{\ell = 2}^{\velocitynumber} \duboisoperatormatrixentry_{1\ell} \duboisoperatormatrixentry_{\ell 1} + \relaxparletter_1 \sum_{\indiceslines = 2}^{\velocitynumber} \frac{1}{\relaxparletter_{\indiceslines}} \Biggl ( \frac{1}{\latticevelocity^2}\partial_{tt} - (1 - \relaxparletter_{\indiceslines}) \sum_{\ell = 1}^{\velocitynumber} \duboisoperatormatrixentry_{\indiceslines \ell}\duboisoperatormatrixentry_{\ell \indiceslines} \Biggr ) \Biggr )  \Biggr )  \diagmatrix\Biggl ( \frac{1}{\relaxparletter_{1}}, \frac{1}{\relaxparletter_2}, \dots, \frac{1}{\relaxparletter_{\velocitynumber}}\Biggr )\\
    &-   \diagmatrix \Biggl ( \frac{1}{\relaxparletter_{1}}, \frac{1}{\relaxparletter_2}, \dots, \frac{1}{\relaxparletter_{\velocitynumber}}\Biggr )  \genericmatrixtwo\diagmatrix \Biggl ( \frac{1}{\relaxparletter_{1}}, \frac{1}{\relaxparletter_2}, \dots, \frac{1}{\relaxparletter_{\velocitynumber}}\Biggr ) \Biggr ) + \bigO{\spacestep^3}.
  \end{align*}
  In this case, we do not even have to take the limit for $\relaxparletter_{1} \to 0$, since all the terms in $\relaxparletter_{1}$ cancel.
  Therefore, for the very first component, we get
  \begin{align}
    \left (   \differentialMatrixOperator{\genericmatrix}{\adjugate (\genericmatrix)}{\genericmatrixtwo} \right )_{11} &=  \spacestep \productrelaxation \Biggl ( \frac{1}{\latticevelocity}\partial_t  \sum_{\ell = 2}^{\velocitynumber} \frac{1}{\relaxparletter_{\ell}} + \sum_{\indiceslines = 2}^{\velocitynumber} \Biggl (\frac{1}{\relaxparletter_{\indiceslines}} - 1 \Biggr ) \duboisoperatormatrixentry_{\indiceslines\indiceslines} \Biggr ) + \frac{\spacestep}{2} \productrelaxation \Biggl (  \frac{1}{\latticevelocity^2}\partial_{tt}  \sum_{\ell = 2}^{\velocitynumber} \frac{1}{\relaxparletter_{\ell}} - \sum_{\indiceslines = 2}^{\velocitynumber} \Biggl (\frac{1}{\relaxparletter_{\indiceslines}} - 1 \Biggr )\sum_{\ell = 1}^{\velocitynumber} \duboisoperatormatrixentry_{\indiceslines \ell} \duboisoperatormatrixentry_{\ell \indiceslines} \Biggr )  + \bigO{\spacestep^3}. \label{eq:FirstDerivativeAdjugatetOnResolvent11}
  \end{align}
  Now consider $\indicescolumns \in \integerinterval{2}{\velocitynumber}$, then
  \begin{equation}\label{eq:FirstDerivativeAdjugatetOnResolvent1j}
    \left (   \differentialMatrixOperator{\genericmatrix}{\adjugate (\genericmatrix)}{\genericmatrixtwo} \right )_{1 \indicescolumns} = - \spacestep \productrelaxation \Biggl (\frac{1}{\relaxparletter_{\indicescolumns}} - 1 \Biggr ) \duboisoperatormatrixentry_{1 \indicescolumns} + \frac{\spacestep^2}{2} \productrelaxation \Biggl (\frac{1}{\relaxparletter_{\indicescolumns}} - 1 \Biggr ) \Biggl (\duboisoperatormatrixentry_{11}\duboisoperatormatrixentry_{1\indicescolumns} + \sum_{\ell = 2}^{\velocitynumber} \duboisoperatormatrixentry_{1\ell} \duboisoperatormatrixentry_{\ell \indicescolumns} \Biggr ) + \bigO{\spacestep^3}.
  \end{equation}
  This gives all the first-order terms on the first row of $\adjugate (\asymptotictimeshiftoperator \matricial{\identity} - \asymptoticschememoments)$ and part of the second-order terms.
  \item Let $\genericmatrix = \relaxationmatrix \in \lineargroup{\velocitynumber}{\reals} \subset \lineargroup{\velocitynumber}{\taylorseries}$ and $\genericmatrixtwo = \spacestep \left (\frac{1}{\latticevelocity}\partial_t \matricial{\identity} + \duboisoperatormatrix (\matricial{\identity} - \relaxationmatrix) \right ) + \bigO{\spacestep^2}  \in \matrixspace{\velocitynumber}{\taylorseries}$ from \Cref{lemma:ExpansionResolvent}. 
  We reuse computations from \eqref{eq:ExpansionSecondDerivativeDeterminant} as well as \eqref{eq:secondDerivativeAjugate}.
  \begin{align*}
    \secondDifferentialMatrixOperator{\genericmatrix}{\adjugate (\genericmatrix)}{\genericmatrixtwo}{\genericmatrixtwo} = \Biggl ( \spacestep^2  \productrelaxation \Biggl ( &2 \Bigl ( \frac{1}{\latticevelocity}\partial_t + (1-\relaxparletter_1) \duboisoperatormatrixentry_{11} \Bigr ) \sum_{\indiceslines = 2}^{\velocitynumber} \frac{1}{\relaxparletter_{\indiceslines}} \Bigl (\frac{1}{\latticevelocity} \partial_t + (1 - \relaxparletter_{\indiceslines} ) \duboisoperatormatrixentry_{\indiceslines \indiceslines} \Bigr ) + \relaxparletter_{1} \Biggl ( \sum_{\indiceslines = 2}^{\velocitynumber} \frac{1}{\relaxparletter_{\indiceslines}} \Bigl ( \frac{1}{\latticevelocity}\partial_t + (1 - \relaxparletter_{\indiceslines}  ) \duboisoperatormatrixentry_{\indiceslines \indiceslines} \Bigr )\Biggr )^2 \nonumber \\
    -&2   ( 1 - \relaxparletter_1  )  \sum_{\ell = 2}^{\velocitynumber} \left (\frac{1}{\relaxparletter_{\ell}} - 1 \right ) \duboisoperatormatrixentry_{1\ell} \duboisoperatormatrixentry_{\ell 1} - \relaxparletter_{1} \sum_{\indiceslines = 2}^{\velocitynumber} \frac{1}{\relaxparletter_{\indiceslines}^2} \Bigl ( \frac{1}{\latticevelocity}\partial_t + (1 - \relaxparletter_{\indiceslines}) \duboisoperatormatrixentry_{\indiceslines \indiceslines} \Bigr )^2 \nonumber \\
    - &\relaxparletter_1 \sum_{\indiceslines = 2}^{\velocitynumber} \sum_{\substack{\ell = 2 \\ \ell \neq \indiceslines}}^{\velocitynumber} \left (\frac{1}{\relaxparletter_{\indiceslines}} - 1 \right ) \left (\frac{1}{\relaxparletter_{\ell}} - 1 \right ) \duboisoperatormatrixentry_{\indiceslines \ell} \duboisoperatormatrixentry_{ \ell \indiceslines} \Biggr )\Biggr )\diagmatrix \Biggl ( \frac{1}{\relaxparletter_{1}}, \frac{1}{\relaxparletter_2}, \dots, \frac{1}{\relaxparletter_{\velocitynumber}}\Biggr ) \\
    +2 \relaxparletter_{1} \productrelaxation \diagmatrix \Biggl ( \frac{1}{\relaxparletter_{1}}, \frac{1}{\relaxparletter_2}, \dots, \frac{1}{\relaxparletter_{\velocitynumber}}\Biggr ) &\Biggl ( \genericmatrixtwo \relaxationmatrix^{-1} \genericmatrixtwo - \trace(\relaxationmatrix^{-1} \genericmatrixtwo ) \genericmatrixtwo \Biggr )\diagmatrix \Biggl ( \frac{1}{\relaxparletter_{1}}, \frac{1}{\relaxparletter_2}, \dots, \frac{1}{\relaxparletter_{\velocitynumber}}\Biggr )
    +\bigO{\spacestep^3}.
  \end{align*}
  Then we have, for the first matrix entry
  \begin{align}
    \left (  \secondDifferentialMatrixOperator{\genericmatrix}{\adjugate (\genericmatrix)}{\genericmatrixtwo}{\genericmatrixtwo} \right )_{11} =  \spacestep^2 \productrelaxation \Biggl ( \Biggl ( \sum_{\indiceslines = 2}^{\velocitynumber} \frac{1}{\relaxparletter_{\indiceslines}} \Bigl (\frac{1}{\latticevelocity}\partial_t + (1 - \relaxparletter_{\indiceslines} ) \duboisoperatormatrixentry_{\indiceslines \indiceslines} \Bigr ) \Biggr )^2 &- \sum_{\indiceslines = 2}^{\velocitynumber} \frac{1}{\relaxparletter_{\indiceslines}^2} \Bigr (\frac{1}{\latticevelocity}\partial_t + (1 - \relaxparletter_{\indiceslines}) \duboisoperatormatrixentry_{\indiceslines \indiceslines} \Bigr )^2 \label{eq:SecondDerivativeAdjugatetOnResolvent11} \\
     &- \sum_{\indiceslines = 2}^{\velocitynumber}  \Biggl (\frac{1}{\relaxparletter_{\indiceslines}} - 1 \Biggr ) \sum_{\substack{\ell = 2 \\ \ell \neq \indiceslines}}^{\velocitynumber} \Biggl (\frac{1}{\relaxparletter_{\ell}} - 1 \Biggr ) \duboisoperatormatrixentry_{\indiceslines \ell} \duboisoperatormatrixentry_{ \ell \indiceslines} \Biggr ) + \bigO{\spacestep^3}, \nonumber
  \end{align}
  independent from $\relaxparletter_{1}$.
  For $\indicescolumns \in \integerinterval{2}{\velocitynumber}$
  \begin{align}
    \left (  \secondDifferentialMatrixOperator{\genericmatrix}{\adjugate (\genericmatrix)}{\genericmatrixtwo}{\genericmatrixtwo} \right )_{1 \indicescolumns} = 2 \spacestep^2  \productrelaxation  \Biggl (\frac{1}{\relaxparletter_{\indicescolumns}} - 1 \Biggr ) \Biggl ( \frac{1}{\relaxparletter_{\indicescolumns}} &\duboisoperatormatrixentry_{1\indicescolumns} \Bigl ( \frac{1}{\latticevelocity} \partial_t + (1 - \relaxparletter_{\indicescolumns}) \duboisoperatormatrixentry_{\indicescolumns \indicescolumns} \Bigr ) + \sum_{\substack{\ell = 2 \\ \ell \neq \indicescolumns}}^{\velocitynumber} \Biggl (\frac{1}{\relaxparletter_{\ell}} - 1 \Biggr ) \duboisoperatormatrixentry_{1 \ell} \duboisoperatormatrixentry_{\ell \indicescolumns} \nonumber \\
    - &\duboisoperatormatrixentry_{1 \indicescolumns} \sum_{\indiceslines = 2}^{\velocitynumber} \frac{1}{\relaxparletter_{\indiceslines}} \Bigl (\frac{1}{\latticevelocity}\partial_t + (1 - \relaxparletter_{\indiceslines}) \duboisoperatormatrixentry_{\indiceslines \indiceslines} \Bigr ) \Biggr ) + \bigO{\spacestep^3}. \label{eq:SecondDerivativeAdjugatetOnResolvent1j}
  \end{align}
\end{itemize}
Using \eqref{eq:FirstDerivativeAdjugatetOnResolvent11} and \eqref{eq:SecondDerivativeAdjugatetOnResolvent11}, we have that the first entry on the first row of $\adjugate (\asymptotictimeshiftoperator \matricial{\identity} - \asymptoticschememoments)$ is
\begin{align}
  &\lim_{\relaxparletter_1 \to 0}(\adjugate (\asymptotictimeshiftoperator \matricial{\identity} - \asymptoticschememoments))_{11} = \productrelaxation  +  \spacestep \productrelaxation \Biggl ( \frac{1}{\latticevelocity}\partial_t \sum_{\ell = 2}^{\velocitynumber} \frac{1}{\relaxparletter_{\ell}} + \sum_{\indiceslines = 2}^{\velocitynumber} \Biggl (\frac{1}{\relaxparletter_{\indiceslines}} - 1 \Biggr ) \duboisoperatormatrixentry_{\indiceslines\indiceslines} \Biggr ) +  \frac{\spacestep^2}{2} \productrelaxation \Biggl (\frac{1}{\latticevelocity^2}\partial_{tt}  \sum_{\ell = 2}^{\velocitynumber} \frac{1}{\relaxparletter_{\ell}}- \sum_{\indiceslines = 2}^{\velocitynumber} \Biggl (\frac{1}{\relaxparletter_{\indiceslines}} - 1 \Biggr )\sum_{\ell = 1}^{\velocitynumber} \duboisoperatormatrixentry_{\indiceslines \ell} \duboisoperatormatrixentry_{\ell \indiceslines} \nonumber \\
  &+ \Biggl ( \sum_{\indiceslines = 2}^{\velocitynumber} \frac{1}{\relaxparletter_{\indiceslines}} \Bigl ( \frac{1}{\latticevelocity}\partial_t + (1 - \relaxparletter_{\indiceslines}) \duboisoperatormatrixentry_{\indiceslines \indiceslines} \Bigr )\Biggr )^2 - \sum_{\indiceslines = 2}^{\velocitynumber} \frac{1}{\relaxparletter_{\indiceslines}^2} \Bigl ( \frac{1}{\latticevelocity}\partial_t + (1 - \relaxparletter_{\indiceslines}) \duboisoperatormatrixentry_{\indiceslines \indiceslines} \Bigr )^2  - \sum_{\indiceslines = 2}^{\velocitynumber}\Biggl (\frac{1}{\relaxparletter_{\indiceslines}} - 1 \Biggr )  \sum_{\substack{\ell = 2 \\ \ell \neq \indiceslines}}^{\velocitynumber}  \Biggl (\frac{1}{\relaxparletter_{\ell}} - 1 \Biggr ) \duboisoperatormatrixentry_{\indiceslines \ell} \duboisoperatormatrixentry_{ \ell \indiceslines}\Biggr )+ \bigO{\spacestep^3}. \label{eq:AdjugatetExpansionOnResolvent11}
\end{align}
Using \eqref{eq:FirstDerivativeAdjugatetOnResolvent1j} and \eqref{eq:SecondDerivativeAdjugatetOnResolvent1j}, for any $\indicescolumns \in \integerinterval{2}{\velocitynumber}$, we write 
\begin{align}
  \lim_{\relaxparletter_1 \to 0}(\adjugate (\asymptotictimeshiftoperator \matricial{\identity} - \asymptoticschememoments))_{1\indicescolumns} = - \spacestep  \productrelaxation \Biggl (\frac{1}{\relaxparletter_{\indicescolumns}} - 1 \Biggr ) \duboisoperatormatrixentry_{1\indicescolumns} &+ \frac{\spacestep^2}{2} \productrelaxation \Biggl (\frac{1}{\relaxparletter_{\indicescolumns}} - 1 \Biggr ) \Biggl ( \duboisoperatormatrixentry_{11}\duboisoperatormatrixentry_{1\indicescolumns} + \sum_{\ell = 2}^{\velocitynumber} \duboisoperatormatrixentry_{1\ell} \duboisoperatormatrixentry_{\ell \indicescolumns} +  \frac{2}{\relaxparletter_{\indicescolumns}} \duboisoperatormatrixentry_{1\indicescolumns} \Bigl (\frac{1}{\latticevelocity}\partial_t + (1 - \relaxparletter_{\indicescolumns}) \duboisoperatormatrixentry_{\indicescolumns \indicescolumns} \Bigr ) \nonumber \\
  &+ 2\sum_{\substack{\ell = 2 \\ \ell \neq \indicescolumns}}^{\velocitynumber} \Biggl (\frac{1}{\relaxparletter_{\ell}} - 1 \Biggr ) \duboisoperatormatrixentry_{1 \ell} \duboisoperatormatrixentry_{\ell \indicescolumns}  - 2\duboisoperatormatrixentry_{1 \indicescolumns} \sum_{\indiceslines = 2}^{\velocitynumber} \frac{1}{\relaxparletter_{\indiceslines}} \Bigl ( \frac{1}{\latticevelocity}\partial_t + (1 - \relaxparletter_{\indiceslines}) \duboisoperatormatrixentry_{\indiceslines \indiceslines} \Bigr )  \Biggr ) + \bigO{\spacestep^3}. \label{eq:AdjugatetExpansionOnResolvent1j}
\end{align}
In general, we have written, for the first row, the leading terms in $\adjugate (\asymptotictimeshiftoperator \matricial{\identity} - \asymptoticschememoments)$. We shall take its product with $\asymptoticschemeequil$. Thus, one has
\begin{align}
  \adjugate (\asymptotictimeshiftoperator \matricial{\identity} - \asymptoticschememoments) \asymptoticschemeequil = \termatorder{\adjugate (\asymptotictimeshiftoperator \matricial{\identity} - \asymptoticschememoments) }{0} \termatorder{\asymptoticschemeequil}{0} &+ \spacestep \left ( \termatorder{\adjugate (\asymptotictimeshiftoperator \matricial{\identity} - \asymptoticschememoments) }{0} \termatorder{\asymptoticschemeequil}{1} +  \termatorder{\adjugate (\asymptotictimeshiftoperator \matricial{\identity} - \asymptoticschememoments) }{1} \termatorder{\asymptoticschemeequil}{0}\right ), \label{eq:ExpansionRightHandSide} \\
  &+ \spacestep^2 \left (  \termatorder{\adjugate (\asymptotictimeshiftoperator \matricial{\identity} - \asymptoticschememoments) }{0} \termatorder{\asymptoticschemeequil}{2} +  \termatorder{\adjugate (\asymptotictimeshiftoperator \matricial{\identity} - \asymptoticschememoments) }{1} \termatorder{\asymptoticschemeequil}{1} +  \termatorder{\adjugate (\asymptotictimeshiftoperator \matricial{\identity} - \asymptoticschememoments) }{2} \termatorder{\asymptoticschemeequil}{0}\right ) + \bigO{\spacestep^3}, \nonumber
\end{align}
generating products of terms in the fashion of the Cauchy product.
This completes the preliminary results needed to prove \Cref{thm:AcousticScaling}.

\subsection{Overall computation}

We now put all the previous calculations together to prove \Cref{thm:AcousticScaling}.
As previously pointed out, we can assume, without loss of generality, that $\relaxparletter_1 = 0$, passing to the limit.
This allows to deal with simpler expressions with less terms.

\subsubsection{First-order equations}\label{sec:FirstOrderComputations}
To find the target PDE, it is sufficient to truncate all the formal power series at $\bigO{\spacestep^2}$.
In particular, using the fact that the first column of $\asymptoticschemeequil$ is zero for $\relaxparletter_1 = 0$, we have that $\lim_{\relaxparletter_1 \to 0}(\adjugate (\asymptotictimeshiftoperator \matricial{\identity} - \asymptoticschememoments) \asymptoticschemeequil)_{11} = 0$.
Observe that if the relaxation parameter corresponding to the conserved moment were not equal to zero, we would have $(\adjugate (\asymptotictimeshiftoperator \matricial{\identity} - \asymptoticschememoments) \asymptoticschemeequil)_{11} = \bigO{1}$. Still the matrix $\relaxationmatrix$ would not be singular, thus we would have some non vanishing zero-order term in $\determinant (\asymptotictimeshiftoperator \matricial{\identity} - \asymptoticschememoments)$ to compensate the one from the adjugate.

For any $\indicescolumns \in \integerinterval{2}{\velocitynumber}$, using \eqref{eq:AdjugatetExpansionOnResolvent1j}, \Cref{lemma:LinkBetweenWeandDubois} and \eqref{eq:ExpansionRightHandSide}, entails
\begin{align*}
  \lim_{\relaxparletter_1 \to 0}(\adjugate (\asymptotictimeshiftoperator \matricial{\identity} - \asymptoticschememoments) \asymptoticschemeequil)_{1 \indicescolumns} &= \lim_{\relaxparletter_1 \to 0} \spacestep \left ( \left ( \termatorder{\adjugate (\asymptotictimeshiftoperator \matricial{\identity} - \asymptoticschememoments)}{0} \termatorder{\asymptoticschemeequil}{1}  \right )_{1\indicescolumns} + \left ( \termatorder{\adjugate (\asymptotictimeshiftoperator \matricial{\identity} - \asymptoticschememoments)}{1} \termatorder{\asymptoticschemeequil}{0} \right )_{1\indicescolumns} \right ) + \bigO{\spacestep^2} = - \spacestep \productrelaxation \duboisoperatormatrixentry_{1\indicescolumns}  + \bigO{\spacestep^2}. \nonumber
\end{align*}
\eqref{eq:DeterminantExpansionOnResolvent} directly yields
\begin{equation*}
  \lim_{\relaxparletter_1 \to 0}\determinant (\asymptotictimeshiftoperator \matricial{\identity} - \asymptoticschememoments) = \spacestep  \productrelaxation \Bigl ( \frac{1}{\latticevelocity}\partial_t  + \duboisoperatormatrixentry_{11} \Bigr ) + \bigO{\spacestep^2},
\end{equation*}
thus we obtain the modified equation (whatever the choice of $\relaxparletter_1 \in \reals$)
\begin{equation*}
  \spacestep \frac{\productrelaxation}{\latticevelocity}  \Biggl ( \partial_t \momentletter_1  + \latticevelocity \duboisoperatormatrixentry_{11}  \momentletter_1 + \latticevelocity \sum_{\indicescolumns = 2}^{\velocitynumber} \duboisoperatormatrixentry_{1\indicescolumns} \momentletter_{\indicescolumns}^{\atequilibrium} \Biggr ) = \bigO{\spacestep^2},
\end{equation*}
giving the desired result for $\consmomentsnumber = 1$ upon dividing by the constant $\productrelaxation$.
We explicitly see the target PDE.
Observe that the term $\productrelaxation$ is never present in the computations by \cite{dubois2019nonlinear} because they are done on the original \lbm scheme \eqref{eq:SchemeAB} or \eqref{eq:SchemeFullyOperator} which has only one time step.
For instance, in \cite{dubois2019nonlinear}, the multi-step nature of the problem, generated by the non-conserved moments relaxing away from the equilibrium, is damped at the very beginning of the procedure by performing the Taylor expansions of the scheme on the non-conserved variables and then plugging them into the expansions for the conserved moments.

Before clarifying the terms at the next order in the modified equation (for any $\momentletter_1$) or equivalently, finding the precise expression of the truncation error (for $\momentletter_1 \equiv \tilde{\momentletter}_1$ solution of the target PDE), let us utilize the previous equation to get rid of the time derivatives in the second order terms.
This can be rigorously done if $\momentletter_1 \equiv \tilde{\momentletter}_1$, where $ \tilde{\momentletter}_1$ is the smooth solution of the target PDE and yields the truncation error. For any $\momentletter_1$, this is formal because one assumes that differentiation preserves the asymptotic relations from the symbol $\bigO{\cdot}$.
This process constitutes the policy by \cite{dubois2008equivalent, dubois2019nonlinear} and is common to all the approaches (Chapman-Enskog, equivalent equation, Maxwell iteration, \emph{etc.}) in order to find the value of the diffusion coefficients from the second-order terms. 
Moreover, this is classical for \fd schemes, see \cite{warming1974modified, carpentier1997derivation}.
Notice that in this case, where $\consmomentsnumber = 1$, $\gammaDubois{1}{\cdot}$ is a scalar, here denoted $\gammaDuboisScalar{1}$ for brevity.
\begin{align}
    \partial_t \momentletter_1 &=  -\latticevelocity \duboisoperatormatrixentry_{11} \momentletter_{1}  - \latticevelocity \sum_{\indicescolumns = 2}^{\velocitynumber} \duboisoperatormatrixentry_{1\indicescolumns} \momentletter_{\indicescolumns}^{\atequilibrium} + \bigO{\spacestep} = -\gammaDuboisScalar{1} + \bigO{\spacestep}, \label{eq:DerivationPreviousOrder} \\
    \partial_t \vectorial{\momentletter}^{\atequilibrium} &= \straightderivative{\vectorial{\momentletter}^{\atequilibrium}}{\momentletter_1} \partial_t \momentletter_1 = -\straightderivative{\vectorial{\momentletter}^{\atequilibrium}}{\momentletter_1}  \Biggl (\latticevelocity\duboisoperatormatrixentry_{11} \momentletter_{1}  + \latticevelocity \sum_{\indicescolumns = 2}^{\velocitynumber} \duboisoperatormatrixentry_{1\indicescolumns} \momentletter_{\indicescolumns}^{\atequilibrium} \Biggr ) + \bigO{\spacestep} = -\straightderivative{\vectorial{\momentletter}^{\atequilibrium}}{\momentletter_1}  \gammaDuboisScalar{1} + \bigO{\spacestep} , \label{eq:DerivativeTimeEquilibria} \\
    \partial_{tt} \momentletter_1 &= -\partial_t \Biggl ( \latticevelocity \duboisoperatormatrixentry_{11} \momentletter_{1} + \latticevelocity \sum_{\indicescolumns = 2}^{\velocitynumber} \duboisoperatormatrixentry_{1\indicescolumns} \momentletter_{\indicescolumns}^{\atequilibrium} \Biggr ) + \bigO{\spacestep} =  - \latticevelocity\duboisoperatormatrixentry_{11} \partial_t\momentletter_{1} - \latticevelocity\sum_{\indicescolumns = 2}^{\velocitynumber} \duboisoperatormatrixentry_{1\indicescolumns} \partial_t \momentletter_{\indicescolumns}^{\atequilibrium} + \bigO{\spacestep}, \\
    &=\latticevelocity \duboisoperatormatrixentry_{11}  \gammaDuboisScalar{1} + \latticevelocity \sum_{\indicescolumns = 2}^{\velocitynumber} \duboisoperatormatrixentry_{1\indicescolumns} \straightderivative{{\momentletter}_{\indicescolumns}^{\atequilibrium}}{\momentletter_1}  \gammaDuboisScalar{1} + \bigO{\spacestep} = \latticevelocity^2 \duboisoperatormatrixentry_{11} \duboisoperatormatrixentry_{11} \momentletter_{1} + \latticevelocity^2 \duboisoperatormatrixentry_{11} \sum_{\indicescolumns = 2}^{\velocitynumber} \duboisoperatormatrixentry_{1\indicescolumns} \momentletter_{\indicescolumns}^{\atequilibrium} + \latticevelocity \sum_{\indicescolumns = 2}^{\velocitynumber} \duboisoperatormatrixentry_{1\indicescolumns} \straightderivative{{\momentletter}_{\indicescolumns}^{\atequilibrium}}{\momentletter_1}  \gammaDuboisScalar{1} + \bigO{\spacestep}. \label{eq:DerivationPreviousOrderTwice}
\end{align}
These formal equalities are obtained by taking advantage either of the chain rule, since the moments at equilibrium are functions of the conserved moments, or of the re-injection of \eqref{eq:DerivationPreviousOrder} by assuming that the differentiation preserves the asymptotic relations from the symbol $\bigO{\cdot}$.
These equalities become rigorous and lack of the $\bigO{\spacestep}$ term if $\momentletter_1 \equiv \tilde{\momentletter}_1$, the smooth solution of the target PDE.

\subsubsection{Second-order equations}

We can now go to the computation of the truncation error in \Cref{thm:AcousticScaling}, which is more involved due to the presence of more terms to estimate.
To make the link with the findings of \cite{dubois2019nonlinear}, the increased complexity comes from the more intricate and entangled block structure of $\duboisoperatormatrix^2$.
We have to treat the second-order term in \eqref{eq:ExpansionRightHandSide}, made up of three products.
For any $\indicescolumns \in \integerinterval{2}{\velocitynumber}$ (once again, the first component vanishes for $\relaxparletter_1 = 0$)
\begin{itemize}
  \item Using \Cref{lemma:LinkBetweenWeandDubois} and the zero-order expansion of the adjugate gives
  \begin{align*}
    \lim_{\relaxparletter_1 \to 0}\left (\termatorder{\adjugate (\asymptotictimeshiftoperator \matricial{\identity} - \asymptoticschememoments)}{0} \termatorder{\asymptoticschemeequil}{2}\right )_{1\indicescolumns} &=  \frac{\relaxparletter_{\indicescolumns}\productrelaxation}{2}  \Biggl (\duboisoperatormatrixentry_{11} \duboisoperatormatrixentry_{1\indicescolumns} + \sum_{\ell = 2}^{\velocitynumber} \duboisoperatormatrixentry_{1\ell}\duboisoperatormatrixentry_{\ell \indicescolumns} \Biggr ).
  \end{align*}
  \item Using \Cref{lemma:LinkBetweenWeandDubois} with \eqref{eq:AdjugatetExpansionOnResolvent11} and  \eqref{eq:AdjugatetExpansionOnResolvent1j}
  \begin{align*}
    \lim_{\relaxparletter_1 \to 0}\left ( \termatorder{\adjugate (\asymptotictimeshiftoperator \matricial{\identity} - \asymptoticschememoments)}{1} \termatorder{\asymptoticschemeequil}{1} \right )_{1 \indicescolumns} = -\relaxparletter_{\indicescolumns} \productrelaxation   \Biggl (  \frac{1}{\latticevelocity} \duboisoperatormatrixentry_{1\indicescolumns}  \partial_t  \sum_{\ell = 2}^{\velocitynumber} \frac{1}{\relaxparletter_{\ell}}  + \duboisoperatormatrix_{1\indicescolumns} \sum_{\ell = 2}^{\velocitynumber} \Biggl (\frac{1}{\relaxparletter_{\ell}} - 1 \Biggr )\duboisoperatormatrixentry_{\ell \ell} -  \sum_{\ell = 2}^{\velocitynumber} \Biggl (\frac{1}{\relaxparletter_{\ell}} - 1 \Biggr )\duboisoperatormatrixentry_{1\ell} \duboisoperatormatrixentry_{\ell \indicescolumns} \Biggr ).
  \end{align*}
  \item Using \Cref{lemma:LinkBetweenWeandDubois} and \eqref{eq:AdjugatetExpansionOnResolvent1j}
  \begin{align*}
    \lim_{\relaxparletter_1 \to 0}\left ( \termatorder{\adjugate (\asymptotictimeshiftoperator \matricial{\identity} - \asymptoticschememoments)}{2} \termatorder{\asymptoticschemeequil}{0} \right )_{1\indicescolumns} &= \productrelaxation (1 - \relaxparletter_{\indicescolumns} ) \Biggl ( \frac{1}{2}\duboisoperatormatrixentry_{11}\duboisoperatormatrixentry_{1\indicescolumns}+ \sum_{\ell = 2}^{\velocitynumber} \Biggl (\frac{1}{\relaxparletter_{\ell}} - \frac{1}{2} \Biggr ) \duboisoperatormatrixentry_{1 \ell} \duboisoperatormatrixentry_{\ell \indicescolumns}  - \Biggl (\frac{1}{\relaxparletter_{\indicescolumns}} - 1 \Biggr ) \duboisoperatormatrixentry_{1\indicescolumns} \duboisoperatormatrixentry_{\indicescolumns \indicescolumns} \\
    &+  \frac{1}{\relaxparletter_{\indicescolumns}} \duboisoperatormatrixentry_{1\indicescolumns} \Bigl (\frac{1}{\latticevelocity} \partial_t + (1 - \relaxparletter_{\indicescolumns}) \duboisoperatormatrixentry_{\indicescolumns \indicescolumns} \Bigr ) - \duboisoperatormatrixentry_{1 \indicescolumns} \sum_{\indiceslines = 2}^{\velocitynumber} \frac{1}{\relaxparletter_{\indiceslines}} \Bigl ( \frac{1}{\latticevelocity}\partial_t + (1 - \relaxparletter_{\indiceslines}) \duboisoperatormatrixentry_{\indiceslines \indiceslines} \Bigr )  \Biggr ).
  \end{align*}
\end{itemize}
Summing these three contributions and after some straightforward but tedious computations, the second-order term in \eqref{eq:ExpansionRightHandSide} is given by
\begin{equation*}
  \lim_{\relaxparletter_1 \to 0}\left ( \termatorderparenthesis{\adjugate (\asymptotictimeshiftoperator \matricial{\identity} - \asymptoticschememoments) \asymptoticschemeequil}{2} \right )_{1\indicescolumns}= \productrelaxation  \Biggl ( \frac{1}{2}\duboisoperatormatrixentry_{11}\duboisoperatormatrixentry_{1\indicescolumns} + \sum_{\ell = 2}^{\velocitynumber} \Biggl (\frac{1}{\relaxparletter_{\ell}} - \frac{1}{2} \Biggr ) \duboisoperatormatrixentry_{1 \ell} \duboisoperatormatrixentry_{\ell \indicescolumns} -  \frac{1}{\latticevelocity} \Biggl (1 + \sum_{\substack{\ell = 2 \\ \ell \neq \indicescolumns}}^{\velocitynumber} \frac{1}{\relaxparletter_{\ell}} \Biggr ) \duboisoperatormatrixentry_{1\indicescolumns}  \partial_t - \duboisoperatormatrixentry_{1\indicescolumns} \sum_{\ell = 2}^{\velocitynumber} \Biggl (\frac{1}{\relaxparletter_{\ell}} - 1 \Biggr )\duboisoperatormatrixentry_{\ell \ell} \Biggr ).
\end{equation*}
Hence, using \eqref{eq:DerivativeTimeEquilibria} to get rid of the time derivative of the equilibria, we have
\begin{align*}
  \lim_{\relaxparletter_1 \to 0}\sum_{\indicescolumns = 2}^{\velocitynumber} \left ( \termatorderparenthesis{\adjugate (\asymptotictimeshiftoperator \matricial{\identity} - \asymptoticschememoments) \asymptoticschemeequil}{2} \right )_{1\indicescolumns}\momentletter_{\indicescolumns}^{\atequilibrium} =  \productrelaxation \sum_{\indicescolumns = 2}^{\velocitynumber}  \Biggl ( \frac{1}{2}&\duboisoperatormatrixentry_{11}\duboisoperatormatrixentry_{1\indicescolumns} \momentletter_{\indicescolumns}^{\atequilibrium} + \sum_{\ell = 2}^{\velocitynumber} \Biggl (\frac{1}{\relaxparletter_{\ell}} - \frac{1}{2} \Biggr ) \duboisoperatormatrixentry_{1 \ell} \duboisoperatormatrixentry_{\ell \indicescolumns} \momentletter_{\indicescolumns}^{\atequilibrium} \\
  + &\frac{1}{\latticevelocity}\Biggl (1 + \sum_{\substack{\ell = 2 \\ \ell \neq \indicescolumns}}^{\velocitynumber} \frac{1}{\relaxparletter_{\ell}} \Biggr ) \duboisoperatormatrixentry_{1\indicescolumns}  \straightderivative{\momentletter_{\indicescolumns}^{\atequilibrium}}{\momentletter_1}  \gammaDuboisScalar{1} - \duboisoperatormatrixentry_{1\indicescolumns} \sum_{\ell = 2}^{\velocitynumber} \Biggl (\frac{1}{\relaxparletter_{\ell}} - 1 \Biggr )\duboisoperatormatrixentry_{\ell \ell} \momentletter_{\indicescolumns}^{\atequilibrium} \Biggr ) + \bigO{\spacestep}.
\end{align*}
Notice that in this result, a reminder of order $\bigO{\spacestep}$ appears.
Once again, if $\momentletter_1 \equiv \tilde{\momentletter}_1$, this reminder is not present and we would find part of the truncation error.
Once more, using \eqref{eq:DerivationPreviousOrder} and \eqref{eq:DerivationPreviousOrderTwice} to eliminate the time derivatives in the second-order terms from \eqref{eq:DeterminantExpansionOnResolvent} gives
\begin{align*}
  \lim_{\relaxparletter_1 \to 0}\termatorderparenthesis{\determinant (\asymptotictimeshiftoperator \matricial{\identity} - \asymptoticschememoments)}{2} \momentletter_{1} = \productrelaxation \Biggl (& \frac{1}{\latticevelocity^2} \Biggl (\frac{1}{2} + \sum_{\ell = 2}^{\velocitynumber} \frac{1}{\relaxparletter_{\ell}} \Biggr ) \partial_{tt} \momentletter_{1} + \frac{1}{\latticevelocity}  \duboisoperatormatrixentry_{11} \partial_{t} \sum_{\ell = 2}^{\velocitynumber} \frac{1}{\relaxparletter_{\ell}} + \frac{1}{\latticevelocity}\sum_{\indiceslines = 2}^{\velocitynumber} \Biggl ( \frac{1}{\relaxparletter_{\indiceslines}} - 1 \Biggr ) \duboisoperatormatrixentry_{\indiceslines \indiceslines} \partial_t  \momentletter_{1}  \\
  &- \frac{1}{2} \duboisoperatormatrixentry_{11}\duboisoperatormatrixentry_{11} \momentletter_{1}
  - \sum_{\ell = 2}^{\velocitynumber} \Biggl (\frac{1}{\relaxparletter_{\ell}} - \frac{1}{2} \Biggr )\duboisoperatormatrixentry_{1\ell} \duboisoperatormatrixentry_{\ell 1} \momentletter_{1}  + \duboisoperatormatrixentry_{11} \sum_{\indiceslines = 2}^{\velocitynumber} \left ( \frac{1}{\relaxparletter_{\indiceslines}} - 1 \right ) \duboisoperatormatrixentry_{\indiceslines \indiceslines} \momentletter_{1} \Biggr ) \\
  = \productrelaxation  \Biggl ( &\Biggl ( \frac{1}{2} + \sum_{\ell  = 2}^{\velocitynumber} \frac{1}{\relaxparletter_{\ell}} \Biggr ) \Biggl (  \duboisoperatormatrixentry_{11} \duboisoperatormatrixentry_{11} \momentletter_1 +  \duboisoperatormatrixentry_{11} \sum_{\indicescolumns = 2}^{\velocitynumber} \duboisoperatormatrixentry_{1\indicescolumns}\momentletter_{\indicescolumns}^{\atequilibrium}  + \frac{1}{\latticevelocity} \sum_{\indicescolumns = 2}^{\velocitynumber} \duboisoperatormatrixentry_{1\indicescolumns} \straightderivative{{\momentletter}_{\indicescolumns}^{\atequilibrium}}{\momentletter_1}  \gammaDuboisScalar{1}  \Biggr ) \\
  &- \duboisoperatormatrixentry_{11}\Biggl (\duboisoperatormatrixentry_{11} \momentletter_{1} + \sum_{\indicescolumns = 2}^{\velocitynumber} \duboisoperatormatrixentry_{1\indicescolumns} \momentletter_{\indicescolumns}^{\atequilibrium} \Biggr ) \sum_{\ell = 2}^{\velocitynumber} \frac{1}{\relaxparletter_{\ell}}  - \Biggl (\sum_{\indiceslines = 2}^{\velocitynumber} \Biggl ( \frac{1}{\relaxparletter_{\indiceslines}} - 1 \Biggr ) \duboisoperatormatrixentry_{\indiceslines \indiceslines} \Biggr ) \Biggl (\duboisoperatormatrixentry_{11} \momentletter_{1} + \sum_{\indicescolumns = 2}^{\velocitynumber} \duboisoperatormatrixentry_{1\indicescolumns} \momentletter_{\indicescolumns}^{\atequilibrium}\Biggr ) \\
  &-\frac{1}{2}\duboisoperatormatrixentry_{11} \duboisoperatormatrixentry_{11} \momentletter_{1} - \sum_{\ell = 2}^{\velocitynumber} \Biggl (\frac{1}{\relaxparletter_{\ell}} - \frac{1}{2} \Biggr ) \duboisoperatormatrixentry_{1\ell} \duboisoperatormatrixentry_{\ell 1} \momentletter_{1} + \duboisoperatormatrixentry_{11} \sum_{\indiceslines = 2}^{\velocitynumber} \Biggl ( \frac{1}{\relaxparletter_{\indiceslines}} - 1 \Biggr ) \duboisoperatormatrixentry_{\indiceslines \indiceslines} \momentletter_{1} \Biggr )  + \bigO{\spacestep}.
\end{align*}
With this, after simplifications, we obtain the remaining term to master the second-order contributions in the modified equation of the \fd scheme \eqref{eq:FDSchemeAdjoint}.
\begin{align*}
  &\termatorderparenthesis{\determinant (\asymptotictimeshiftoperator \matricial{\identity} - \asymptoticschememoments)}{2} \momentletter_{1} -  \sum_{\indicescolumns = 2}^{\velocitynumber} \left ( \termatorderparenthesis{\adjugate (\asymptotictimeshiftoperator \matricial{\identity} - \asymptoticschememoments) \asymptoticschemeequil}{2} \right )_{1\indicescolumns}\momentletter_{\indicescolumns}^{\atequilibrium} \\
  &= -\productrelaxation \Biggl ( \sum_{\indicescolumns = 2}^{\velocitynumber} \Biggl (\frac{1}{\relaxparletter_{\indicescolumns}} - \frac{1}{2} \Biggr ) \duboisoperatormatrixentry_{1\indicescolumns} \duboisoperatormatrixentry_{\indicescolumns 1} \momentletter_{1}  + \sum_{\indicescolumns = 2}^{\velocitynumber} \sum_{\ell = 2}^{\velocitynumber}  \Biggl (\frac{1}{\relaxparletter_{\ell}} - \frac{1}{2} \Biggr ) \duboisoperatormatrixentry_{1\ell} \duboisoperatormatrixentry_{\ell \indicescolumns} \momentletter_{\indicescolumns}^{\atequilibrium} - \frac{1}{\latticevelocity} \sum_{\indicescolumns = 2}^{\velocitynumber} \Biggl (\frac{1}{\relaxparletter_{\indicescolumns}} - \frac{1}{2} \Biggr ) \duboisoperatormatrixentry_{1\indicescolumns} \straightderivative{\momentletter_{\indicescolumns}^{\atequilibrium}}{\momentletter_{1}} \gammaDuboisScalar{1} \Biggr ) + \bigO{\spacestep}.
\end{align*}
To wrap up, these computations yield, together with the ones from \Cref{sec:FirstOrderComputations}, the expected result for $\consmomentsnumber = 1$, which reads
\begin{equation*}
    \spacestep\frac{\productrelaxation}{\latticevelocity} \Biggl ( \partial_{t} \momentletter_1 + \gammaDuboisScalar{1} -\latticevelocity \spacestep \sum_{\indicescolumns = 2}^{\velocitynumber} \Biggl (\frac{1}{\relaxparletter_{\indicescolumns}} - \frac{1}{2} \Biggr ) \duboisoperatormatrixentry_{1\indicescolumns} \Biggl (\duboisoperatormatrixentry_{\indicescolumns 1} \momentletter_1 + \sum_{\ell = 2}^{\velocitynumber} \duboisoperatormatrixentry_{\indicescolumns \ell} \momentletter_{\ell}^{\atequilibrium} - \frac{1}{\latticevelocity} \straightderivative{\momentletter^{\atequilibrium}_{\indicescolumns}}{\momentletter_1} \gammaDuboisScalar{1} \Biggr )\Biggr ) = \bigO{\spacestep^3},
\end{equation*}
and thus proves \Cref{thm:AcousticScaling}.

 \section{Extension of the proofs to several conserved moments: key ideas}\label{sec:ExtansionSeveralConservedMoments}

 In this Section, we sketch the demonstration of \Cref{thm:AcousticScaling} for any $\consmomentsnumber \geq 1$.
 For the sake of providing a quick and effective presentation of this matter, we limit ourselves to first-order in $\spacestep$.
 Select a conserved moment, which shall be indexed by $\indiceconserved \in \integerinterval{1}{\consmomentsnumber}$.
 \begin{remark}
   The operation selecting rows and columns to yield $\schememoments_{\indiceconserved}$ and $\schememomentsother_{\indiceconserved}$ from \Cref{prop:ReductionFiniteDifferenceNGeq1Old} bis does not change the orders of the expansions. 
   This is, let $\genericdiscretematrix \in \matrixspace{\velocitynumber}{\discretetimespaceoperators}$ and $\genericcontinuousmatrix = \sum_{\perturbationorderindices = 0}^{\perturbationorderindices = +\infty} \spacestep^{\perturbationorderindices} \termatorder{\genericcontinuousmatrix}{\perturbationorderindices} \in \matrixspace{\velocitynumber}{\taylorseries}$ such that $\genericdiscretematrix \asymtoticequivalence \genericcontinuousmatrix$ and $\setoflines \subset \integerinterval{1}{\velocitynumber}$ a set of indices, then
   \begin{equation*}
     \cutmatrixsquare{\genericdiscretematrix}{\setoflines} \asymtoticequivalence \cutmatrixsquare{\left (\sum_{\perturbationorderindices = 0}^{+\infty} \spacestep^{\perturbationorderindices} \termatorder{\genericcontinuousmatrix}{\perturbationorderindices} \right )}{\setoflines}  = \sum_{\perturbationorderindices = 0}^{+\infty} \spacestep^{\perturbationorderindices} \cutmatrixsquare{\left ( \termatorder{\genericcontinuousmatrix}{\perturbationorderindices} \right )}{\setoflines}.
   \end{equation*}
 \end{remark}
 Thus we have the analogous of \Cref{lemma:ExpansionResolvent}, where $\timeshiftoperator \matricial{\identity} - \schememoments_{\indiceconserved} \asymtoticequivalence \asymptotictimeshiftoperator \matricial{\identity} - \asymptoticschememoments_{\indiceconserved}$ , with
\begin{equation*}
  \asymptotictimeshiftoperator \matricial{\identity} - \asymptoticschememoments_{\indiceconserved} = \sum_{\perturbationorderindices = 0}^{+\infty} \frac{\spacestep^{\perturbationorderindices}}{ \perturbationorderindices ! } \left ( \frac{1}{\latticevelocity^{\perturbationorderindices}}\partial_t^{\perturbationorderindices} \matricial{\identity} - (-1)^{\perturbationorderindices} \cutmatrixsquare{\left ( \duboisoperatormatrix^{\perturbationorderindices} (\matricial{\identity} - \relaxationmatrix) \right )}{\{ \indiceconserved\} \cup \integerinterval{\consmomentsnumber + 1}{\velocitynumber}} \right ).
\end{equation*}
The first two term in the expansion of the inverse of the resolvent are
\begin{equation*}
  \termatorderparenthesis{ \asymptotictimeshiftoperator \matricial{\identity} - \asymptoticschememoments_{\indiceconserved}}{0} = \diagmatrix (1, \dots, 1, \relaxparletter_{\indiceconserved}, 1, \dots, 1, \relaxparletter_{\consmomentsnumber + 1}, \dots, \relaxparletter_{\velocitynumber}).
\end{equation*}
In the spirit of \Cref{rem:Invertibility} and the computations developed in \Cref{sec:DetailedProofs}, for the case $\relaxparletter_{\indiceconserved} = 0$, we introduce a regularization with $\relaxparletter_{\indiceconserved} \neq 0$ and then we pass to the limit.
Moreover
\begin{equation*}
  \termatorderparenthesis{ \asymptotictimeshiftoperator \matricial{\identity} - \asymptoticschememoments_{\indiceconserved}}{1}  = \frac{1}{\latticevelocity}\partial_t \matricial{\identity} + 
  \left(\begin{array}{@{}ccc|c|ccc||ccc@{}}
      0 & \cdots & 0 & 0 & 0 & \cdots & 0 & 0 & \cdots & 0\\
      \vdots & \ddots & \vdots & \vdots & \vdots & \ddots & \vdots & \vdots & \ddots & \vdots\\
      0 & \cdots & 0 & 0 & 0 & \cdots & 0 & 0 & \cdots & 0 \\ \hline
      0 & \cdots & 0 & (1 - \relaxparletter_{\indiceconserved}) \duboisoperatormatrixentry_{\indiceconserved \indiceconserved} & 0 & \cdots & 0 & (1-\relaxparletter_{\consmomentsnumber + 1})\duboisoperatormatrixentry_{\indiceconserved (\consmomentsnumber + 1)} & \cdots & (1-\relaxparletter_{\velocitynumber})\duboisoperatormatrixentry_{\indiceconserved \velocitynumber} \\ \hline
      0 & \cdots & 0 & 0 & 0 & \cdots & 0 & 0 & \cdots & 0\\
      \vdots & \ddots & \vdots & \vdots & \vdots & \ddots & \vdots & \vdots & \ddots & \vdots\\
      0 & \cdots & 0 & 0 & 0 & \cdots & 0 & 0 & \cdots & 0 \\ \hline \hline
      0 & \cdots & 0 & (1 - \relaxparletter_{\indiceconserved})\duboisoperatormatrixentry_{(\consmomentsnumber + 1)\indiceconserved} & 0 & \cdots & 0 & (1-\relaxparletter_{\consmomentsnumber + 1})\duboisoperatormatrixentry_{(\consmomentsnumber + 1)(\consmomentsnumber + 1)} & \cdots &  (1-\relaxparletter_{\velocitynumber})\duboisoperatormatrixentry_{(\consmomentsnumber + 1)\velocitynumber} \\
      \vdots & \ddots & \vdots & \vdots & \vdots & \ddots & \vdots & \vdots & \ddots & \vdots \\
      0 & \cdots & 0 & (1 - \relaxparletter_{\indiceconserved})\duboisoperatormatrixentry_{\velocitynumber \indiceconserved} & 0 & \cdots & 0 & (1-\relaxparletter_{\consmomentsnumber + 1})\duboisoperatormatrixentry_{\velocitynumber(\consmomentsnumber + 1)} & \cdots &  (1-\relaxparletter_{\velocitynumber})\duboisoperatormatrixentry_{\velocitynumber \velocitynumber}
  \end{array}
  \right ).
\end{equation*}

We thus have
\begin{itemize}
  \item As for the case $\consmomentsnumber = 1$ treated in detail, we have that $\lim_{\relaxparletter_{\indiceconserved} \to 0}\termatorder{ \determinant ((\asymptotictimeshiftoperator \matricial{\identity} - \asymptoticschememoments_{\indiceconserved})}{0}) = 0$.
  Using the formula for the adjugate of an upper triangular matrix, see \cite{horn2012matrix}, we have $\lim_{\relaxparletter_{\indiceconserved} \to 0}\termatorder{ \adjugate ( (\asymptotictimeshiftoperator \matricial{\identity} - \asymptoticschememoments_{\indiceconserved})}{0})  = \productrelaxation \canonicalbasisvector_{\indiceconserved} \otimes \canonicalbasisvector_{\indiceconserved}$, where in this Section $\productrelaxation \definitionequality \prod_{\ell = \consmomentsnumber + 1}^{\ell = \velocitynumber} \relaxparletter_{\ell}$.
  \item Taking $\genericmatrix = \termatorderparenthesis{ \asymptotictimeshiftoperator \matricial{\identity} - \asymptoticschememoments_{\indiceconserved}}{0} \in \lineargroup{\velocitynumber}{\reals} \subset \lineargroup{\velocitynumber}{\taylorseries} $ and $\genericmatrixtwo = \spacestep \termatorderparenthesis{ \asymptotictimeshiftoperator \matricial{\identity} - \asymptoticschememoments_{\indiceconserved}}{1} + \bigO{\spacestep^2} \in \matrixspace{\velocitynumber}{\taylorseries}$ in the Jacobi formula \eqref{eq:JacobiFormula}
  \begin{align*}
    \lim_{\relaxparletter_{\indiceconserved} \to 0}\differentialMatrixOperator{\genericmatrix}{\determinant (\genericmatrix)}{\genericmatrixtwo} &= \lim_{\relaxparletter_{\indiceconserved} \to 0} \spacestep  \productrelaxation  \Biggl ( \frac{\relaxparletter_{\indiceconserved} (\consmomentsnumber - 1)}{\latticevelocity} \partial_t + \frac{1}{\latticevelocity}\partial_t +(1-\relaxparletter_{\indiceconserved}) \duboisoperatormatrixentry_{\indiceconserved \indiceconserved}  +\sum_{\ell = \consmomentsnumber + 1}^{\velocitynumber} \frac{1}{\relaxparletter_{\ell}} \Bigl ( \frac{1}{\latticevelocity}\partial_t + (1 - \relaxparletter_{\ell}) \duboisoperatormatrixentry_{\ell \ell} \Bigr ) \Biggr ) + \bigO{\spacestep^2} \\
    &= \spacestep \productrelaxation \Bigl ( \frac{1}{\latticevelocity}\partial_t + \duboisoperatormatrixentry_{\indiceconserved \indiceconserved} \Bigr ) + \bigO{\spacestep^2}.
  \end{align*}

  To handle the term with the adjugate, observe that the first-order term is made up of the terms
  \begin{equation}\label{eq:tmp1}
    \termatorderparenthesis{\adjugate(\asymptotictimeshiftoperator \matricial{\identity} - \asymptoticschememoments_{\indiceconserved}) \asymptoticschememomentsother_{\indiceconserved}}{1} = \termatorderparenthesis{\adjugate(\asymptotictimeshiftoperator \matricial{\identity} - \asymptoticschememoments_{\indiceconserved}) }{0} \termatorderparenthesis{\asymptoticschememomentsother_{\indiceconserved}}{1} + \termatorderparenthesis{\adjugate(\asymptotictimeshiftoperator \matricial{\identity} - \asymptoticschememoments_{\indiceconserved}) }{1} \termatorderparenthesis{\asymptoticschememomentsother_{\indiceconserved}}{0},
  \end{equation}
  and in particular, we are interested in the $\indiceconserved$-th line of this matrix.
  Because of the fact that $\termatorderparenthesis{\asymptoticschememomentsother_{\indiceconserved}}{0} = \diagmatrix (1 - \relaxparletter_{1}, \dots, 1 - \relaxparletter_{\indiceconserved - 1}, 0, 1 - \relaxparletter_{\indiceconserved + 1}, \dots, 1 - \relaxparletter_{\consmomentsnumber}, 0, \dots, 0)$, the $\indiceconserved$-th line of the second term on the right hand side of \eqref{eq:tmp1} is zero, thus we do not have do study it.
  For the remaining term, it can be easily seen that 
  \begin{align*}
    \termatorderparenthesis{\asymptoticschememomentsother_{\indiceconserved} }{1}  
    = -(\matricial{\identity} - \relaxationmatrix)
    \arraycolsep=0.5pt\def\arraystretch{1.5}
    \left(\begin{array}{@{}ccc|c|ccc||ccc@{}}
        \duboisoperatormatrixentry_{11} & \cdots & \duboisoperatormatrixentry_{1(\indiceconserved - 1)} & \duboisoperatormatrixentry_{1\indiceconserved} & \duboisoperatormatrixentry_{1 (\indiceconserved + 1)} & \cdots & \duboisoperatormatrixentry_{1 \consmomentsnumber} & \duboisoperatormatrixentry_{1 (\consmomentsnumber + 1)} & \cdots & \duboisoperatormatrixentry_{1\velocitynumber}\\
        \vdots & \ddots & \vdots & \vdots & \vdots & \ddots & \vdots & \vdots & \ddots & \vdots\\
        \duboisoperatormatrixentry_{ (\indiceconserved - 1)1} & \cdots & \duboisoperatormatrixentry_{(\indiceconserved - 1)(\indiceconserved - 1)} & \duboisoperatormatrixentry_{(\indiceconserved - 1)\indiceconserved} & \duboisoperatormatrixentry_{(\indiceconserved - 1)(\indiceconserved + 1)} & \cdots & \duboisoperatormatrixentry_{(\indiceconserved - 1) \consmomentsnumber} &  \duboisoperatormatrixentry_{(\indiceconserved - 1) (\consmomentsnumber + 1)} & \cdots & \duboisoperatormatrixentry_{(\indiceconserved - 1)\velocitynumber}\\ \hline
        \duboisoperatormatrixentry_{\indiceconserved 1} & \cdots & \duboisoperatormatrixentry_{\indiceconserved (\indiceconserved - 1)}  & 0 &  \duboisoperatormatrixentry_{\indiceconserved (\indiceconserved + 1)} & \cdots & \duboisoperatormatrixentry_{\indiceconserved  \consmomentsnumber} & 0 & \cdots & 0\\ \hline
        \duboisoperatormatrixentry_{ (\indiceconserved + 1)1} & \cdots & \duboisoperatormatrixentry_{(\indiceconserved + 1)(\indiceconserved - 1)} & \duboisoperatormatrixentry_{(\indiceconserved + 1)\indiceconserved} & \duboisoperatormatrixentry_{(\indiceconserved + 1)(\indiceconserved + 1)} & \cdots & \duboisoperatormatrixentry_{(\indiceconserved + 1) \consmomentsnumber} & \duboisoperatormatrixentry_{(\indiceconserved + 1) (\consmomentsnumber + 1)} & \cdots & \duboisoperatormatrixentry_{(\indiceconserved + 1)\velocitynumber}\\
        \vdots & \ddots & \vdots & \vdots & \vdots & \ddots & \vdots & \vdots & \ddots & \vdots\\
        \duboisoperatormatrixentry_{ \consmomentsnumber 1} & \cdots & \duboisoperatormatrixentry_{\consmomentsnumber(\indiceconserved - 1)} & \duboisoperatormatrixentry_{\consmomentsnumber\indiceconserved} & \duboisoperatormatrixentry_{\consmomentsnumber(\indiceconserved + 1)} & \cdots & \duboisoperatormatrixentry_{\consmomentsnumber \consmomentsnumber} &  \duboisoperatormatrixentry_{\consmomentsnumber (\consmomentsnumber + 1)} & \cdots & \duboisoperatormatrixentry_{\consmomentsnumber\velocitynumber} \\ \hline \hline
        \duboisoperatormatrixentry_{(\consmomentsnumber + 1) 1} & \cdots & \duboisoperatormatrixentry_{(\consmomentsnumber + 1)(\indiceconserved - 1)}  & 0 & \duboisoperatormatrixentry_{(\consmomentsnumber+1)(\indiceconserved + 1)} & \cdots & \duboisoperatormatrixentry_{(\consmomentsnumber+1) \consmomentsnumber} & 0 & \cdots & 0 \\
        \vdots & \ddots & \vdots & \vdots & \vdots & \ddots & \vdots & \vdots & \ddots & \vdots \\
        \duboisoperatormatrixentry_{\velocitynumber 1} & \cdots & \duboisoperatormatrixentry_{\velocitynumber(\indiceconserved - 1)}  & 0 & \duboisoperatormatrixentry_{\velocitynumber(\indiceconserved + 1)} & \cdots & \duboisoperatormatrixentry_{\velocitynumber \consmomentsnumber} & 0 & \cdots & 0 
    \end{array}
    \right ),
\end{align*}
thus we deduce that 
\begin{equation*}
  \left (\termatorderparenthesis{\adjugate(\asymptotictimeshiftoperator \matricial{\identity} - \asymptoticschememoments_{\indiceconserved}) \asymptoticschememomentsother_{\indiceconserved}}{1} \right )_{\indiceconserved, \cdot} = -\productrelaxation \left ( (1 - \relaxparletter_{1})\duboisoperatormatrixentry_{\indiceconserved 1}, \dots, (1 - \relaxparletter_{\indiceconserved - 1})\duboisoperatormatrixentry_{\indiceconserved  (\indiceconserved - 1)}, 0, (1 - \relaxparletter_{\indiceconserved + 1})\duboisoperatormatrixentry_{\indiceconserved  (\indiceconserved + 1)}, \dots, (1 - \relaxparletter_{\consmomentsnumber}) \duboisoperatormatrixentry_{\indiceconserved \consmomentsnumber}, 0, \dots, 0\right ).
\end{equation*}
Dealing with the zero and first order term in $\adjugate(\asymptotictimeshiftoperator \matricial{\identity} - \asymptoticschememoments_{\indiceconserved}) \asymptoticschemeequil$ works the same than $\consmomentsnumber = 1$, thus we do not repeat it.
Moreover, these terms allow for the compensation of the dependence on the choice of the relaxation parameter of the other conserved moments $\relaxparletter_{1}, \cdots, \relaxparletter_{\indiceconserved - 1}, \relaxparletter_{\indiceconserved + 1}, \dots, \relaxparletter_{\consmomentsnumber}$ in the previous equation, as claimed in \Cref{sec:InvarianceChoiceRelaxationParamtersConserved}, thanks to \eqref{eq:EquilibriumConservedMoments}.
\end{itemize}

Putting all the previously discussed facts together into the truncated \eqref{eq:DevelopedEquation} yields
\begin{equation*}
  \spacestep \frac{\productrelaxation}{\latticevelocity} \Biggl (\partial_t \momentletter_{\indiceconserved} + \latticevelocity \duboisoperatormatrixentry_{\indiceconserved \indiceconserved} \momentletter_{\indiceconserved} + \latticevelocity \sum_{\substack{\indicescolumns = 1 \\ \indicescolumns \neq \indiceconserved}}^{\consmomentsnumber} \duboisoperatormatrixentry_{\indiceconserved \indicescolumns} \momentletter_{\indicescolumns} + \latticevelocity\sum_{\indicescolumns = \consmomentsnumber + 1}^{\velocitynumber} \duboisoperatormatrixentry_{\indiceconserved \indicescolumns} \momentletter_{\indicescolumns}^{\atequilibrium} \Biggr ) = \bigO{\spacestep^2},
\end{equation*}
which is the result from \Cref{thm:AcousticScaling} for $\consmomentsnumber \geq 1$ at dominant order.
The next order is demonstrated in the same way.

\section{Link with the existing approaches}\label{sec:LinkExistingApproaches}

To finish our contribution, we briefly sketch the links with previous works on the target PDEs and modified equations like \cite{yong2016theory} and \cite{dubois2008equivalent, dubois2019nonlinear}.
A more complete study shall be the object of future investigations.

\subsection{Equivalent equations}

Our result \Cref{thm:AcousticScaling} coincides with the analogous result in \cite{dubois2019nonlinear} up to second order.
The substantial difference is that we apply the Taylor expansions to the solution of the corresponding \fd scheme given either by \Cref{prop:ReductionFiniteDifferenceGeneralOld} bis or \Cref{prop:ReductionFiniteDifferenceNGeq1Old} bis, where non-conserved moments have been removed.
We therefore reasonably conjecture that the obtained macroscopic equations coincide at any order. 
The mathematical justification of this conjecture shall be the object of future investigations.

The quasi-equilibrium,  which is extensively used in \cite{dubois2019nonlinear} can be somehow recovered in our previous discussion.
Let $\consmomentsnumber = 1$ to fix ideas.
In the proof of \Cref{prop:ReductionFiniteDifferenceGeneralOld} bis, nothing prevents us from selecting, instead of the first row, the $\indicemoments \in \integerinterval{2}{\velocitynumber}$ row, corresponding to a non-conserved moment.
This is
\begin{equation}\label{eq:FDSchemeNonConservedMoments}
  \determinant(\timeshiftoperator \matricial{\identity} - \schememoments)\momentletter_{\indicemoments} = \left ( \adjugate(\timeshiftoperator \matricial{\identity} - \schememoments) \schemeequil \vectorial{\momentletter}^{\atequilibrium} \right )_{\indicemoments}.
\end{equation}
Let us stress that even if this could seem to be a viable \fd scheme for the non-conserved variable $\momentletter_{\indicemoments}$, it is not independent from the conserved moment $\momentletter_{1}$ the equilibria depend on and furthermore, this formulation certainly depends on the choice of $\relaxparletter_{1}$, the relaxation parameter of the conserved moment. This is somehow unwanted since $\relaxparletter_{1}$ is \emph{in fine} not present in the original \lbm scheme.
From the computations of \Cref{sec:DetailedProofs}, we see that 
\begin{equation*}
  \determinant(\asymptotictimeshiftoperator \matricial{\identity} - \asymptoticschememoments) = \relaxparletter_{1} \productrelaxation + \bigO{\spacestep}, \qquad \adjugate(\asymptotictimeshiftoperator \matricial{\identity} - \asymptoticschememoments) =  \productrelaxation \diagmatrix \Biggl (1, \frac{\relaxparletter_{1} }{\relaxparletter_{2}}, \dots, \frac{\relaxparletter_{1} }{\relaxparletter_{\velocitynumber}} \Biggr ) + \bigO{\spacestep}, \qquad \asymptoticschemeequil = \relaxationmatrix + \bigO{\spacestep}.
\end{equation*}
Using the asymptotic equivalents truncated at leading order in \eqref{eq:FDSchemeNonConservedMoments} thus provides
\begin{equation*}
  \relaxparletter_{1} \productrelaxation \momentletter_{\indicemoments} + \bigO{\spacestep} = \relaxparletter_{1} \productrelaxation  \momentletter_{\indicemoments}^{\atequilibrium} + \bigO{\spacestep}, \qquad \text{hence also} \qquad \momentletter_{\indicemoments} = \momentletter_{\indicemoments}^{\atequilibrium} + \bigO{\spacestep},
\end{equation*}
provided that $\relaxparletter_{1} \neq 0$.
This is the quasi-equilibrium of the non-conserved moments, which is re-injected in the \lbm schemes to eliminate them in the procedure by \cite{dubois2019nonlinear}.
The previous procedure is formal because there is no guarantee that the discrete non-conserved moments $\momentletter_{\indicemoments}$ for $\indicemoments \in \integerinterval{2}{\velocitynumber}$ in the scheme originate from the point-wise values of a smooth function.

\subsection{Maxwell iteration}
In \cite{yong2016theory}, the computations have been carried only for the \scheme{2}{9} scheme by \cite{lallemand2000theory} with $\consmomentsnumber = 3$, which we have presented in \Cref{ex:D2Q9N3Acoustic}.
In this part of our work, we are first going to develop the computations until third-order for any \lbm scheme under acoustic scaling, \emph{i.e.} \Cref{ass:Acoustic}.
Then, we are going to demonstrate that the modified equations obtained by the Maxwell iteration \cite{yong2016theory} and the one from the corresponding \fd schemes are the same at any order, regardless of the time-space scaling.
In this Section, it is crucial to assume that $\relaxationmatrix \in \lineargroup{\velocitynumber}{\reals}$.
Observe that this assumption ensures that $\determinant(\asymptotictimeshiftoperator \matricial{\identity} - \asymptoticschememoments)$ is a unit (invertible) in the ring $\taylorseries$ or equivalently that $\asymptotictimeshiftoperator \matricial{\identity} - \asymptoticschememoments$ belongs to $\lineargroup{\velocitynumber}{\taylorseries}$.
The Maxwell iteration \cite{yong2016theory} at step $\indicesmaxwelliteration \in \naturals$ reads, after simple computations
\begin{equation}\label{eq:DefinitionMaxwellIteration}
  \ordermaxwelliteration{\vectorial{\momentletter}}{\indicesmaxwelliteration} = \Biggl (\sum_{\perturbationorderindices = 0}^{k} (-\relaxationmatrix^{-1} (\asymptotictimeshiftoperator \conj{\asymptoticstreammoments} - \matricial{\identity}))^{\perturbationorderindices}\Biggr ) \vectorial{\momentletter}^{\atequilibrium},
\end{equation}
where the quasi-equilibrium is encoded in the choice $\ordermaxwelliteration{\vectorial{\momentletter}}{0} =  \vectorial{\momentletter}^{\atequilibrium}$  and where we have taken, as for \eqref{eq:ExpansionsStreamMatrix}
\begin{equation*}
  \conj{\streammoments} \definitionequality \momentsmatrix \diagmatrix(\shiftoperator{-\normalizedvelocityletter_1}, \dots, \shiftoperator{-\normalizedvelocityletter_{\velocitynumber}}) \momentsmatrix^{-1} \asymtoticequivalence \momentsmatrix \Biggl ( \sum_{\multiindicemodule{\multiindice} \geq 0} \frac{\spacestep^{\multiindicemodule{\multiindice}}}{\multiindice!}\diagmatrix \left (\vectorial{\vectorial{\normalizedvelocityletter}_1}^{\multiindice}, \dots, \vectorial{\vectorial{\normalizedvelocityletter}_{\velocitynumber}}^{\multiindice}\right ) \multiindicederivative{\multiindice} \Biggr ) \momentsmatrix^{-1} =: \conj{\asymptoticstreammoments} \in \matrixspace{\velocitynumber}{\taylorseries}.
\end{equation*}
It is easy to see that $\asymptoticstreammoments \conj{\asymptoticstreammoments} = \conj{\asymptoticstreammoments}\asymptoticstreammoments = \matricial{\identity}$ and moreover, in analogy with \Cref{lemma:ExpansionResolvent}
\begin{equation}\label{eq:ExpansionForMaxwell}
  \asymptotictimeshiftoperator \conj{\asymptoticstreammoments} - \matricial{\identity} = \spacestep \Bigl ( \frac{1}{\latticevelocity} \partial_t \matricial{\identity} + \duboisoperatormatrix \Bigr ) + \frac{\spacestep^2}{2} \Bigr ( \frac{1}{\latticevelocity^2} \partial_{tt}  \matricial{\identity}  + \frac{2}{\latticevelocity} \duboisoperatormatrix \partial_t + \duboisoperatormatrix^{2} \Bigr ) + \bigO{\spacestep^3}.
\end{equation}
The Maxwell iteration works by assuming that $\vectorial{\momentletter} = \ordermaxwelliteration{\vectorial{\momentletter}}{\indicesmaxwelliteration} + \bigO{\spacestep^{\indicesmaxwelliteration + 1}}$.
Taking $\indicesmaxwelliteration = 1$ in \eqref{eq:DefinitionMaxwellIteration} and using \eqref{eq:ExpansionForMaxwell}, we have
\begin{align*}
  \vectorial{\momentletter} = \vectorial{\momentletter}^{\atequilibrium} -\relaxationmatrix^{-1}   \spacestep \Bigl ( \frac{1}{\latticevelocity} \partial_t \matricial{\identity} + \duboisoperatormatrix \Bigr )\vectorial{\momentletter}^{\atequilibrium} + \bigO{\spacestep^2}.
\end{align*}
Let $\indiceconserved \in \integerinterval{1}{\consmomentsnumber}$, then taking advantage of \eqref{eq:EquilibriumConservedMoments}
\begin{equation*}
  \momentletter_{\indiceconserved} = \momentletter_{\indiceconserved} -  \spacestep \frac{1}{ \relaxparletter_{\indiceconserved}} \Biggl ( \frac{1}{\latticevelocity} \partial_t \momentletter_{\indiceconserved}  + \sum_{\indicescolumns = 1}^{\consmomentsnumber} \duboisoperatormatrixentry_{\indiceconserved \indicescolumns}  \momentletter_{\indicescolumns} + \sum_{\indicescolumns = \consmomentsnumber + 1}^{\velocitynumber} \duboisoperatormatrixentry_{1\indicescolumns} \momentletter_{\indicescolumns}^{\atequilibrium} \Biggr )+ \bigO{\spacestep^2},
\end{equation*}
which upon division, is the same result than \Cref{thm:AcousticScaling}.
Going up to order two considering $\indicesmaxwelliteration = 2$, we have
\begin{align*}
  \vectorial{\momentletter} = \vectorial{\momentletter}^{\atequilibrium} &-\relaxationmatrix^{-1}   \spacestep \Bigl ( \frac{1}{\latticevelocity} \partial_t \matricial{\identity} + \duboisoperatormatrix \Bigr )\vectorial{\momentletter}^{\atequilibrium}  \\
  &+ \frac{\spacestep^2}{2} \relaxationmatrix^{-1} \Biggl (  \frac{1}{\latticevelocity^2}(2 \relaxationmatrix^{-1} - \matricial{\identity}) \partial_{tt} +  \frac{2}{\latticevelocity} (\relaxationmatrix^{-1}\duboisoperatormatrix +\duboisoperatormatrix \relaxationmatrix^{-1} - \duboisoperatormatrix )\partial_t + \duboisoperatormatrix (2\relaxationmatrix^{-1} - \matricial{\identity})\duboisoperatormatrix \Biggr ) \vectorial{\momentletter}^{\atequilibrium} + \bigO{\spacestep^3}.
\end{align*}
Once more, selecting the $\indiceconserved$-th row provides
\begin{align*}
  \momentletter_{\indiceconserved} = \momentletter_{\indiceconserved} &- \spacestep\frac{1}{\relaxparletter_{\indiceconserved}} \Biggl ( \frac{1}{\latticevelocity}\partial_t \momentletter_{\indiceconserved}  + \frac{1}{\latticevelocity}\gammaDubois{1}{\indiceconserved} \Biggl ) + \spacestep^2\frac{1}{\relaxparletter_{\indiceconserved}} \Biggl ( \frac{1}{\latticevelocity^2}\Biggl (\frac{1}{\relaxparletter_{\indiceconserved}} - \frac{1}{2} \Biggr )\partial_{tt} \momentletter_{\indiceconserved}  + \frac{1}{\latticevelocity}\sum_{\indicescolumns = 1}^{\velocitynumber} \Biggl ( \frac{1}{\relaxparletter_{\indiceconserved}} + \frac{1}{\relaxparletter_{\indicescolumns}} - 1 \Biggr )\duboisoperatormatrixentry_{\indiceconserved \indicescolumns} \partial_t \momentletter_{\indicescolumns}^{\atequilibrium} + \sum_{\indicescolumns = 1}^{\velocitynumber} \sum_{\ell = 1}^{\velocitynumber}\Biggl (\frac{1}{\relaxparletter_{\ell}} - \frac{1}{2} \Biggr ) \duboisoperatormatrixentry_{\indiceconserved \ell} \duboisoperatormatrixentry_{\ell\indicescolumns}  \momentletter_{\indicescolumns}^{\atequilibrium}\Biggr )   \Biggr ) \\
  &+ \bigO{\spacestep^3}.
\end{align*}
Using relations analogous to \eqref{eq:DerivationPreviousOrder}, \eqref{eq:DerivativeTimeEquilibria} and \eqref{eq:DerivationPreviousOrderTwice} for $\consmomentsnumber \geq 1$, formally obtained by differentiating the result at the previous order , we finally obtain, after tedious but elementary computations
\begin{align*}
  \momentletter_{\indiceconserved} = \momentletter_{\indiceconserved} -  \frac{\spacestep}{\latticevelocity \relaxparletter_{\indiceconserved}} \Biggl ( \partial_t \momentletter_{\indiceconserved}  + \gammaDubois{1}{\indiceconserved} -  \latticevelocity\spacestep \sum_{\indicescolumns = \consmomentsnumber + 1}^{\velocitynumber} \Biggl (\frac{1}{\relaxparletter_{\indicescolumns}} - \frac{1}{2} \Biggr ) \duboisoperatormatrixentry_{\indiceconserved \indicescolumns} \Biggl ( \sum_{\ell = 1}^{\consmomentsnumber}\duboisoperatormatrixentry_{\indicescolumns \ell} \momentletter_{\ell} + \sum_{\ell = \consmomentsnumber + 1}^{\velocitynumber} \duboisoperatormatrixentry_{\indicescolumns \ell} \momentletter_{\ell}^{\atequilibrium} - \frac{1}{\latticevelocity} \sum_{\ell = 1}^{\consmomentsnumber} \straightderivative{\momentletter_{\indicescolumns}^{\atequilibrium}}{\momentletter_{\ell}} \gammaDubois{1}{\ell} \Biggr )   \Biggr )+ \bigO{\spacestep^3},
\end{align*}
which coincides with the result from \Cref{thm:AcousticScaling}.
Therefore, up to order two, our approach yields results consistent with those from the procedure by \cite{yong2016theory}.

To demonstrate that we recover the same result at any order for any scaling between time and space discretizations, let us assume $\consmomentsnumber = 1$.
Then we have, using that $\relaxationmatrix \in \lineargroup{\velocitynumber}{\reals}$, $\asymptoticstreammoments \conj{\asymptoticstreammoments} = \conj{\asymptoticstreammoments}\asymptoticstreammoments = \matricial{\identity}$, the rule for the inverse of a product of matrix and the identity relative to geometric series in the context of formal power series, that
\begin{align*}
  \matricial{0} &= \determinant (\asymptotictimeshiftoperator \matricial{\identity} - \asymptoticschememoments) \vectorial{\momentletter} - \adjugate (\asymptotictimeshiftoperator \matricial{\identity} - \asymptoticschememoments) \asymptoticschemeequil \vectorial{\momentletter}^{\atequilibrium} = \determinant (\asymptotictimeshiftoperator \matricial{\identity} - \asymptoticschememoments) \left ( \vectorial{\momentletter} -  (\asymptotictimeshiftoperator \matricial{\identity} - \asymptoticstreammoments (\matricial{\identity} - \relaxationmatrix))^{-1} \asymptoticstreammoments \relaxationmatrix \vectorial{\momentletter}^{\atequilibrium}\right ), \\
  &=\determinant (\asymptotictimeshiftoperator \matricial{\identity} - \asymptoticschememoments) \left ( \vectorial{\momentletter} -  (\relaxationmatrix^{-1}\conj{\asymptoticstreammoments}(\asymptotictimeshiftoperator \matricial{\identity} - \asymptoticstreammoments (\matricial{\identity} - \relaxationmatrix)))^{-1} \vectorial{\momentletter}^{\atequilibrium}\right ) =\determinant (\asymptotictimeshiftoperator \matricial{\identity} - \asymptoticschememoments) \left ( \vectorial{\momentletter} - (\matricial{\identity} + \relaxationmatrix^{-1}(\asymptotictimeshiftoperator \conj{\asymptoticstreammoments} - \matricial{\identity}))^{-1} \vectorial{\momentletter}^{\atequilibrium}\right ), \\
  &=\determinant (\asymptotictimeshiftoperator \matricial{\identity} - \asymptoticschememoments) \Biggl ( \vectorial{\momentletter} - \Biggl (\sum_{\perturbationorderindices = 0}^{+\infty} (-\relaxationmatrix^{-1}(\asymptotictimeshiftoperator \conj{\asymptoticstreammoments} - \matricial{\identity}))^{\perturbationorderindices} \Biggr ) \vectorial{\momentletter}^{\atequilibrium}\Biggr ) =\determinant (\asymptotictimeshiftoperator \matricial{\identity} - \asymptoticschememoments) \Biggl ( \vectorial{\momentletter} - \lim_{\indicesmaxwelliteration \to +\infty} \ordermaxwelliteration{\vectorial{\momentletter}}{\indicesmaxwelliteration} \Biggr ).
\end{align*}
Therefore the expansion of the \fd scheme from \Cref{prop:ReductionFiniteDifferenceGeneralOld} bis and the non-truncated Maxwell iteration method on the \lbm scheme coincide up to a multiplication by a formal power series of time-space differential operators, \emph{i.e.} $\determinant (\asymptotictimeshiftoperator \matricial{\identity} - \asymptoticschememoments) \in \taylorseries$.
\emph{A priori}, the resulting modified equations are not the same, but since $\determinant (\asymptotictimeshiftoperator \matricial{\identity} - \asymptoticschememoments) = \determinant(\relaxationmatrix) + \bigO{\spacestep} = \relaxparletter_{1} \productrelaxation + \bigO{\spacestep}$, thus we ``pay'' only a constant factor we can divide by at dominant order, the modified equations at leading order are the same. Then, at each order, the result must be the same because we re-inject, in a recursive fashion, the solution truncated at the previous order to eliminate the higher-order time derivatives, see for instance \eqref{eq:DerivativeTimeEquilibria} and \eqref{eq:DerivationPreviousOrderTwice}.
The fact that the modified equations recovered by the Maxwell iteration are the same than the ones from the corresponding \fd scheme at any order provides an \emph{a posteriori} justification of the Maxwell iteration. We also emphasize that using the Maxwell iteration to compute these equations is generally less involved in terms of computations than doing the same on the corresponding \fd schemes.

\section{Conclusions and perspectives}\label{sec:Conclusions}

In this paper, we have rigorously derived the target PDEs for any \lbm scheme under acoustic and diffusive scalings by restating it as a multi-step macroscopic \fd scheme on the conserved moments \cite{bellotti2021fd}.
Moreover, the modified equations -- which the schemes are ``more consistent''  with -- have been found up to second order.
These findings allow to utilize -- upon studying the stability \cite{bellotti2021fd} of the \lbm scheme at hand -- the Lax equivalence theorem \cite{lax1956survey} to conclude on its convergence and order of convergence towards the solution of the target PDEs.
Since the passage from the kinetic to the macroscopic standpoint is fully discrete, our analysis can handle any type of time-space scaling and be pushed forward to reach higher orders in the discretization parameters.
Contrarily to the existing techniques, the quasi-equilibrium of the non-conserved moments in the limit of small discretization parameters or the introduction of several time scales in the problem are not the keys to eliminate the non-conserved variables from the macroscopic equations.
The obtained results confirm, going beyond empirical evidence, that the formal Taylor expansion by \cite{dubois2008equivalent,dubois2019nonlinear} and the Maxwell iteration by \cite{yong2016theory} are well-grounded from the perspective of numerical analysts and traditional numerical methods for PDEs, such as Finite Difference.
In particular, we have extended the Maxwell iteration \cite{yong2016theory} to any \lbm scheme and shown that the modified equations found by this procedure are the same than the ones from the corresponding \fd schemes, at any order.
The general results that we have presented allow to immediately recover the modified equations without need for computing the corresponding \fd schemes, which would be time consuming. This allows -- for example -- to easily consider families of schemes depending on some parameters and investigate the dependence of the modified equations on these factors.

An improvement of the present work could be the establishment of the equivalence between different analyses \cite{chen1998lattice, qian2000higher, lallemand2000theory, dubois2008equivalent,dubois2019nonlinear, junk2003rigorous, junk2005asymptotic, junk2009convergence} for higher orders and ideally for any order.
Even if more involved from the standpoint of computations, the extension can be easily done by considering derivatives of higher order for the determinant and adjugate functions, in the spirit of \Cref{lemma:DerivativesDeterminant} and \Cref{lemma:DerivativesAjugate}. 
In this work, all the computations have been done by hand but one could envision to seek some help from symbolic computations.
This is a current path of investigation which final aim is to provide the computation -- inside the package \texttt{pylbm}\footnote{\url{https://pylbm.readthedocs.io}}  -- of the modified equations of any \lbm scheme either by the corresponding \fd scheme or using the Maxwell iteration.

\section*{Acknowledgments}
The author deeply thanks his PhD advisors, M. Massot and B. Graille, for the fruitful discussions and advice on the subject, his brother P. Bellotti for the useful tips to improve the style of manuscript and S. Simonis for having read and commented the preprint of this paper.
The author also thanks the two anonymous referees for the valuable questions and suggestions.
The author is supported by a PhD funding (year 2019) from the Ecole polytechnique.

\bibliographystyle{acm}
\bibliography{biblio}

\begin{thebibliography}{10}

\bibitem{allaire2007numerical}
{\sc Allaire, G.}
\newblock {\em Numerical analysis and optimization: an introduction to
  mathematical modelling and numerical simulation}.
\newblock Oxford University Press, 2007.

\bibitem{bellotti2021high}
{\sc Bellotti, T., Gouarin, L., Graille, B., and Massot, M.}
\newblock {High accuracy analysis of adaptive multiresolution-based lattice
  Boltzmann schemes via the equivalent equations}.
\newblock {\em The SMAI Journal of Computational Mathematics 8\/} (2022),
  161--199.

\bibitem{bellotti2021fd}
{\sc Bellotti, T., Graille, B., and Massot, M.}
\newblock Finite difference formulation of any lattice {B}oltzmann scheme.
\newblock {\em {Numerische Mathematik} 152\/} (2022), 1--40.

\bibitem{boghosian2018curious}
{\sc Boghosian, B., Dubois, F., Graille, B., Lallemand, P., and Tekitek, M.-M.}
\newblock Curious convergence properties of lattice boltzmann schemes for
  diffusion with acoustic scaling.
\newblock {\em {Communications in Computational Physics} 23\/} (2018),
  1263--1278.

\bibitem{bouchut2000diffusive}
{\sc Bouchut, F., Guarguaglini, F.~R., and Natalini, R.}
\newblock Diffusive {BGK} approximations for nonlinear multidimensional
  parabolic equations.
\newblock {\em {Indiana University Mathematics Journal}\/} (2000), 723--749.

\bibitem{brewer1986linear}
{\sc Brewer, J.~W., Bunce, J.~W., and Van~Vleck, F.~S.}
\newblock {\em Linear systems over commutative rings}.
\newblock CRC Press, 1986.

\bibitem{caiazzo2009comparison}
{\sc Caiazzo, A., Junk, M., and Rheinl{\"a}nder, M.}
\newblock Comparison of analysis techniques for the lattice {B}oltzmann method.
\newblock {\em {Computers \& Mathematics with Applications} 58}, 5 (2009),
  883--897.

\bibitem{carpentier1997derivation}
{\sc Carpentier, R., de~La~Bourdonnaye, A., and Larrouturou, B.}
\newblock On the derivation of the modified equation for the analysis of linear
  numerical methods.
\newblock {\em ESAIM: Mathematical Modelling and Numerical Analysis 31}, 4
  (1997), 459--470.

\bibitem{chapman1990mathematical}
{\sc Chapman, S., and Cowling, T.~G.}
\newblock {\em The mathematical theory of non-uniform gases: an account of the
  kinetic theory of viscosity, thermal conduction and diffusion in gases}.
\newblock Cambridge university press, 1990.

\bibitem{chen1998lattice}
{\sc Chen, S., and Doolen, G.~D.}
\newblock Lattice {B}oltzmann method for fluid flows.
\newblock {\em {Annual Review of Fluid Mechanics} 30}, 1 (1998), 329--364.

\bibitem{cheng2003partial}
{\sc Cheng, S.~S.}
\newblock {\em Partial difference equations}, vol.~3.
\newblock CRC Press, 2003.

\bibitem{dellacherie2014construction}
{\sc Dellacherie, S.}
\newblock Construction and analysis of lattice {B}oltzmann methods applied to a
  {1D} convection-diffusion equation.
\newblock {\em {Acta Applicandae Mathematicae} 131}, 1 (2014), 69--140.

\bibitem{dhumieres1992}
{\sc D'Humi\`eres, D.}
\newblock {\em Generalized {L}attice-{B}oltzmann {E}quations}.
\newblock American Institute of Aeronautics and Astronautics, Inc., 1992,
  pp.~450--458.

\bibitem{ding2007eigenvalues}
{\sc Ding, J., and Zhou, A.}
\newblock Eigenvalues of rank-one updated matrices with some applications.
\newblock {\em {Applied Mathematics Letters} 20}, 12 (2007), 1223--1226.

\bibitem{dubois2008equivalent}
{\sc Dubois, F.}
\newblock Equivalent partial differential equations of a lattice {B}oltzmann
  scheme.
\newblock {\em {Computers \& Mathematics with Applications} 55}, 7 (2008),
  1441--1449.

\bibitem{dubois2019hal}
{\sc Dubois, F.}
\newblock {General third order {C}hapman-{E}nskog expansion of lattice
  {B}oltzmann schemes}.
\newblock In {\em {16th International Conference for Mesoscopic Methods in
  Engineering and Science, Edinburgh, 22--26 July 2019.}\/} (Edimburgh, United
  Kingdom, July 2019).

\bibitem{dubois2019nonlinear}
{\sc Dubois, F.}
\newblock Nonlinear fourth order {T}aylor expansion of lattice {B}oltzmann
  schemes.
\newblock {\em {Asymptotic Analysis}}, 1 Jan. 2021 (2021), 1--41.

\bibitem{dubois2020notion}
{\sc Dubois, F., Graille, B., and Rao, S.~R.}
\newblock A notion of non-negativity preserving relaxation for a
  mono-dimensional three velocities scheme with relative velocity.
\newblock {\em {Journal of Computational Science} 47\/} (2020), 101181.

\bibitem{dubois2009towards}
{\sc Dubois, F., and Lallemand, P.}
\newblock Towards higher order lattice {B}oltzmann schemes.
\newblock {\em {Journal of Statistical Mechanics: Theory and Experiment} 2009},
  06 (2009), P06006.

\bibitem{dubois2011quartic}
{\sc Dubois, F., and Lallemand, P.}
\newblock Quartic parameters for acoustic applications of lattice {B}oltzmann
  scheme.
\newblock {\em {Computers \& Mathematics with Applications} 61}, 12 (2011),
  3404--3416.

\bibitem{fevrier2014extension}
{\sc F{\'e}vrier, T.}
\newblock {\em Extension et analyse des sch{\'e}mas de Boltzmann sur
  r{\'e}seau: les sch{\'e}mas {\`a} vitesse relative}.
\newblock PhD thesis, Universit{\'e} Paris Sud-Paris XI, 2014.

\bibitem{fuvcik2021equivalent}
{\sc Fu{\v{c}}{\'\i}k, R., and Straka, R.}
\newblock Equivalent finite difference and partial differential equations for
  the lattice {B}oltzmann method.
\newblock {\em {Computers \& Mathematics with Applications} 90\/} (2021),
  96--103.

\bibitem{guo2013lattice}
{\sc Guo, Z., and Shu, C.}
\newblock {\em Lattice {B}oltzmann method and its application in engineering},
  vol.~3.
\newblock World Scientific, 2013.

\bibitem{gustafsson1995time}
{\sc Gustafsson, B., Kreiss, H.-O., and Oliger, J.}
\newblock {\em Time dependent problems and difference methods}, vol.~24.
\newblock John Wiley \& Sons, 1995.

\bibitem{henon1987viscosity}
{\sc H{\'e}non, M.}
\newblock Viscosity of a lattice gas.
\newblock {\em {Lattice Gas Methods for Partial Differential Equations}\/}
  (1987), 179--207.

\bibitem{higuera1989boltzmann}
{\sc Higuera, F.~J., and Jim{\'e}nez, J.}
\newblock {B}oltzmann approach to lattice gas simulations.
\newblock {\em {EPL (Europhysics Letters)} 9}, 7 (1989), 663.

\bibitem{horn2012matrix}
{\sc Horn, R.~A., and Johnson, C.~R.}
\newblock {\em Matrix analysis}.
\newblock Cambridge university press, 2012.

\bibitem{huang1987}
{\sc Huang, K.}
\newblock {\em Statistical {M}echanics}, 2~ed.
\newblock John Wiley \& Sons, 1987.

\bibitem{johnson2002curious}
{\sc Johnson, W.~P.}
\newblock The curious history of {F}a{\`a} di {B}runo's formula.
\newblock {\em {The American Mathematical Monthly} 109}, 3 (2002), 217--234.

\bibitem{junk2005asymptotic}
{\sc Junk, M., Klar, A., and Luo, L.-S.}
\newblock Asymptotic analysis of the lattice {B}oltzmann equation.
\newblock {\em {Journal of Computational Physics} 210}, 2 (2005), 676--704.

\bibitem{junk2009convergence}
{\sc Junk, M., and Yang, Z.}
\newblock Convergence of lattice {B}oltzmann methods for {N}avier--{S}tokes
  flows in periodic and bounded domains.
\newblock {\em {Numerische Mathematik} 112}, 1 (2009), 65--87.

\bibitem{junk2003rigorous}
{\sc Junk, M., and Yong, W.-A.}
\newblock Rigorous {N}avier--{S}tokes limit of the lattice {B}oltzmann
  equation.
\newblock {\em {Asymptotic Analysis} 35}, 2 (2003), 165--185.

\bibitem{jury1964theory}
{\sc Jury, E.~I.}
\newblock {\em Theory and {A}pplication of the z-{T}ransform {M}ethod}.
\newblock Krieger Publishing Co., 1964.

\bibitem{kassel95quantum}
{\sc Kassel, C.}
\newblock {\em Quantum {G}roups}, 1~ed.
\newblock Graduate Texts in Mathematics. Springer-Verlag New York, 1995.

\bibitem{kruger2017lattice}
{\sc Kr{\"u}ger, T., Kusumaatmaja, H., Kuzmin, A., Shardt, O., Silva, G., and
  Viggen, E.~M.}
\newblock The lattice {B}oltzmann method.
\newblock {\em {Springer International Publishing} 10}, 978-3 (2017).

\bibitem{lallemand2000theory}
{\sc Lallemand, P., and Luo, L.-S.}
\newblock Theory of the lattice {B}oltzmann method: {D}ispersion, dissipation,
  isotropy, {G}alilean invariance, and stability.
\newblock {\em {Physical Review E} 61}, 6 (2000), 6546.

\bibitem{lang02algebra}
{\sc Lang, S.}
\newblock {\em Algebra}, 3~ed.
\newblock Graduate Texts in Mathematics. Springer-Verlag New York, 2002.

\bibitem{lax1956survey}
{\sc Lax, P.~D., and Richtmyer, R.~D.}
\newblock Survey of the stability of linear finite difference equations.
\newblock {\em Communications on pure and applied mathematics 9}, 2 (1956),
  267--293.

\bibitem{mcnamara1988use}
{\sc McNamara, G.~R., and Zanetti, G.}
\newblock Use of the {B}oltzmann equation to simulate lattice-gas automata.
\newblock {\em {Physical Review Letters} 61}, 20 (1988), 2332.

\bibitem{miller1960finitedifference}
{\sc Miller, K.~S.}
\newblock {\em An Introduction to the {C}alculus of {F}inite {D}ifferences and
  {D}ifference {E}quations}.
\newblock Dover Publications, 1960.

\bibitem{monforte2013formal}
{\sc Monforte, A.~A., and Kauers, M.}
\newblock Formal {L}aurent series in several variables.
\newblock {\em {Expositiones Mathematicae} 31}, 4 (2013), 350--367.

\bibitem{niven1969formal}
{\sc Niven, I.}
\newblock Formal power series.
\newblock {\em {The American Mathematical Monthly} 76}, 8 (1969), 871--889.

\bibitem{qian2000higher}
{\sc Qian, Y.-H., and Zhou, Y.}
\newblock Higher-order dynamics in lattice-based models using the
  {C}hapman-{E}nskog method.
\newblock {\em {Physical Review E} 61}, 2 (2000), 2103.

\bibitem{rheinlander2007analysis}
{\sc Rheinl{\"a}nder, M.~K.}
\newblock {\em Analysis of lattice-{B}oltzmann methods: asymptotic and numeric
  investigation of a singularly perturbed system}.
\newblock PhD thesis, 2007.

\bibitem{roman2005umbral}
{\sc Roman, S.}
\newblock {\em The {U}mbral {C}alculus}.
\newblock Dover Publications, 2005.

\bibitem{simonis2020relaxation}
{\sc Simonis, S., Frank, M., and Krause, M.~J.}
\newblock On relaxation systems and their relation to discrete velocity
  {B}oltzmann models for scalar advection--diffusion equations.
\newblock {\em Philosophical Transactions of the Royal Society A 378}, 2175
  (2020), 20190400.

\bibitem{stewart1998adjugate}
{\sc Stewart, G.}
\newblock On the adjugate matrix.
\newblock {\em {Linear Algebra and its Applications} 283}, 1-3 (1998),
  151--164.

\bibitem{strikwerda2004finite}
{\sc Strikwerda, J.~C.}
\newblock {\em Finite difference schemes and partial differential equations}.
\newblock SIAM, 2004.

\bibitem{succi2001lattice}
{\sc Succi, S.}
\newblock {\em The lattice {B}oltzmann equation: for fluid dynamics and
  beyond}.
\newblock Oxford University Press, 2001.

\bibitem{suga2010accurate}
{\sc Suga, S.}
\newblock An accurate multi-level finite difference scheme for {1D} diffusion
  equations derived from the lattice {B}oltzmann method.
\newblock {\em {Journal of Statistical Physics} 140}, 3 (2010), 494--503.

\bibitem{van2009smooth}
{\sc Van~Leemput, P., Rheinl{\"a}nder, M., and Junk, M.}
\newblock Smooth initialization of lattice {B}oltzmann schemes.
\newblock {\em {Computers \& Mathematics with Applications} 58}, 5 (2009),
  867--882.

\bibitem{warming1974modified}
{\sc Warming, R.~F., and Hyett, B.}
\newblock The modified equation approach to the stability and accuracy analysis
  of finite-difference methods.
\newblock {\em {Journal of Computational Physics} 14}, 2 (1974), 159--179.

\bibitem{yong2016theory}
{\sc Yong, W.-A., Zhao, W., Luo, L.-S., et~al.}
\newblock Theory of the lattice {B}oltzmann method: Derivation of macroscopic
  equations via the {M}axwell iteration.
\newblock {\em {Physical Review E} 93}, 3 (2016), 033310.

\bibitem{zhang2019lattice}
{\sc Zhang, M., Zhao, W., and Lin, P.}
\newblock {Lattice Boltzmann method for general convection-diffusion equations:
  MRT model and boundary schemes}.
\newblock {\em {Journal of Computational Physics} 389\/} (2019), 147--163.

\bibitem{zhao2017maxwell}
{\sc Zhao, W., and Yong, W.-A.}
\newblock Maxwell iteration for the lattice {B}oltzmann method with diffusive
  scaling.
\newblock {\em {Physical Review E} 95}, 3 (2017), 033311.

\bibitem{zwillinger2018crc}
{\sc Zwillinger, D.}
\newblock {\em {CRC} standard mathematical tables and formulas}.
\newblock Chapman and Hall - CRC, 2018.

\end{thebibliography}

\end{document}